\documentclass[11pt, english, reqno]{amsart}

\usepackage{amsmath, amsfonts, amssymb, amscd, enumerate, color, graphics, mathtools, wasysym, empheq, stmaryrd, slashed, babel, mathtools, tikz, hyperref}
\usetikzlibrary{cd,matrix,arrows,decorations.markings,positioning} 
\usepackage[normalem]{ulem}		
\usepackage{mathrsfs,bbm}
\usepackage{here}
\usepackage{pdflscape}
\usepackage[margin=3cm]{geometry}
\usepackage[graphicx]{realboxes}
\usepackage[figuresright]{rotating}
\usepackage{longtable}
\theoremstyle{plain}
\newtheorem{theo}{Theorem}
\theoremstyle{definition}

\newtheorem{definition}[theo]{Definition}
\theoremstyle{plain}
\newtheorem{lemma}[theo]{Lemma}
\newtheorem{theorem}[theo]{Theorem}
\newtheorem{corollary}[theo]{Corollary}
\newtheorem{proposition}[theo]{Proposition}

\theoremstyle{definition}

\newtheorem{remark}[theo]{Remark}

\newenvironment{colored}{\color{red}}{}
\newcommand{\bc}{\begin{colored}}
\newcommand{\ec}{\end{colored}}
\newcommand{\opp}[1]{\operatorname{#1}}

\renewcommand{\inf}{\mathfrak{inf}}

\newcommand{\beq}{\begin{equation}}
\newcommand{\eeq}{\end{equation}}
\renewcommand{\a}{\alpha}
\renewcommand{\b}{\beta}

\newcommand{\g}{\gamma}

\newcommand{\x}{\xi}

\renewcommand{\O}{\Omega}

\newcommand{\bC}{\mathbb{C}}

\newcommand{\bR}{\mathbb{R}}
\newcommand\R{{\mathbb R}}
\newcommand\C{{\mathbb C}}
\newcommand{\bZ}{\mathbb{Z}}

\newcommand{\bP}{\mathbb{P}}

\newcommand{\gb}{\mathfrak{b}}
\newcommand{\gc}{\mathfrak{c}}

\newcommand{\gf}{\mathfrak{f}}
\renewcommand{\gg}{\mathfrak{g}}

\newcommand{\gh}{\mathfrak{h}}

\newcommand{\gm}{\mathfrak{m}}

\newcommand{\gp}{\mathfrak{p}}

\newcommand{\gs}{\mathfrak{s}}

\renewcommand\sl{\mathfrak{sl}}

\newcommand{\cD}{\mathcal{D}}
\newcommand{\cE}{\mathcal{E}}

\newcommand{\cI}{\mathcal{I}}
\newcommand{\cJ}{\mathcal{J}}
\newcommand{\cK}{\mathcal{K}}
\newcommand{\cL}{\mathcal{L}}
\newcommand{\clL}{\mathscr{L}}
\newcommand{\clB}{\mathscr{B}}
\newcommand{\cM}{\mathcal{M}}

\newcommand{\cO}{\mathcal{O}}
\newcommand{\cP}{\mathcal{P}}
\newcommand{\cQ}{\mathcal{Q}}

\newcommand{\cU}{\mathcal{U}}
\newcommand{\cV}{\mathcal{V}}

\newcommand{\p}{\partial}

\renewcommand{\square}{\kern1pt\vbox
{\hrule height 0.6pt\hbox{\vrule width 0.6pt\hskip 3pt
\vbox{\vskip 6pt}\hskip 3pt\vrule width 0.6pt}\hrule height0.6pt}\kern1pt}

\DeclareMathOperator\Id{Id}

\DeclareFontFamily{U}{mathx}{}
\DeclareFontShape{U}{mathx}{m}{n}{<-> mathx10}{}
\DeclareSymbolFont{mathx}{U}{mathx}{m}{n}
\DeclareMathAccent{\widehat}{0}{mathx}{"70}
\DeclareMathAccent{\widecheck}{0}{mathx}{"71}

\renewcommand\Re{\operatorname{\mathfrak{Re}}}

\newcommand{\wt}{\widetilde}

\newcommand{\ot}{\otimes}

\newcommand{\be}{\begin{equation}}
\newcommand{\ee}{\end{equation}}

\def\<#1,#2>{\langle\,#1,\,#2\,\rangle}
\newcommand{\arr}{\begin{array}{rlll}}
\newcommand{\ea}{\end{array}}
\newcommand{\bea}{\begin{eqnarray}}
\newcommand{\eea}{\end{eqnarray}}
\newcommand{\bean}{\begin{eqnarray*}}
\newcommand{\eean}{\end{eqnarray*}}
\newcommand{\rk}{\mathrm{rk}}

\newcommand{\vardbtilde}[1]{\widetilde{\raisebox{0pt}[0.85\height]{$\widetilde{#1}$}}}

\title[$3$-nondegenerate CR manifolds in dimension $7$ (II)]
{On $3$-nondegenerate CR manifolds in dimension $7$ (II):\\ the intransitive case}
\address{Boris Kruglikov, Department of Mathematics and Statistics, UiT The Arctic University of Norway, Troms\o\,  9037, Norway}
\email{boris.kruglikov@uit.no}
\address{Andrea Santi, Dipartimento di Matematica,
Universit\'a degli Studi di Roma ``Tor Vergata'',
Via della ricerca scientifica 1, 00133 Roma, ITALY}
\email{asanti.math@gmail.com, santi@mat.uniroma2.it}
\thanks{}
\author{Boris Kruglikov, Andrea Santi}

\begin{document}

\begin{abstract} 
We investigate $3$-nondegenerate CR structures in the lowest possible dimension $7$ and
show that $8$ is the maximal dimension for the Lie algebra of symmetries of such structures. The next possible symmetry dimension is $6$, and 
for the automorphism groups the dimension $7$ is also realizable.  
This part (II) is devoted to the case where the symmetry algebra acts intransitively.
We use various methods to bound its dimension and 
demonstrate the existence of infinitely many non-equivalent submaximally symmetric models.
Summarizing, we get a stronger form of Beloshapka's conjecture on the symmetry dimension of 
hypersurfaces in $\mathbb C^4$. 

\end{abstract}

\maketitle
\null \vspace*{-.50in}

\tableofcontents

\section{Introduction}
\setcounter{equation}{0}\setcounter{section}{1}

An {\it almost CR-structure} is a triple $(\cM,\cD,\cJ)$ consisting of a manifold, a vector distribution
$\cD\subset T\cM$, called the CR-distribution, and a field of complex structures
$\cJ\in\Gamma(\operatorname{End}(\cD))$. 
The complexified CR-distribution splits as the direct sum of its holomorphic and antiholomorphic parts, 
$\cD\otimes\mathbb C=\cD_{10}\oplus \cD_{01}$, where
 $$
\cD_{10}=\{X-i\cJ X\mid X\in \cD\},\;\;\cD_{01}=\{X+i\cJ X\mid X\in \cD\}=\overline{\cD_{10}}.
 $$
A CR-structure is an integrable almost CR-structure, i.e., an almost CR-structure for which the distribution $\cD_{10}$ is involutive. This paper, as its preceding part \cite{KS1}, studies CR-structures of hypersurface type:
CR-codimension $1$ and CR-dimension $n=\frac12(\dim \cM-1)=\operatorname{rank}_{\mathbb C}(\cD)$.

There are two intimately related problems that are important in CR geometry: equivalence and symmetry/automorphism
of CR structures.  For Levi-nodegenerate CR manifolds these problems were addressed and resolved in 
the classical works \cite{C,CM,Ta1,Ta2},  in particular it was shown that the dimension of the algebra of germs of 
any Levi-nondegenerate connected CR-hypersurface $\cM$ of CR-dimension $n$ 
does not exceed $n^2+4n+3$. 

In what follows, the manifold $\cM$ will be assumed connected and real-analytic. In this case the CR structure $(\cM,\cD,\cJ)$
can be locally realized as a real submanifold of codimension 1 in $\bC^{n+1}$; then $\cD=T\cM\cap \cJ_oT\cM$,
where $\cJ_o=i$ is the complex structure in $\bC^{n+1}$, and $\cJ=\cJ_o|_{\cD}$. The realization, however,
is not always possible globally, and in this paper we will not rely on a choice of realization.
All other involved objects, such as infinitesimal symmetries and global automorphisms, will be also assumed
real-analytic. As algebra of infinitesimal symmetries, we can either choose the algebra of 
germs of symmetries $\gg=\inf(\cM,\cD,\cJ;p)$ at a fixed point $p\in\cM$ or the algebra of 
global vector fields $\gg=\inf(\cM,\cD,\cJ)$. It is also possible to consider fields defined on a neighbourhood
$\cU\subset\cM$, our results hold true in each of these three cases.
In the arguments of our proofs, we will often tacitly assume $(\cM,\cD,\cJ)$ to be {\it regular}, i.e., the ranks of all involved bundles are constant.  
In fact, they are constant on an open dense subset of $\cM$ due to upper-semicontinuity, and we may restrict 
to it by analyticity.  

In the Levi-nondegenerate case, the {\it Levi form\/} of a CR-hypersurface 
\begin{equation}
\label{eq:usual-Levi-form}
\clL:\cD_{10}\otimes\cD_{01}\to(T\cM/\cD)\otimes\mathbb C,\quad
\clL(Z,\overline{Z'})=i[Z,\overline{Z'}]\!\!\!\mod\cD\otimes\bC\,,
\end{equation}
is a conformal $\cJ$-Hermitian form on the distribution $\cD$ of signature $(p,n-p)$,
and a choice of orientation of the normal tangent bundle $T\cM/\cD$ interchanges the signature with $(n-p,p)$.
The maximally symmetric Levi-nondegenerate CR-hypersurfaces are locally CR-equivalent 
to a hyperquadric, which can be written as the tube
 \begin{equation*}
{\mathcal Q}_{(p)}=\Sigma_{(p)}\times\R^{n+1}(y)\subset\bC^{n+1}(z),\quad
\Sigma_{(p)}=\bigl\{x_0=x_1^2+\dots+x_p^2-x_{p+1}^2-\dots-x_n^2\bigr\}\subset\R^{n+1}(x)\,,
 \end{equation*}
where $z=x+iy\in\mathbb C^{n+1}$.
In the absence of Levi-nondegeneracy,  finite-dimensionality of $\gg$ is guaranteed by 
{\it holomorphic nondegeneracy} \cite[\S12.5]{BER}, or equivalently,  by $k$-nondegeneracy 
for some $1\le k\le n$ \cite[\S11.1-11.3]{BER}. 
We will recall this notion as well as the Freeman filtration \cite{Fr} in \S\ref{sec:2.1}.
The 2-nondegenerate CR-structures were extensively studied recently, see \cite{Gr,PZ,SZ}. 

In this paper we focus on understanding the symmetries of $3$-nondegenerate 
CR-manifolds for the smallest possible value $n=3$. (Along the way, we will also deal with some $2$-degenerate CR-structures.)  In \cite{KS1} we treated 
the homogeneous and locally homogeneous CR-manifolds, now we concentrate on the complementary case.

 \begin{theorem}\label{T1}
If the symmetry algebra $\gg$ of a 7-dimensional 3-nondegenerate CR-hypersurface acts locally intransitively,  i.e., the orbits have dimension strictly less than $7$,  then $\dim\gg<7$ as well.
 \end{theorem}

One can expect from the formulation of Theorem \ref{T1} the construction of a canonical frame on the CR-manifold,
however our approach is different. In fact, we do not claim that the local action of $\gg$ has trivial stabilizer $\gg_p$ at $p\in\cM$, although this may be true, but apply a variety of methods depending on the dimension of
the generic orbits (as noted above we can assume, without loss of generality,
that the dimension of 
$\gg/\gg_p$ is constant for all $p\in\cM$). We employ the Tanaka-Weisfeiler filtration of $\gg$ at regular intransitive points, retaining information along directions transverse to the orbits. 
The complex structure $\cJ$ is never projectable to the leaf space of the Cauchy characteristic space of $\cD$, still $\gg$ projects faithfully as infinitesimal contact symmetries.
Therefore, it often happens that $\gg$ is effectively represented on a fixed orbit $\cO^G_p$ as Lie algebra of infinitesimal symmetries of a smaller dimensional geometric structure (e.g., a non-integrable CR structure). We then take advantage of the Tanaka-Weisfeiler filtration of $\gg|_{\cO^G_p}$ at the transitive point $p\in \cO^G_p$, which is more constrained than the intransitive filtration. In the most involved cases, such filtrations are combined with a careful study in Cartan's spirit of the structure equations of adapted frames   (on the CR-manifold and on the orbit), providing crucial restrictions to disprove existence of the relevant filtered deformations. 

Together with the main result of \cite{KS1} the above Theorem \ref{T1} finishes the proof of Theorem 1 of \cite{KS1}.
Namely we conclude that a 3-nondegenerate 7-dimensional real-analytic CR-hypersurface 
has sharp bound of the symmetry dimension $\dim\gg\leq8$. 

We recall that in the homogeneous case the only possible symmetry dimension is 8
and that the CR-structure with such symmetry dimension is locally unique, given in the tube form as follows \cite{KS1}:
 \begin{equation*}
{\mathcal R}^7=(TR\setminus R)\times\R^4(y)\subset\bC^4(z),\quad
R=\bigl\{x_0x_2=x_1^2,\ x_0x_3=x_1x_2,\ x_1x_3=x_2^2\bigr\}\subset\R^4(x)\,,
 \end{equation*} 
where $R$ is cone over the rational normal curve of degree $3$ and $TR$ its tangent variety. 
Theorem \ref{T1} implies that the symmetry dimension 7 is also non-realizable in the intransitive case.
Concerning the next realizable symmetry dimension $\dim\gg$ we have the following statement.

 \begin{theorem}\label{T2}
The submaximal symmetry dimension for 7-dimensional 3-nondegenerate CR-hypersurfaces is 6.
There is a continuum of pairwise non-equivalent such CR-hypersurfaces with 
the dimension of the symmetry algebra $\dim\gg=6$.
 \end{theorem}

For global symmetry the situation is different: the submaximal dimension is larger.
 \begin{theorem}\label{T3}
Let $G$ be the automorphism group of a 3-nondegenerate 7-dimensional real-analytic CR-hypersurface $(\cM,\cD,\cJ)$. Then
$\dim G\leq8$, and the bound corresponds to a finite or countable cover of the model $\mathcal{R}^7$. Otherwise:
 \begin{enumerate}
\item Either $\dim G=7$, which is achieved for a unique $(\cM,\cD,\cJ)$ that is locally but not globally isomorphic to 
$\mathcal{R}^7$. The automorphism group of this submaximal model is 
$G=B\ltimes S^3\R^2$, where $B$ is the Borel subgroup in $GL_2(\R)$;
\item Or $\dim G<7$, and the dimensions 6, 5, and 4 are achieved for infinitely many non-equivalent 
intransitive $(\cM,\cD,\cJ)$.
 \end{enumerate}
 \end{theorem}

In fact, we will show that the submaximal model is the tube over (the nonsingular part of)
the tangent variety of the cone over the punctured rational normal curve of degree $3$. 

Regarding the maximal and submaximal symmetry dimensions in CR geometry, without 
specification of nondegeneracy type, we have the following: 

 \begin{theorem}\label{T4}
For any real-analytic connected holomorphically nondegenerate CR-hyper\-surface $(\cM,\cD,\cJ)$ in $\C^4$ 
the symmetry dimension is bounded by $\dim\gg\le 24$. If the CR manifold is not everywhere spherical,
i.e., not locally CR-equivalent to the hyperquadric ${\mathcal Q}_{(p)}$, 
then $\dim\gg\le 17$.
 \end{theorem}

In particular, as noticed in \cite{KS1}, this asserts Beloshapka's conjecture for $n=3$
by methods independent of \cite{B2}. 
We recall that the original conjecture of Beloshapka (see \cite[p.~38]{B2}) 
for $n=1$ follows directly from \cite{C}, while for $n=2$ it is due to \cite{IZ,MS} (see also \cite{MP}). 
A generalized form of it was established in \cite{KS,IK1} for $n=1$ and in \cite{IK2,Kr} for $n=2$. 
Our Theorem \ref{T4} is a version of such generalization for $n=3$, the symmetry dimension bound in it is sharp
and it is attained only on inhomogeneous CR manifolds (see the proof for details).

Note that the classification in dimension $7$ of locally homogeneous CR-structures (or those with large symmetry)
is far from completion both for 1-nondegenerate and 2-nondegenerate CR-manifolds, cf. \cite{Lo}. 
Surprisingly, the 3-nondegenerate case in dimension 7 has been fully resolved in \cite{KS1} 
-- this has to be compared with the 2-nondegenerate case in dimension 5 \cite{FK2}.

\vskip0.15cm\par\noindent
{\it Structure of the paper.} 
In \S\ref{sec:2} we recall the notions of the Freeman filtration in dimension $7$, of the Tanaka--Weisfeiler filtration of the symmetry algebra (including the intransitive case), and establish first general results on the Freeman bundles and their adapted frames. We also describe the geometric structures induced on the local leaf spaces of the Freeman bundles. The following \S\ref{sec:3} settles the upper bound on the CR-symmetry dimension of 
an intransitive 3-nondegenerate CR structure in dimension 7, with careful geometric and Lie-theoretic analysis that depend on the dimension of generic orbits. It starts with general arguments in \S\ref{sec:3.1}-\S\ref{sec:3.2}, including an a priori characterization of the CR-symmetries when the Cauchy characteristic space has small codimension in the orbit distribution, and it continues with \S\ref{sec:3.3}-\S\ref{sec:3.6}, where Theorems \ref{thm:dim123}, \ref{thm:dim4}, \ref{thm:dim6}, and \ref{thm:dim5} are demonstrated. Theorem \ref{T1} is a consequence of these theorems.
We note that the results of \S\ref{sec:2} and \S\ref{sec:3.1}-\S\ref{sec:3.2} are general and might be of independent
interest. The paper ends with \S\ref{sec:4}, where infinitesimal and global submaximal CR-symmetry dimensions are established via certain geometric constructions on nondegenerate projective curves in $\R\mathbb{P}^3$ -- the proofs of Theorems \ref{T2} and \ref{T3} can be found in \S\ref{sec:5.1} and \S\ref{sec:5.2}, respectively. We also consider some generalizations in \S\ref{sec:5.3} and ends in \S\ref{sec:5.4} with a stronger version of Beloshapka's conjecture in dimension $7$ -- the proof of Theorem \ref{T4} can be found here.

\vskip0.15cm\par\noindent
{\it Notations.} 
For a real vector space $V$ we set $V_\times=\left\{v\in V\,|\,v\neq0\right\}$. We decompose any section $X$ of $\cD\otimes\bC$ into the sum $X=X_{10}+X_{01}$ of its holomorphic and antiholomorphic components, which are sections of $\cD_{10}$ and $\cD_{01}$.
 
\vskip0.15cm\par\noindent
{\it Acknowledgments.} 
The research leading to these results has received funding from the Norwegian Financial Mechanism 2014-2021 (project registration number 2019/34/H/ST1/00636) and the Troms\o{} Research Foundation (project ``Pure Mathematics in Norway''). The second author acknowledges the MIUR Excellence Department Project MatMod@TOV
awarded to the Department of Mathematics, University of Rome Tor Vergata, CUP E83C23000330006. 
This article/publication was also supported by the ``National Group for
Algebraic and Geometric Structures, and their Applications'' GNSAGA-INdAM (Italy) and it is
based upon work from COST Action CaLISTA CA21109 supported by COST
(European Cooperation in Science and Technology). 
\vskip0.1cm\par\noindent

\section{Some background on degenerate CR-structures}\hfill\par\label{sec:2}
\setcounter{section}{2}\setcounter{equation}{0}

\subsection{The Freeman filtration in dimension 7}\label{sec:2.1}

Let $(\cM, \cD, \cJ)$ be a $2n+1$-dimensional CR manifold of hypersurface type, 
with Lie algebra $\gg$ of infinitesimal symmetries. We will work in the real-analytic category, since
the focus of this part (II) is on real-analytic CR manifolds.
As explained in the Introduction and at the beginning of \S\ref{sec:3}, we may assume that $(\cM, \cD, \cJ)$ is {\it regular}, i.e., the sections of any sheaf of (complex or real) vector fields over $\cM$ are sections of an associated bundle. (In fact, the constant rank assumption can always be enforced upon localization to an open dense subset of $\cM$.) 
The notion of $k$-nondegeneracy can then be given in terms of Freeman filtrands  \cite{Fr}, 
and we refer the reader to \S 2.1 of part (I) for full details.

For $7$-dimensional CR-hypersurfaces $(\cM,\cD,\cJ)$, the Freeman filtrands are a sequence of complex bundles  
$\cD_{10}\supset \cK_{10}\supset \cL_{10}$, with $\cK_{10}$ the holomorphic part 
of the Cauchy characteristic space $\cK$ of $\cD$. The filtrand $\cK_{10}$ can also be defined as the left kernel of the Levi form $\clL_1=-i\clL$ 
as in \eqref{eq:usual-Levi-form}.  We let $\cD_{01}=\overline{\cD_{10}}$ and $\cK_{01}=\overline{\cK_{10}}$ be the anti-holomorphic parts of $\cD$ and $\cK$, and note
that $\cK$ is $\cJ$-stable but it does not depend on the complex structure $\cJ$.

In a similar fashion, we may consider the second-order Levi form
\begin{equation}
\label{eq:h.o.l.f.2}
\begin{aligned}
\clL_{2}:&\;\; \cK_{10}\otimes\cD_{01}\longrightarrow ({\cD}_{10}\oplus\cD_{01})/ (\cK_{10}\oplus\cD_{01})\cong \cD_{10}/\cK_{10}\\
&\;(X,Y)\longrightarrow [X,Y]\!\!\!\mod\cK_{10}\oplus\cD_{01}\,,
\end{aligned}
\end{equation}
whose left kernel is the filtrand $\cL_{10}$. We set $\cL_{01}=\overline{\cL_{10}}$, 
$\cL=\Re(\cL_{10}\oplus\cL_{01})$, and note that, contrary to the Cauchy characteristic space, the distribution $\cL$ depends on $\cD$ but also on $\cJ$. It is not difficult to see that both $\cK$ and $\cL$ are integrable distributions. Finally, the left-kernel of the
third-order Levi form
\begin{equation}
\label{eq:h.o.l.f.3}
\begin{aligned}
\clL_{3}:&\;\; \cL_{10}\otimes\cD_{01}\longrightarrow ({\cK}_{10}\oplus\cD_{01})/ (\cL_{10}\oplus\cD_{01})\cong \cK_{10}/\cL_{10}\\
&\;(X,Y)\longrightarrow [X,Y]\!\!\!\mod\cL_{10}\oplus\cD_{01}\,,
\end{aligned}
\end{equation}
either vanishes, in which case the CR-hypersurface is called holomorphically non-degenerate, or it is non-trivial, in which case the CR-hypersurface is holomorphically degenerate. By \cite{Fr} a holomorphically degenerate CR-hypersurface is locally isomorphic to the direct product 
of a lower-dimensional CR-manifold and a complex space.

\begin{definition}\cite{BER, Fr}  
A holomorphically non-degenerate $7$-dimensional CR-manifold $(\cM, \cD, \cJ)$ of hypersurface type is called:
\begin{itemize}
	\item[$(i)$] Levi-nondegenerate if $\cK_{10}=0$,
	\item[$(ii)$] $2$-nondegenerate if $\cK_{10}\neq 0$ and $\cL_{10}=0$,
	\item[$(iii)$] $3$-nondegenerate if $\cL_{10}\neq 0$.
\end{itemize}
\end{definition}

When $(\cM, \cD, \cJ)$ is $3$-nondegenerate, then by dimensional reasons
 $$
\rk_\mathbb C(\cD_{10})=3>\rk_\mathbb C(\cK_{10})=2>\rk_\mathbb C(\cL_{10})=1.
 $$ 
Here we collect some preliminary useful results for such hypersurfaces.

\begin{lemma}
\label{lem:preliminarybrackets}
The Lie brackets between the sections of the Freeman bundles of a $7$-dimensional CR-hypersurface $(\cM, \cD, \cJ)$ are collected in the following table: 
\par
 \bigskip
\centerline{\rotatebox{0}{\footnotesize
$\begin{array}{||c|c|c|c|c|c|c||}\hline
[-,-] &
\cD_{10}& \cK_{10}& \cL_{10}& \cD_{01}& \cK_{01}& \cL_{01} \\
\hline\hline
\cD_{10} & \cD_{10} & \cD_{10}& \cD_{10} &
T\cM\otimes\mathbb C  & \cD\otimes\mathbb C & \cD_{10}\oplus\cK_{01} \\
\hline
\cK_{10} & \star & \cK_{10} & \cK_{10} & \star & \cK\otimes\mathbb C & \cK_{10}\oplus\cL_{01} \\
\hline
\cL_{10} & \star & \star  & \cL_{10} & \star & \star  & \cL\otimes\mathbb C \\
\hline
\end{array}$
}}
\bigskip
\centerline{\small\it Lie brackets between vector fields of the Freeman bundles.}
\bigskip
\par\noindent
where ``$\star$'' refers to brackets that can be directly obtained from the others by skew-symmetry and conjugation, and we omitted brackets with first antiholomorphic entry for the same reason.
\end{lemma}

\begin{proof}
The table follows readily from the definitions of the Freeman bundles, the integrability of the CR structure $\mathcal J$ and of the distributions $\cK$ and $\cL$, except for the bracket $[\cK_{10},\cL_{01}]$.
First note that $[\cK_{10},\cL_{01}]\subset[\cD_{10},\cL_{01}]\subset \cD_{10}\oplus \cK_{01}$. The identities
\begin{align*}
[\cD_{01},[\cK_{10},\cL_{01}]]&\subset [[\cD_{01},\cK_{10}],\cL_{01}]+[\cK_{10},[\cD_{01},\cL_{01}]]
\\& \subset 
[\cD\otimes\mathbb C,\cL_{01}]+[\cK_{10},\cD_{01}]\subset\cD\otimes\mathbb C\,,\\
[\cD_{10},[\cK_{10},\cL_{01}]]&\subset [[\cD_{10},\cK_{10}],\cL_{01}]+[\cK_{10},[\cD_{10},\cL_{01}]]\\
&\subset 
[\cD_{10},\cL_{01}]+[\cK_{10},\cD_{10}\oplus \cK_{01}]\subset \cD_{10}\oplus \cK_{01} \,,
\end{align*}
then say that $[\cK_{10},\cL_{01}]\subset \cK_{10}\oplus \cL_{01}$.
\end{proof}

\begin{lemma}
\label{lem:simple-but-useful}
Let $(\cM,\cD,\cJ)$ be a $3$-nondegenerate $7$-dimensional CR-hypersurface and $\cU\subset \cM$ a sufficiently small open subset. Then there exist local trivializing sections $X_{10}=\tfrac12\big(X-i\cJ X\big)$, $Y_{10}=\tfrac12\big(Y-i\cJ Y\big)$, $Z_{10}=\tfrac12\big(Z-i\cJ Z\big)$ of the complex line bundles $\cD_{10}/\cK_{10}$, $\cK_{10}/\cL_{10}$, $\cL_{10}$, satisfying the normalization conditions 
\begin{equation}
\label{eq:norm-conditions}
\clL_{2}(Y_{10},\overline{X_{10}})=X_{10}\,,\qquad\mathcal \clL_{3}(Z_{10},\overline{X_{10}})=Y_{10}\,.
\end{equation}
Furthermore:
\begin{itemize}
	\item[$(i)$] 
Such sections are uniquely defined up to transformations of the form
 $$
X_{10}\mapsto \lambda e^{i\varphi}X_{10},\ 
Y_{10}\mapsto e^{2i\varphi}Y_{10},\ 
Z_{10}\mapsto \lambda^{-1}e^{3i\varphi} Z_{10}
 $$ 
for arbitrary functions $\lambda:\cU\rightarrow \mathbb R_+$, $e^{i\varphi}:\cU\rightarrow S^1\subset\mathbb C$,
	\item[$(ii)$] 
If there is a canonical real line subbundle in one of the above complex line bundles, then the corresponding normalized section may be taken real, and $e^{i\varphi}$ is a sixth root of unity at each point. Hence, modulo a finite cover of $\cU$, $Y_{10}$ is canonical and $X_{10}\mapsto \lambda X_{10}$, $Z_{10}\mapsto \lambda^{-1}Z_{10}$ (in particular the real lines of $X_{10}$ and $Z_{10}$ are canonical at each point),
	\item[$(iii)$] 
If there is a canonical trivializing section in one of the complex line bundles $\cD_{10}/\cK_{10}$ and $\cL_{10}$, 
then that section may be taken as a normalized section and:
  \vskip0.2cm\par\noindent
	\begin{enumerate}
\item In the case $\cD_{10}/\cK_{10}$, then $Y_{10}$ and $Z_{10}$ are canonical too,
\item In the case $\cL_{10}$, then $X_{10}$ and $Y_{10}$ are canonical too, 
modulo a finite cover of $\cU$.
	\end{enumerate}
\end{itemize}
\end{lemma}

\begin{proof}
The fact that the normalization conditions \eqref{eq:norm-conditions} are well-defined 
follows directly from the definitions of the higher-order Levi forms and  Lemma \ref{lem:preliminarybrackets}.
The proof is then straightforward and we omit it.
\end{proof}

\subsection{On 2-nondegenerate structures}\label{sec:2.2}

In dimension 5, the tube over the future light cone 
 \begin{equation}\label{5D}
{\mathcal C}_5=\Bigl\{z=(z_0,z_1,z_2)\in\mathbb C^3:x=\Re(z)\text{ satisfies }
x_0^2=x_1^2+x_2^2,\;x_0>0\Bigr\}
 \end{equation}
is the most symmetric 2-nondegenerate CR-manifold. 
In fact, any maximally symmetric $2$-nondegenerate CR-manifold is locally isomorphic to the projective completion in 
$\mathbb C\mathbb P^4$ of \eqref{5D}, seen as the homogeneous manifold $SO^\circ(3,2)/H$ 
for an appropriate $5$-dimensional closed subgroup $H$ of $SO^\circ(3,2)$. 
Absolute parallelisms for 5-dimensional 2-nondegenerate CR-structures  have been constructed in \cite{IZ, MS} 
and the existence of an $\mathfrak{so}(3,2)$-valued Cartan connection addressed in \cite{Gr}.
The submaximal symmetry dimension is 5, which is attained only in the locally homogeneous case \cite{MP}.
All the submaximally symmetric homogeneous models are simply transitive solvmanifolds, they were classified in \cite{FK2}.

Such CR-manifolds will be relevant in \S\ref{subsec:3.6.1}-\S\ref{subsec:3.6.3}. In the latter, however, we will also deal with $5$-dimensional $2$-nondegenerate {\it almost} CR-structures that are 
integrable when restricted to the Cauchy characteristic distribution but not on the whole CR-distribution.

In dimension $7$ and higher, absolute parallelisms for $2$-nondegenerate CR-hypersurfaces have been found
via various modifications of Tanaka's geometric prolongation scheme \cite{PZ, SZ}. 
(This approach does not extend immediately to CR-manifolds satisfying the higher $k$-nondegeneracy conditions. 
Ultimately, this is due to the fact that describing the $k$-nondegeneracy of a CR-structure at the filtered level requires 
the assignment of nonnegative degrees to the vectors in the Cauchy characteristic space; 
however, from 3-nondegeneracy on, the $p^{th}$-Freeman filtrand may have non-trivial components 
of filtration degree $< p$, as explained in detail in \cite[\S 2-\S 3]{KS1}).

The moduli space of CR-symbols of $2$-nondegenerate $7$-dimensional CR-structures contains continuous parameters. Under the constancy assumption on the CR-symbols, the maximal symmetry dimension is given by $16$ \cite{PZ,SZ} and it is attained on the hypersurface
 \begin{equation}\label{7D2NonDeg}
{\mathcal D}_5=\Bigl\{z=(z_0,z_1,z_2,z_3)\in\mathbb C^4:
\mathfrak{Im}\bigl(z_0-z_1\overline z_2+z_1^2\overline z_3\bigr)=0\Bigr\}\,,
 \end{equation}
having symmetry algebra $\mathfrak{co}(3,1)\ltimes S^2_o(\mathbb R^{3,1})\cong\mathfrak{co}(3,1)\ltimes\mathbb R^9$. 
In the general case the bounds increases to $17$ \cite{B2}, but this bound does not seem to be sharp.

\subsection{The intransitive Tanaka-Weisfeiler filtration}\label{sec:2.3}

Let $(\cM,\cD,\cJ)$ be a $3$-nondegenerate $7$-dimensional CR-hypersurface. Since its Lie algebra $\gg$ of infinitesimal symmetries is finite-dimensional, we may consider the connected, simply connected Lie group $G$ with Lie algebra of right-invariant vector fields $\gg$ and effective local action on $\cM$ (a family $\{\Theta_p\}_{p\in \cM}$ of compatible local actions $\Theta_p:\cV_p \times \cU_p \rightarrow \cM$, with $\cU_p$ an open neighborhood of $p\in \cM$ and  $\cV_p$ an open neighborhood of the identity in $G$).
In this section we assume that the action is intransitive, i.e., all $G$-orbits $\cO^G_q$
have dimension $<7$. In this case, the Tanaka--Weisfeiler filtration $$\cdots\supset \gg(p)^{i}\supset \gg(p)^{i+1}\supset\cdots$$ of $\gg$ at $p\in\cM$ can be introduced following \cite{Kru2011} (see also \cite[\S 2.3.2]{KruThe2017}), as we now briefly recall.

Given any $p\in \cM$, we consider the usual metabelian Lie algebra $\gm(p)=\gm_{-2}(p)\oplus\gm_{-1}(p)$ (sometimes called ``symbol'' or ``Carnot algebra'' of the distribution $\cD$ at $p\in\cM$) with
$$\gm_{-1}(p):= \cD|_p\,,\qquad \gm_{-2}(p):=T_p\cM/\cD|_p\,,$$ and 
its maximal graded prolongation $\operatorname{pr}(\gm(p))=\bigoplus_{i\geq -2} \operatorname{pr}_i(\gm(p))$ in the sense of Tanaka. By construction 
$\operatorname{pr}_i(\gm(p))\hookrightarrow\operatorname{Hom}(\ot^{i+1} \gm_{-1}(p), \gm_{-1}(p))$
for all $i \geq 0$. 

On the other hand, we let $\operatorname{ev}_p : \gg \to T_p \cM$ be the evaluation map and define 
$\gg(p)^{-2}:=\gg$, $\gg(p)^{-1} :=\operatorname{ev}_p^{-1}(\cD|_p)$ and $\gg(p)^0 := \ker(\operatorname{ev}_p)$.  Inductively, given any $\xi\in\gg(p)^i$ with $i \geq 0$, there exists a map 
\begin{align*}
\Psi^{i+1}_\xi&:\ot^{i+1}\cD|_p\to T_p\cM\\
\Psi^{i+1}_\xi&(Y_1,\dots,Y_{i+1})=\bigl[\bigl[\dots\bigl[[\xi,Y_1],Y_2\bigr],\dots\bigr],Y_{i+1}\bigr]|_p\,,
\end{align*}
where each $Y_j\in \cD|_p$ has been tacitly extended to a section of $\cD$ denoted by the same symbol. It can be seen that $\operatorname{Im}(\Psi^{i+1}_\xi) \subset \cD|_p$.
We set $\gg(p)^{i+1} :=\{\xi\in\gg(p)^i \mid \Psi_\xi^{i+1}=0\}$
for all $i\geq 0$. 

Then $\gg$ is a filtered Lie algebra, and each graded component $\gg_i(p)=\gg(p)^{i}/\gg(p)^{i+1}$ injects into the prolongation $\operatorname{pr}_i(\gm(p))$ via $\xi\!\! \mod \gg(p)^{i+1} \mapsto \Psi^{i+1}_\xi$. In the case of reductions given by the presence of an additional geometric structure preserved by $\gg$, we will denote the reduced structure algebra by $\mathfrak f_0(p)\subset\operatorname{pr}_0(\gm(p))$, and similarly for higher reductions $\mathfrak f_i(p)$    
 with $i>0$.
\begin{definition}
A point $p\in \cM$ is {\it regular} (w.r.t.\ the associated Tanaka-Weisfeiler filtration) if there exists a neighbourhood $\cU\subset \cM$ of
$p$ such that $q\mapsto \dim(\gg_i (q))$ is a constant function of $q\in\cU$ for all $i\geq -2$. 
\end{definition}
Equivalently, $p$ is a regular point if the dimension of each filtrand $\gg(q)^{i}$ is a constant function of $q\in\cU$. By analyticity, the set of regular points is an open dense subset of $\cM$ \cite[Prop. 2]{Kru2014}. 
It is a general result for filtered structures at regular points that 
the graded Lie algebra $\opp{gr}(\gg)$ (associated to the filtered Lie algebra $\gg$) is contained in the prolongation $\operatorname{pr}(\gm(p)\rtimes\gg_0(p))$ 
of the non-positively graded Lie algebra $\gm(p)\rtimes\gg_0(p)$
\cite[Thm. 2]{Kru2014}.

It follows that $q\mapsto \gg|_q$ defines a distribution of constant rank 
$d:=\dim\gg-\dim\gg(q)^0$
on $\cU$. We denote this distribution by $\mathfrak S$ and, by fixing symmetries $\xi_1,\ldots,\xi_s$ which are a basis of $\mathfrak S$, we see that $\mathfrak S$ is involutive and $\gg$-stable.

\begin{lemma}
\label{lem:1}
Let $p\in\cM$ be a regular point. If $\gg\neq 0$ then $\gg_{-2}(q)\neq 0$ for all $q\in \cU$.
\end{lemma}

\begin{proof}
Assume by contradiction that $\gg_{-2}(q)=0$ for some $q\in \cU$. By regularity, $\gg_{-2}(q)=0$ for all $q\in\cU$ and any symmetry $\xi$ belongs to $\gg=\gg(q)^{-2}=\gg(q)^{-1}$, i.e., it is tangent to $\cD$ everywhere on $\cU$. Then $\xi=0$ on $\cU$ by $3$-nondegeneracy, $\xi=0$ everywhere by analyticity.
\end{proof}

We conclude this section with a useful generalization of the simple fact that infinitesimal CR-symmetries restrict to right-invariant vector fields on any simply transitive orbit, hence commute with left-invariant vector fields. To avoid unnecessary minus signs in our formulae, in the rest of \S\ref{sec:2.3} we denote by $\gg$ the Lie algebra of left-invariant vector fields on $G$, instead of the Lie algebra of infinitesimal CR-symmetries. 

\begin{lemma}\label{lem:from-left-to-right}
Let $\cE_1,\cE_2,\cE_3$ be $G$-stable
subbundles of $T\cM$ that are tangent to the orbits, i.e., they are subbundles of the distribution $\mathfrak S$. We identify 
$T_p\cO^G_p\cong\gg/\gg(p)^0$ and let $\mathfrak{e}_1,\mathfrak{e}_2,\mathfrak{e}_3$ be the corresponding 
$\gg(p)^0$-stable subspaces of $\gg$ that contain $\gg(p)^0$ and such that $\cE_j|_p\cong\mathfrak{e}_j/\gg(p)^0$. Then:
\vskip0.1cm\par\noindent
\begin{itemize}
\item[$(i)$] If the map 
$$\Gamma(\cM,\cE_1)\times \Gamma(\cM,\cE_2)
\rightarrow \Gamma(\cM,\mathfrak S)/\Gamma(\cM,\cE_3)$$ induced by the Lie bracket of vector fields 
defines a tensor $\clB:\cE_1\otimes \cE_2\to \mathfrak S/\cE_3$, then
$\clB|_p(Y_1,Y_2)\cong[\xi_1,\xi_2]\!\mod\mathfrak{e}_3$ as an element of $T_p\cO^G_p/\cE_3|_p\cong \gg/\mathfrak{e}_3$,
for all $Y_j\in \cE_j|_p$ and any  choice of their representatives $\xi_j\in\mathfrak{e}_j$, $j=1,2$;
\item[$(ii)$] Assume now there exists a canonical section $U_1$ of $\cE_1$ and choose any representative  $\xi_1\in\mathfrak{e}_1$ of $U_1|_p\in\cE_1|_p$. If the map 
$$\Gamma(\cM,\cE_2)
\rightarrow \Gamma(\cM,\mathfrak S)/\Gamma(\cM,\cE_3)$$ induced by the Lie bracket with $U_1$ 
defines a tensor $\clB:\cE_2\to \mathfrak S/\cE_3$, then
$\clB|_p(Y_2)\cong[\xi_1,\xi_2]\!\mod\mathfrak{e}_3$ as an element of $T_p\cO^G_p/\cE_3|_p\cong \gg/\mathfrak{e}_3$,
for all $Y_2\in \cE_2|_p$ and any choice of representative $\xi_2\in\mathfrak{e}_2$. The analogous statement holds for canonical sections of $\cE_2$;
\item[$(iii)$] If $U_j$ is a canonical section of $\cE_j$ then 
$
[U_1,U_2]|_p\cong[\xi_1,\xi_2]\mod\gg(p)^0$ as an element of $T_p\cO^G_p\cong \gg/\gg(p)^0
$, for any choice of representatives $\xi_j\in\mathfrak{e}_j$, $j=1,2$.
\end{itemize}
Analogous statements remain true upon restriction to any open subset $\cU\subset \cM$ and upon complexification of bundles and symmetries.
\end{lemma}

\begin{proof}
Claim $(i)$ is a slight reformulation of Main Lemma 1 of \cite[p. 899]{Fel}, which, after fixing a vector space complement $\widetilde\gm$ in $\gg$ to $\gg(p)^0$, is carried out using a local model of $(\cM,\cD,\cJ)$ and 
constructing special vector fields $\zeta$ on $G$ that are
right-invariant for the stabilizer subgroup, see \cite[p. 900]{Fel}. Hence, they are
$\pi$-projectable for the quotient map $\pi:G\to\cO^G_p$ and they project to special sections $Y=\pi_*(\zeta)$ of the above bundles $\cE$. Now
$$\clB|_p(Y_1|_p,Y_2|_p)=
[Y_1,Y_2]|_p\!\mod\cE_3|_p=\pi_*[\zeta_1,\zeta_2]|_e\!\mod\cE_3|_p$$
by tensoriality and $\pi$-projectability, and 
$[\zeta_1,\zeta_2]|_e\cong[\xi_1,\xi_2]$ by how the $\zeta$'s are constructed (by construction, they are ``left-invariant'' in a weaker sense and just at the identity).
Furthermore, the result does not depend on the choice of representatives since $\mathfrak{e}_1+\mathfrak{e}_2\subset \mathfrak{e}_3$ by tensoriality.

The fact that Main Lemma 1 \cite{Fel} is stated in the locally homogeneous case poses no difficulty, since our bundles and symmetries are tangent to the orbits and can just be restricted there. 

Finally, if $U_1$ is a canonical section of $\cE_1$ (i.e., it is $G$-invariant), then its restriction to the orbit $\cO^G_p$ is determined by the value $$U_1|_p\in T_p\cO^G_p\cong\mathfrak{e}_1/\gg(p)^0\cong\widetilde\gm\cap \mathfrak{e}_1\,,$$ 
an element $\x_1\in\widetilde\gm\cap \mathfrak{e}_1$ stable for the adjoint action of the stabilizer subgroup, modulo $\gg(p)^0$. The associated special vector field $\zeta_1$ projects then to a $G$-invariant section $Y_1= 
\pi_*(\zeta_1)$ of $\cE_1$ with $Y_1|_p=U_1|_p$. Hence $U_1=Y_1$, and claim $(ii)$ is established following the same lines as $(i)$, using that $\mathfrak{e}_2\subset\mathfrak{e}_3$ by tensoriality. Claim $(iii)$ and last statements are now straightforward.
\end{proof}

\begin{remark}
\label{rem:change-sign}
Lemma \ref{lem:from-left-to-right} allows to trade the bracket of sections with the bracket of symmetries, up to appropriate quotients and with an overall minus sign. The minus sign, however, can be removed by just changing sign to symmetries (for instance, if a basis of the symmetry algebra is given, it is enough to consider the opposite basis).
\end{remark}

\subsection{The leaf spaces of $\cK$ and $\cL$}\label{sec:2.4}

Let 
\begin{align}
\label{eq:projectionstoleafspaceL}
\widetilde\pi&:\cM\to\widetilde \cM\,,\\
\label{eq:projectionstoleafspaceK}
\vardbtilde{\pi}&:\cM\to\vardbtilde{\cM}\,,
\end{align}
be the leaf spaces of $\cL$ and $\cK$, respectively, with their natural projections. 
Their topological structure can be quite complicated, however 
we tacitly restrict to a sufficiently small open subset $\cU$ in $\cM$ where the leaf spaces are smooth manifolds. 

\begin{proposition}\label{prop:inducedcomplexstructure}
Let $(\cM,\cD,\cJ)$ be a $7$-dimensional CR-hypersurface. Then:
\begin{itemize}
\item[$(i)$] The Freeman filtrands $\cD\supset\cK\supset\cL$ descend to well-defined distributions $\widetilde \cD, \widetilde \cK$ on $\widetilde \cM$, and $\vardbtilde{\cD}$ on $\vardbtilde{\cM}$,
\item[$(ii)$] $\cJ$ projects to a complex structure $\widetilde \cJ$ on the bundle
	$\widetilde \cD/ \widetilde \cK\oplus  \widetilde \cK$ over $\widetilde \cM$, preserving both summands. It is an integrable complex structure on the second summand.
\end{itemize}
Moreover, if $(\cM,\cD,\cJ)$ is $3$-nondegenerate, then:
\begin{itemize}
\item[$(iii)$] $\widetilde \cD$ is a distribution on $\widetilde \cM$ with Cauchy characteristic space $\widetilde \cK$ of rank $2$, and $\vardbtilde{\cD}$ is a contact distribution on $\vardbtilde{\cM}$,
	\item[$(iv)$] $\cJ$ is not projectable neither to $\widetilde \cD$ or to $\vardbtilde{\cD}$,
	\item[$(v)$]  Any infinitesimal CR-symmetry is $\widetilde\pi$- and $\vardbtilde{\pi}$-projectable, and the Lie algebra $\gg$ embeds into the corresponding Lie algebra of symmetries on $\widetilde \cM$ and $\vardbtilde{\cM}$:
	\begin{align*}
	\gg&\cong\widetilde\pi_\star(\gg)\subset \inf(\widetilde \cM,\widetilde \cD,\widetilde \cK,\widetilde \cJ)\,,\\
	\gg&\cong\vardbtilde{\pi}_\star(\gg)\subset \inf(\vardbtilde{\cM},\vardbtilde{\cD})\,.
	\end{align*}
\end{itemize}
In particular $\gg$ is isomorphic to a subalgebra of the Lie algebra of infinitesimal symmetries of the $3$-dimensional 
contact manifold $(\vardbtilde{\cM},\vardbtilde{\cD})$. 
\end{proposition}

\begin{proof}
We will use Lemma \ref{lem:preliminarybrackets} and that distributions on $\widetilde\cM$ (resp. $\vardbtilde{\cM}$) are in bijective correspondence with distributions on $\cM$ that contain $\cL$ (resp. $\cK$) and are $\cL$-stable (resp. $\cK$-stable). 

Since $\cK$ is the Cauchy characteristic space of $\cD$, and it is integrable, claim $(i)$ is immediate.
Now $[\cL,\cL_{10}\oplus \cK_{01}]
\subset \cL_{10}\oplus \cK_{01}$. In other words, the bundle $\cL_{10}\oplus \cK_{01}$ is $\widetilde\pi$-projectable to $\widetilde \cK_{01}=\widetilde\pi_\star(\cL_{10}\oplus \cK_{01})$. Similarly $\widetilde \cK_{10}=\widetilde\pi_\star(\cK_{10}\oplus \cL_{01})$, so that $\widetilde \cK\otimes\mathbb C=\widetilde\pi_\star (\cK\otimes\mathbb C)=\widetilde \cK_{10}\oplus \widetilde \cK_{01}$. This proves that $\cJ$ induces a complex structure $\wt\cJ$ on $\wt\cK$, which is clearly integrable since $\cJ$ is so and $\cK$ is $\cJ$-stable. Similarly $[\cL,\cD_{10}\oplus \cK_{01}]\subset \cD_{10}\oplus \cK_{01}$, so the bundle $\cD_{10}\oplus \cK_{01}$ is $\widetilde\pi$-projectable. The quotient $(\cD_{10}\oplus \cK_{01})/\cK\otimes\mathbb C$ is then $\widetilde\pi$-projectable to 
$(\widetilde \cD/\widetilde \cK)_{10}=\widetilde\pi_\star((\cD_{10}\oplus \cK_{01})/\cK\otimes\mathbb C)$ and evidently
$(\widetilde \cD/\widetilde \cK)\otimes\mathbb C=(\widetilde \cD/\widetilde \cK)_{10}\oplus(\widetilde \cD/\widetilde \cK)_{01}$. This proves claim $(ii)$. 

Claim $(iii)$ is a direct consequence of definitions, we prove $(iv)$.
If $\cJ$ is projectable along $\widetilde\pi$, then the distribution $\cD_{10}\oplus \cL_{01}$ is $\widetilde\pi$-projectable, in other words, $[\cL,\cD_{10}\oplus \cL_{01}]\subset \cD_{10}\oplus \cL_{01}$. Hence $[\cL_{01},\cD_{10}]\subset \cD_{10}\oplus \cL_{01}$, which contradicts $3$-nondegeneracy.
If $\cJ$ is $\vardbtilde{\pi}$-projectable, we similarly see that $[\cK_{01},\cD_{10}]\subset \cD_{10}\oplus \cK_{01}$, again contradicting $3$-nondegeneracy.

Any symmetry $\xi\in\gg$ is seen to be $\widetilde\pi$- and $\vardbtilde{\pi}$-projectable, since it preserves $\cL$ and $\cK$, and the projections 
$\widetilde\pi_\star(\xi)$ and $\vardbtilde{\pi}_\star(\xi)$ are symmetries of the projected geometric structures. If $\widetilde\pi_\star(\xi)=0$ or $\vardbtilde{\pi}_\star(\xi)=0$, then $\xi$ is everywhere tangent to $\cK$, so $\xi=0$ by $3$-nondegeneracy. 
\end{proof}

From $(v)$ of Proposition \ref{prop:inducedcomplexstructure},
there is an action of $G$ as in \S\ref{sec:2.3} by local automorphisms on 
$(\widetilde \cM,\widetilde{\cD},\widetilde{\cK},\widetilde{\cJ})$ and $(\vardbtilde{\cM},\vardbtilde{\cD})$, so that 
$\widetilde\pi$ and $\vardbtilde{\pi}$ are $G$-equivariant.
We denote by $\widetilde{\mathfrak S}$ and $\vardbtilde{\mathfrak S}$ the analogous distributions defined by the local action of $G$ on $\widetilde M$ and $\vardbtilde M$, respectively. 

\begin{lemma}
The distribution $\mathfrak S$ is $\widetilde\pi$- and $\vardbtilde{\pi}$-projectable, with
$\widetilde{\mathfrak S}=\widetilde\pi_\star(\mathfrak S)$, $\vardbtilde{\mathfrak S}=\vardbtilde\pi_\star(\mathfrak S)$. 
\end{lemma}

\begin{proof}
The distribution $\mathfrak S+\cK$ is $\cK$-stable, hence $\vardbtilde\pi$-projectable, and 
$\vardbtilde\pi_\star(\mathfrak S+\cK)=\vardbtilde\pi_\star(\mathfrak S)=\vardbtilde{\mathfrak S}$ by construction.
The other case is proved similarly using $\mathfrak S+\cL$.
\end{proof}

\vskip0.1cm\par\noindent

\section{The upper bound on symmetry dimension}\label{sec:3}

In this section we get an upper bound on the dimension of the algebra of infinitesimal CR-symmetries of 
a 3-nondegenerate CR structure in dimension 7, provided the structure is not locally homogeneous.
More precisely, since the dimension of an orbit is a lower-semicontinuous function, 
the symmetry algebra acts with orbits of maximum dimension on an open dense subset $\cU$ of $\cM$. Let then $d$ be the orbit dimension through a generic point of $\cM$, and restrict to $\cU$ by analyticity. In other words, we may assume w.l.o.g. that the symmetry algebra acts intransitively everywhere, and then separately consider the possible values $0<d<7$. 

We stress that there can be singular orbits of smaller dimensions, however we restrict to the domain $\cU$ in $\cM$ where the orbits dimension is maximum, and the orbits foliate $\cU$.
In fact, we also have the freedom to shrink $\cU$ upon necessity, by regularity or other reasons. 
Since the CR structure and its symmetries are assumed analytic, this localization does not decrease generality. 
From now on, we always denote the domain and its localizations just by $\cU$. 

\vskip0.2cm\par
{\textit{Theorem \ref{T1} is a direct consequence of Theorems \ref{thm:dim123}, \ref{thm:dim4}, \ref{thm:dim6}, \ref{thm:dim5} from \S\ref{sec:3.3}-\S\ref{sec:3.6} that provide upper bounds on the symmetry dimension depending on the dimension $d$ of generic orbits.}}

\subsection{Some general arguments}\label{sec:3.1}
Recall the definition of the distribution $\mathfrak S$ determined by the algebra of infinitesimal CR-symmetries $\gg$ given in \S\ref{sec:2.3}. By the Frobenius theorem, there exist local rectifying coordinates and a foliation by maximal integral submanifolds of $\mathfrak S$ near any regular point $p\in\cU$. In other words,
we have a local diffeomorphism
 $
\cU\cong U\times V
 $,
where $U$ and $V$ have the local coordinates $(u^i)_{i=1}^d$ and $(v^j)_{j=1}^{7-d}$, respectively, and $\mathfrak S=TU=\langle\partial_{u^i}\rangle$. We stress that $V$ is not canonically defined and refer to the $v^j$'s as invariants of $\mathfrak S$. 

Any infinitesimal CR-symmetry has the form 
$$\xi=\sum_{i=1}^d \xi^i(u,v)\p_{u^i}\,,$$
and, in general, the group $G$ acts on each (local) orbit $\mathcal O^G_q\cong U\times\{v_q\}$ in a $v$-dependent fashion. If $\xi\in\gg(p)^0$ (i.e, $\xi|_p=0$), then the linear part $\Xi:=[\xi,-]|_p$ of the symmetry acts on $T_p\,\cU$, and it is represented by a block-diagonal matrix of the form
 $
\Xi=\begin{bsmallmatrix} * & *\\ 0 & 0\end{bsmallmatrix}
 $
w.r.t.\ the basis $(\p_{u^i},\p_{v^j})$. The proof of the following lemma is then straightforward and we omit it.
\begin{lemma}
The linearization $\Xi$ of an infinitesimal CR-symmetry $\xi\in\gg(p)^0$ satisfies the following properties:
\begin{itemize}
	\item[$(i)$] It acts trivially on the quotient $T_p\,\cU/T_p U$, where $T_p\,U=T_p\mathcal O^G_p$ is the tangent space to the orbit through $p$,
	\item[$(ii)$] It preserves the Freeman filtration $\cD\supset\cK\supset\cL$ at $p$,
	\item[$(iii)$] It preserves $\cJ$ on $\cD|_p$ (therefore also its restriction to $\cK|_p$ and to $\cL|_p$),
	\item[$(iv)$] It preserves the Freeman filtration $\cD_{10}\supset \cK_{10}\supset \cL_{10}$ at $p$.
\end{itemize}
\end{lemma}
 As a consequence, the linear part $\Xi\in\opp{End}(T_p\,\cU)$ of any symmetry $\xi\in\gg$ induces a graded morphism $\opp{gr}(\Xi)\in \opp{End}(\gm(p))$ of the metabelian Lie algebra $\gm(p)$, which preserves $\cJ|_p$ and the Freeman filtration $\gm_{-1}(p)= \cD|_p\supset \cK|_p\supset \cL|_p$. We consider the Tanaka--Weisfeiler filtration of the stabilizer subalgebra $\gg(p)^0$ as in \S\ref{sec:2.3} and its $0$-graded component $\gg_0(p)\subset\opp{End}(\gm_{-1}(p))$, with the $0$-graded symbol of $\xi$ given by $\Xi_0:=\Xi|_\cD\in\gg_0(p)$. 
 The full graded morphism $\opp{gr}(\Xi)$ can be simply recovered as the natural extension of $\Xi_0$ as a derivation of the Lie algebra $\gm(p)$.  
\begin{lemma}
\label{lem:idontknow}
If $\gg_0(p)=0$ then $\gg(p)^0=0$ and
$\dim\gg=\dim U<7$.
\end{lemma}
\begin{proof}
At regular points, the graded Lie algebra $\opp{gr}(\gg)$ is contained in the Tanaka prolongation
of the non-positively graded Lie algebra $\gm(p)\rtimes\gg_0(p)$, so $\gg_i(p)=0$ for $i\geq 0$.
\end{proof}
\begin{definition}
Given the subbundle $TU$
of the tangent bundle $T\cU$, we
set 
\begin{align*}
TU_\cD&=TU\cap \cD\,,\quad TU_\cD^\cJ=TU_\cD\cap \cJ(TU_\cD)\,,\\
TU_\cK&=TU\cap \cK\,,\quad TU_\cK^\cJ=TU_\cK\cap \cJ(TU_\cK)\,,\\
TU_\cL&=TU\cap \cL\,,\quad TU_\cL^\cJ=TU_\cL\cap \cJ(TU_\cL)\,.
\end{align*}
\end{definition}
As usual, we may localize to regular domains. More precisely, since the set of points of $\cU$ where the rank of the intersection of two vector bundles attains the minimum is open in $\cU$, we may assume that the above intersections have all constant ranks, i.e., they are bundles. 
\begin{lemma}
\label{lem:codimension}
If $\gg\neq 0$, then:
\begin{itemize}
	\item[$(i)$] $TU_\cD\subset TU$ has codimension 1,
	\item[$(ii)$] $\opp{Im}(\Xi_0)=\opp{Im}(\Xi|_\cD)\subset TU_\cD^\cJ$, and similarly for $\cK$ and $\cL$,
	\item[$(iii)$] If $TU_\cD^\cJ=0$, then $\dim\gg=\dim U<7$.
\end{itemize}
\end{lemma}
\begin{proof}
The first claim follows from the fact that $TU_\cD\subsetneq TU$ due to Lemma \ref{lem:1}. For the second claim, note that $\Xi\cD=\Xi(\cJ\cD)=\cJ(\Xi\cD)\subset TU_\cD\cap \cJ(TU_\cD)$, and similarly for $\cK$ and $\cL$. 

If $TU_\cD^\cJ=0$, then 
all $\Xi_0$ vanish by the second claim, so $\gg_0(p)=0$ and Lemma \ref{lem:idontknow} applies.
\end{proof}
\subsection{CR-symmetries when $TU_\cK$ has small codimension in $TU$}\label{sec:3.2}
Let $\vardbtilde{\pi}:\cU\to\vardbtilde{\cU}$ be the natural projection to the leaf space of $\cK$ as in
\S\ref{sec:2.4}, which is a $3$-dimensional contact manifold $\vardbtilde{\cU}$ with contact  distribution $\vardbtilde{\cD}$. We recall that a Lie algebra of contact vector fields on $\vardbtilde \cU$ consists of point vector fields (for some choice of Darboux cordinates) if and only if it preserves a line subbundle 
$\vardbtilde \cE\subset \vardbtilde \cD$
of the contact distribution. In this case, we have a natural projection 
$p:\vardbtilde{\cU}\to\vardbtilde{\cU}/\vardbtilde{\cE}$
from $\vardbtilde{\cU}$ to an open subset $\vardbtilde{\cU}/\vardbtilde{\cE}$ of the real plane
 and a point vector field  is the first jet prolongation $X^{(1)}$ of a vector field $X$ on the real plane.
If $(x,y,y_x)$ are jet coordinates on $\vardbtilde{\cU}$, the lift of $X=-\varphi_1(x,y)\partial_x+\varphi_0(x,y)\partial_y$ is the contact vector field
\begin{equation}
\label{eq:pointvf}
X^{(1)}=X+\big(\partial_x\varphi_0+y_x(\partial_y\varphi_0+\partial_x\varphi_1)+y_x^2\partial_y\varphi_1\big)\partial_{y_x}\,,
\end{equation}
with generating function $\varphi=\varphi_0+y_x\varphi_1$. Clearly $\vardbtilde{\cE}=\langle\partial_{y_x}\rangle$.

We say that a point in $\vardbtilde{\cU}/\vardbtilde{\cE}$ is generic if 
the induced orbits have constant dimension in a neighborhood of the point and that 
$q\in \cU$ is generic if $p(\vardbtilde{\pi}(q))$ is so. By possibly shrinking $\cU$, we may assume w.l.o.g. 
that it consists of generic points.

\begin{proposition}
\label{prop:smallcodimension}
If $\gg\neq 0$, then:
\begin{itemize}
	\item[$(i)$] If $TU_\cK\subset TU$ has codimension $1$, then $\dim\gg=1$,
	\item[$(ii)$] If $TU_\cK\subset TU$ has codimension $2$, then the Lie algebra of contact vector fields $\vardbtilde{\pi}_\star(\gg)$ consists of point vector fields. Moreover, either $\dim \gg\leq 3$ or $\gg$ is isomorphic  to the Abelian Lie algebra $\mathbb R^r$
	or its scaling extension $\mathbb R\ltimes\mathbb R^r$.
	In these last two cases, there exist local rectifying coordinates $(u^1,\ldots,u^{d-1},u^d,v)$ on $\cU$ such that $\gg$ is represented by the vector fields
	\vskip0.2cm\par\noindent
\begin{enumerate}
\item $\displaystyle\mathbb R^r=\mathrm{span}\big(\partial_{u^1},\ldots,\partial_{u^d},\sum_{i=1}^{d}\xi^i_{d+1}(v)\partial_{u^i},\ldots,
\sum_{i=1}^{d}\xi^i_{r}(v)\partial_{u^i}\big)$, 
\item $\displaystyle\mathbb R\ltimes\mathbb R^r=
		\begin{cases}\displaystyle
		\mathrm{span}\big(\partial_{u^1},\ldots,\partial_{u^d},\sum_{i=1}^{d}\xi^i_{d+1}(v)\partial_{u^i},\ldots,
		\sum_{i=1}^{d}\xi^i_{r}(v)\partial_{u^i},\sum_{i=1}^{d} u^i\partial_{u^i}\big)\;\;\text{or}\;\;\\
		\\
		\displaystyle
	\mathrm{span}\big(\partial_{u^1},\ldots,\partial_{u^{d-1}},\sum_{i=1}^{d-1}\xi^i_{d}(v)\partial_{u^i},\ldots,
		\sum_{i=1}^{d-1}\xi^i_{r}(v)\partial_{u^i},\partial_{u^d}+\sum_{i=1}^{d-1}u^i\partial_{u^i}\big)
		\end{cases}$
\end{enumerate}
\vskip0.3cm\par\noindent
for some functions $\xi^i_k$ of the coordinates $v$, with $r\geq 4$ for $\mathbb R^r$ and $r\geq 3$ for $\mathbb R\ltimes\mathbb R^r$.
\end{itemize}
\end{proposition}

\begin{proof}
Consider the projection $\vardbtilde{\pi}:\cU\to\vardbtilde{\cU}$ to the leaf space of $\cK$ and note that the restriction of the differential $\vardbtilde{\pi}_\star|_q$ to $T_qU\subset T_q\,\cU$ 
has Kernel $T_qU_\cK$. Since $\vardbtilde{\pi}$ sends $G$-orbits on $\cU$ onto $G$-orbits on $\vardbtilde{\cU}$, the latter have dimension $\dim(T_qU/T_qU_\cK)$, i.e., $\vardbtilde{\mathfrak S}$ has rank $\dim(T_qU/T_qU_\cK)$.
\vskip0.15cm\par\noindent
{\it Claim $(i)$}. By Lemma \ref{lem:1}, there exists a symmetry $\xi\in\gg$ such that its projection $\vardbtilde{\xi}=\vardbtilde{\pi}_\star(\xi)$ is nowhere vanishing around 
$\vardbtilde{\pi}(p)$. 
 Since $\vardbtilde{\mathfrak S}$ is a line bundle, the projection $\vardbtilde{\eta}=\vardbtilde{\pi}_\star(\xi)$ of any $\eta\in\gg$ is of the form $\vardbtilde{\eta}=f\vardbtilde{\xi}$, 
for some function $f$ on
$\vardbtilde{\cU}$. Since two non-trivial analytic contact vector fields are locally proportional if and only if they are homothetic, we get $\dim\vardbtilde{\pi}_\star(\gg)=1$. The claim
$\dim\gg=1$ follows then from $(v)$ of Proposition \ref{prop:inducedcomplexstructure}.
\vskip0.15cm\par\noindent
{\it Claim $(ii)$}. In this case $\vardbtilde{\pi}_\star(\gg)$ is an intransitive Lie algebra of contact vector fields on $\vardbtilde{\cU}$,
with $2$-dimensional orbits. Hence it consists of point vector fields,
see \cite[Corollary, p. 5]{DouKom}. It is then a simple matter of jet lifting via \eqref{eq:pointvf} the list in \cite[Table 1]{GKO} of analytic finite-dimensional Lie algebras of vector fields in the real plane (under changes of
local coordinates, at generic points)  to determine the contact Lie algebras with $2$-dimensional generic orbits.

Following the numeration of  \cite[Table 1]{GKO}, these are the lifts of the Lie algebras $(10)$, $(11)$, $(12)$ with $\alpha=1$, $(22)$ with $r=1$, which all have dimension $\leq 3$, and the lifts of $(20)$, $(21)$, which can be explicitly described as follows 
(locally, at generic points):
\vskip0.1cm\par\noindent
 \begin{enumerate}
\item[$(i)$] $\displaystyle\mathbb R^r=\mathrm{span}\big(\partial_y, f_1(x)\partial_y+f_1'(x)\partial_{y_x},\ldots,f_{r-1}(x)\partial_y+f'_{r-1}(x)\partial_{y_x}\big)$, $r\geq 4$,
\item[$(ii)$] $\displaystyle\mathbb R\ltimes\mathbb R^r=\mathrm{span}\big(\partial_y, f_1(x)\partial_y+f_1'(x)\partial_{y_x},\ldots,f_{r-1}(x)\partial_y+f'_{r-1}(x)\partial_{y_x},y\partial_y+y_x\partial_{y_x}\big)$, $r\geq 3$,
 \end{enumerate}
\vskip0.1cm\par\noindent
with the functions $1,f_1,\ldots,f_{r-1}$ linearly independent.
In particular $\vardbtilde{\pi}_\star(\gg)$ is either an Abelian Lie algebra $\mathbb R^r$ or its scaling extension $\mathbb R\ltimes\mathbb R^r$, the same holds for $\gg$ by (5) of Proposition \ref{prop:inducedcomplexstructure}.

If $\gg$ is Abelian, we may apply the straightening theorem for 
a systems of $d$  commuting vector fields that are linearly independent at $p\in\cU$ to get $\partial_{u^1},\ldots,\partial_{u^d}$. The remaining vector fields are sections of $\mathfrak S=\langle\partial_{u^i}\rangle$ commuting with each $\partial_{u^i}$, for $i=1,\ldots,d$. This is the case (1).
If $\gg\cong \mathbb R\ltimes\mathbb R^r$ and the first derived subalgebra $\mathbb R^r$ includes $d$ 
vector fields as above, then $\mathbb R^r$ can be straightened as in the Abelian case. 
It is then easy to see that the scaling element has the form $\sum_{i=1}^d(u^i+g^i)\partial_{u^i}$ for 
some functions $g^i=g^i(v)$, and the change 
of coordinates $\widehat u^i=u^i+g^i$, $\widehat v=v$, leads to the isomorphic subalgebra 
$\mathrm{span}\big(\partial_{\widehat u^1},\ldots,\partial_{\widehat u^d},\xi^i_{d+1}(\widehat v)
\partial_{\widehat u^i},\ldots,\xi^i_{r}(\widehat v)\partial_{\widehat u^i},\widehat u^i\partial_{\widehat u^i}\big)$.

Finally, if $\gg\cong \mathbb R\ltimes\mathbb R^r$ and $\mathbb R^r$ has only $d-1$ vector fields that are linearly independent at $p\in\cU$, we first straighten them to get $\partial_{u^1},\ldots,\partial_{u^{d-1}}$. 
Then, in a preferred local coordinate system, the scaling element has the form 
 $$
\sum_{i=1}^{d-1}(u^i+g^i)\partial_{u^i}+h\partial_{u^d}
 $$
for some functions $g^i$ and $h$ of the coordinates $u^d$ and $v$. The function $h$ cannot be identically zero. 
Restricting to a domain $\cU$ with $h\neq0$, we may rectify the coordinates $(u^d,v)$ in such a way that $h\equiv 1$. 
The final change of coordinates $\widehat u^i=u^i+h^i$, $\widehat u^d=u^d$, $\widehat v=v$, with $h^i$ a function of $u^d$ and $v$ satisfying $\partial_{u^d}h^i=h^i-g^i$, allows us to also set $g^i=0$ for all $i=1,\ldots, d-1$.  

We obtained $\partial_{u^1},\ldots,\partial_{u^{d-1}},\sum_{i=1}^{d-1}u^i\partial_{u^i}+\partial_{u^d}$. The remaining vector fields in $\mathbb R^r$ are of the form $\sum_{i=1}^{d-1}\xi^i(u,v)\partial_{u^i}$ (since $\mathfrak S=\langle\partial_{u^i}\mid 1\leq i\leq d\rangle$ and $\mathbb R^r$ does not have $d$ vector fields linearly independent at some point), the Lie algebra relations force $\partial_{u^k}\xi^i=0$ for all $k=1,\ldots,d$.
\end{proof}

{\it We will now proceed to the proof of Theorem \ref{T1} successively in the dimension of the (generic) orbits $d<7$.
Thanks to $(iii)$ of Lemma \ref{lem:codimension}, we may assume that the bundle $TU_\cD^\cJ$ is non-trivial. In fact, if $TU_\cD^\cJ=0$, then $\gg$ acts simply transitively on the orbits and
$\dim\gg=d$.}

\subsection{Dimensions $1\leq d\leq 3$}\label{sec:3.3}

\begin{theorem}\label{thm:dim123}
If the symmetry algebra $\gg$ of a 7-dimensional 3-nondegenerate CR-hypersurface acts with generic orbits of dimension 
$d=1$ or $d=2$, then $\gg$ acts simply transitively on orbits. If $\gg$ acts with generic orbits of dimension $d=3$, 
then $\dim\gg\leq 5$.
\end{theorem}

If $d=1,2$, then $TU_\cD$ has at most rank $1$ by $(i)$ of Lemma \ref{lem:codimension}, hence $TU_\cD^\cJ=0$. We now turn to the case $d=3$.

If $d=3$, the bundle $TU_\cD$ has rank $2$, so $TU_\cD^\cJ$ is either trivial or equal to $TU_\cD$. Thanks to our remark above, we may assume that $TU_\cD=TU_\cD^\cJ$, i.e., the bundle $TU_\cD$ is $\cJ$-stable. 
In particular $TU_\cK$ is $\cJ$-stable too.
If $TU_\cK=TU_\cD$, then $TU_\cK$ has codimension $1$ in $TU$ and $\dim\gg=1$ by $(i)$ of Proposition \ref{prop:smallcodimension}, which is a contradiction. Therefore $TU_\cK=0$, i.e., $TU$ and $\cK$ are complementary integrable distributions. Thus each $3$-dimensional orbit $\mathcal O^G_q$, $q\in\cU$, carries a natural structure of a $3$-dimensional contact CR manifold $(\mathcal O^G_q, TU_\cD,\cJ|_{TU_\cD})$.

Since $T\cU=TU\oplus\cK$ as complementary integrable distributions, we may take local rectifying coordinates on
 $
\cU\cong U\times V
 $,
where $U$ and $V$ have the local coordinates $(u^i)_{i=1}^3$ and $(v^j)_{j=1}^{4}$, respectively,
and $TU=\langle\partial_{u^i}\rangle$, $\cK=\langle\partial_{v^j}\rangle$. (See, e.g., \cite[Lemma, p. 182]{KN}.) 
An infinitesimal CR-symmetry
 $$
\xi=\sum_{i=1}^3 \xi^i(u,v)\p_{u^i}
 $$
preserves $\cK$, whence $\xi^i=\xi^i(u)$. Then $\gg$ is {\it effectively} represented on each orbit 
$\mathcal O^G_q$ as a Lie algebra of infinitesimal symmetries of a $3$-dimensional contact CR structure.

It is well-known that any maximally symmetric such structure is locally spherical, i.e., locally isomorphic 
to the $3$-dimensional CR sphere seen as the flag variety $SU(1,2)/B$, with $B$ the Borel subgroup.  
The submaximal CR symmetry dimension, i.e.\ the maximal dimension among all non-spherical Levi-nondegenerate
CR-manifolds, is $3$ by \cite{C,Kru2016}. 

Thus we restrict to the case where every orbit $(\mathcal O^G_q, TU_\cD,\cJ|_{TU_\cD})$, $q\in\cU$,
is locally spherical and $\gg\subset\mathfrak{su}(1,2)$.
The maximal dimension of a proper subalgebra of $\mathfrak{su}(1,2)$ is 5, attained only on Borel subalgebras, 
corresponding to the claim $\dim\gg\leq5$.
Assume now that $\gg\cong\mathfrak{su}(1,2)$, so $\dim(\gg|_{\mathcal O^G_q})=\dim\gg=8$ for all $q\in\cU$. 
We claim that this contradicts $3$-nondegeneracy.

First, we recall that the Tanaka-Weisfeiler filtration of the Lie algebra of infinitesimal CR-symmetries 
$\gg\cong\gg|_{\mathcal O^G_q}\cong\mathfrak{su}(1,2)$ of the $3$-dimensional contact CR manifold $(\mathcal O^G_q, TU_\cD,\cJ|_{TU_\cD})$ is the natural filtration associated with the $\mathbb Z$-grading 
corresponding to the Borel subalgebra $\mathfrak b=Lie(B)$ of $\mathfrak{su}(1,2)$: 
$$\gg=\gg_{-2}\oplus\cdots\oplus\gg_{+2}\,,$$ where $\gg_{-2}=\mathbb R$, $\gg_{-1}=\mathbb C$, $\gg_{0}=
\mathfrak{gl}_1(\mathbb C)$, $\gg_{+i}=\gg_{-i}^*$ for $i=1,2$, and $\gb=\gg_{\geq 0}=\gg_0\oplus\gg_1\oplus\gg_2$. 
The tangent space of the orbit at a preferred point can be identified with
$\gm=\gg/\mathfrak b\cong\gg_{-2}\oplus\gg_{-1}$ and invariant distributions correspond to $\mathfrak b$-stable subspaces 
of $\gm$. Since $\mathfrak{gl}_1(\mathbb C)$ acts with different spectrum on $\gg_{-2}$ and $\gg_{-1}$, there is in fact 
a unique invariant distribution of rank $2$: the one corresponding to $\gg_{-1}\subset\gm$. Similarly, 
invariant complex structures on the distribution correspond to $\mathfrak b$-invariant complex structures on $\gg_{-1}$, 
there is only one of them up to sign. In summary: there is a unique (up to sign) contact CR structure on 
$\mathcal O^G_q$ preserved by $\gg|_{\mathcal O^G_q}$. 

As we have already seen, symmetries in $\gg$ have the $v$-independent local form
$$\xi=\sum_{i=1}^3 \xi^i(u)\p_{u^i}\,,$$
so that the Lie algebra $\gg|_{\mathcal O^G_q}$ is not only abstractly isomorphic to $\gg|_{\mathcal O^G_p}$ for all $q\in\cU$ but actually {\it equal}, where we identify $\mathcal O^G_q\cong U\times\{v_q\}$ with $\mathcal O^G_p\cong U\times\{v_p\}$ in the obvious way.
The contact CR structure on $\mathcal O^G_q$ is then equal to that on $\mathcal O^G_p$ up to sign, but since $\cJ$ depends smoothly on $v$, they are equal. Hence $\cJ$ is projectable to $\vardbtilde{D}$, 
a contradiction by $(iv)$ of Proposition \ref{prop:inducedcomplexstructure}.

\subsection{Dimension $d=4$}\label{sec:3.4}

\begin{theorem}\label{thm:dim4}
If the symmetry algebra $\gg$ of a 7-dimensional 3-nondegenerate CR-hypersurface acts with generic orbits of dimension $d=4$, then $\dim\gg\leq 5$.
\end{theorem}

Here $\opp{rk} (TU_\cD)=3$ by $(i)$ of Lemma \ref{lem:codimension}, hence $\opp{rk} (TU_\cD^\cJ)$ is either $0$ or $2$. 
In the first case, we have $\dim\gg=d=4$  by $(iii)$ of Lemma \ref{lem:codimension}, so we may assume $\opp{rk} (TU_\cD^\cJ)=2$ from now on.
Thus, we have a flag of proper inclusions 
$
TU_\cD^\cJ\subset TU_\cD\subset TU
$
of vector bundles of rank $2$, $3$, $4$, respectively. We note that $\opp{rk}(TU_\cK)\geq \opp{rk}(TU_\cD)+\opp{rk}(\cK)-\opp{rk}(\cD)=1$, and turn to consider the cases $\opp{rk}(TU_\cK)=3,2,1$ separately.
We immediately note that if $\opp{rk}(TU_\cK)=3$, then $\dim\gg=1$ by $(i)$ of Proposition \ref{prop:smallcodimension}, which is a contradiction. 
\subsubsection{The case $\opp{rk}(TU_\cK)=2$}
We may fix a basis $(Y_1,Y_2)$ of local sections  of $TU_\cK$
and extend it to a basis $(Y_1,Y_2,Y_3)$ of local sections  of $TU_\cD$. Since $TU_\cK$ is integrable and $\cK$ is the Cauchy characteristic space of $\cD$, we see that 
$TU_\cD$ is integrable too. Moreover
$$TU^\cJ_\cK=TU_\cD^\cJ\cap \cK=TU_\cD^\cJ\cap TU_\cK$$ 
implies that $\opp{rk}(TU^\cJ_\cK)\geq \opp{rk}(TU_\cD^\cJ)+\opp{rk}(TU_\cK)-\opp{rk}(TU_\cD)=1$, so that $TU_\cK=TU_\cK^\cJ=TU^\cJ_\cD$.

If $\dim\gg\geq 4$, then $\gg$ is isomorphic to the Abelian Lie algebra 
\begin{equation*}
\label{eq:abelian4}
\mathbb R^r=\mathrm{span}\big(\xi_a=\partial_{u^a}\;(1\leq a\leq 4),\;\xi_{b}=\sum_{i=1}^{4}\xi^i_{b}(v)\partial_{u^i}\;(5\leq b\leq r)\big), 
\end{equation*}
or its scaling extension
\begin{equation*}
\displaystyle
\label{eq:scalingextension}
\mathbb R\ltimes\mathbb R^r=
		\begin{cases}
		\displaystyle
		\mathrm{span}\big(\xi_a=\partial_{u^a}\;(1\leq a \leq 4),\;\xi_{b}=\sum_{i=1}^{4}\xi^i_{b}(v)\partial_{u^i}\;(5\leq b\leq r),\;
		\xi_{r+1}=\sum_{i=1}^{4} u^i\partial_{u^i}\big)
		\\
		\text{or}\\
		\displaystyle
		\mathrm{span}\big(\xi_a=\partial_{u^a}\;(1\leq a\leq 3),\;\xi_b=\sum_{i=1}^{3}\xi^i_{b}(v)\partial_{u^i}\;(4\leq b\leq r),\;
		\xi_{r+1}=\partial_{u^4}+\sum_{i=1}^{3}u^i\partial_{u^i}\big)
		\end{cases}
		\end{equation*}
thanks to $(ii)$ of Proposition \ref{prop:smallcodimension}. 

\medskip

We first consider the cases where $\gg$
includes the Abelian Lie algebra $\mathbb R^4=\mathrm{span}\big(\partial_{u^1},\ldots,\partial_{u^4}\big)$ that acts simply transitively on the orbits. Now $TU=\langle\partial_{u^i}\rangle$ and, up to a permutation of the coordinates $u^{1},\ldots,u^4$, we may assume that $TU_\cK$ is generated by 
$
Y_1\equiv\partial_{u^1}\!\!\mod\langle\partial_{u^3},\partial_{u^4}\rangle$ and
$Y_2\equiv\partial_{u^2}\!\!\mod\langle\partial_{u^3},\partial_{u^4}\rangle$,
while $TU_\cD$ is generated by $Y_1, Y_2$ together with $Y_3\equiv\partial_{u^3}\!\!\mod\langle\partial_{u^4}\rangle$. 
It is convenient to explicitly write them as follows:
\begin{equation}
\label{eq:basisI}
\begin{aligned}
Y_1&=\partial_{u^1}+a_1Y_3+\alpha_1\partial_{u^4}\,,\\
Y_2&=\partial_{u^2}+a_2Y_3+\alpha_2\partial_{u^4}\,,\\
Y_3&=\partial_{u^3}+\alpha_3\partial_{u^4}\,,
\end{aligned}
\end{equation}
for some functions $a_1,a_2,\alpha_1,\alpha_2,\alpha_3$. Up to a permutation of the coordinates $v^{1},v^2,v^3$, the distribution $\cK$ is then generated by $Y_1,Y_2$ and further two vector fields 
\begin{equation}
\label{eq:basisII}
\begin{aligned}
Y_4&=\partial_{v^1}+a_4Y_3+\alpha_4\partial_{u^4}+\beta_4\partial_{v^3}\,,\\
Y_5&=\partial_{v^2}+a_5Y_3+\alpha_5\partial_{u^4}+\beta_5\partial_{v^3}\,,\\
\end{aligned}
\end{equation}
for some functions $a_4,a_5,\alpha_4,\alpha_5,\beta_4,\beta_5$.
Note that all the  functions appearing in \eqref{eq:basisI}--\eqref{eq:basisII} depend only on the coordinates $v$, since the Abelian algebra $\mathbb R^4$ preserves $TU_\cK$, $TU_\cD$ and $\cK$. 
Furthermore the distribution $\cK$ is integrable and $\mathbb R^4$-stable, therefore $TU+\cK$ is integrable too and there exists a change of coordinates
$\widehat u=u$, $\widehat v=\varphi(v)$ allowing us to set $\beta_4=\beta_5=0$. The distribution $\cD$ is then generated by $Y_1,\ldots,Y_5$ and an additional vector field 
\begin{equation}
\label{eq:basisIII}
Y_6=\partial_{v^3}+\alpha_6\partial_{u^4}\,,
\end{equation}
for, again, a function $\alpha_6$ that depends only on the coordinates $v$.

Now, the Lie brackets 
$[Y_4,Y_3]
=(\partial_{v^1}\alpha_3)\partial_{u^4}$ and 
$[Y_5,Y_3]
=(\partial_{v^2}\alpha_3)\partial_{u^4}$
are in $\cD$, since $\cK$ is the characteristic space of $\cD$, so they vanish and 
$\alpha_3=\alpha_3(v^3)$. Note that $\alpha_3$ cannot be constant, otherwise
$Y_3$ would be a non-trivial infinitesimal CR-symmetry everywhere tangent to $\cD$ (which is not possible by $3$-nondegeneracy).
We compute
$
[Y_{k+3},Y_l]=(\partial_{v^k}a_l)Y_3+(\partial_{v^k}\alpha_l)\partial_{u^4}
$
for $k,l=1,2$ and infer that all the functions in \eqref{eq:basisI} depend only on $v^3$, since $\cK$ is integrable. 

We established that all the Lie brackets between vector fields in \eqref{eq:basisI}-\eqref{eq:basisII} vanish, except possibly
$
[Y_4,Y_5]=\big((\partial_{v^1}a_5)-(\partial_{v^2}a_4)\big)Y_3
+\big((\partial_{v^1}\alpha_5)
-(\partial_{v^2}\alpha_4)\big)\partial_{u^4}
$. However, this Lie bracket has to be in $\cK$, so it vanishes as well and
\begin{align}
\label{eq:s=4eq7}
\partial_{v^1}a_5&=\partial_{v^2}a_4\Rightarrow \,a_4=\partial_{v^1}a,\; a_5=\partial_{v^2}a\,,\\
\label{eq:s=4eq8}
\partial_{v^1}\alpha_5&=\partial_{v^2}\alpha_4\Rightarrow \alpha_4=\partial_{v^1}\alpha, \alpha_5=\partial_{v^2}\alpha\,,
\end{align}
for some functions $a=a(v)$ and $\alpha=\alpha(v)$. To conclude the analysis of the Lie brackets, we need to consider those involving \eqref{eq:basisIII}:
\begin{equation*}
\label{eq:remainingbrackets}
\begin{aligned}
{}[Y_6,Y_1]&=\dot{a}_1Y_3+(a_1\dot{\alpha}_3+\dot{\alpha}_1)\partial_{u^4}\,,\\
[Y_6,Y_2]&=\dot{a}_2Y_3+(a_2\dot{\alpha}_3+\dot{\alpha}_2)\partial_{u^4}\,,\\
[Y_6,Y_3]&=\dot{\alpha_3}\partial_{u^4}\,,\\
[Y_6,Y_4]&=\dot{a}_4Y_3+\big(a_4\dot{\alpha_3}+\dot{\alpha_4}-(\partial_{v^1}\alpha_6)\big)\partial_{u^4}\,,\\
[Y_6,Y_5]&=\dot{a}_5Y_3+\big(a_5\dot{\alpha_3}+\dot{\alpha_5}-(\partial_{v^2}\alpha_6)\big)\partial_{u^4}\,,\\
\end{aligned}
\end{equation*}
where the dot refers to the differentiation w.r.t.\ the coordinate $v^3$. All the above Lie brackets except $[Y_6,Y_3]$ have to be in $\cD$, so 
\begin{align}
\label{eq:s=4eq9}
a_1\dot{\alpha}_3+\dot{\alpha}_1=0\,,\\
\label{eq:s=4eq10}
a_2\dot{\alpha}_3+\dot{\alpha}_2=0\,,\\
\label{eq:s=4eq11}
a_4\dot{\alpha_3}+\dot{\alpha_4}=\partial_{v^1}\alpha_6\,,\\
\label{eq:s=4eq12}
a_5\dot{\alpha_3}+\dot{\alpha_5}=\partial_{v^2}\alpha_6\,,
\end{align}
and, combining equations \eqref{eq:s=4eq7}-\eqref{eq:s=4eq8} together with \eqref{eq:s=4eq11}-\eqref{eq:s=4eq12}, we arrive at
\begin{align*}
\partial_{v^1}\big(a\dot{\alpha_3}+\dot{\alpha}-\alpha_6\big)=0\,,\\
\partial_{v^2}\big(a\dot{\alpha_3}+\dot{\alpha}-\alpha_6\big)=0\,,
\end{align*}
namely $\alpha_6=a\dot{\alpha_3}+\dot{\alpha}+\beta$ for some function $\beta=\beta(v^3)$ that depends only on $v^3$. The change of coordinates $\widehat u^i=u^i$ for $i=1,2$, $\widehat v^j=v^j$ for $j=1,2,3$ and $\widehat u^3=u^3-a$, $\widehat u^4=u^4-\alpha-\alpha_3a$ (recall that $a=a(v)$, $\alpha=\alpha(v)$  and $\alpha_3=\alpha_3(v^3)$), finally allows us to set to zero all the functions appearing in \eqref{eq:basisII}. We stress that
the infinitesimal CR-symmetries in $\gg$ are not modified by this change of coordinates, except for the scaling symmetry $\xi_{r+1}$ of $\gg=\mathbb R\ltimes\mathbb R^r$.

In summary:
\begin{equation}\label{eq:finalbasis}
\begin{aligned}
Y_1&=\partial_{u^1}+a_1Y_3+\alpha_1\partial_{u^4}\,,\qquad Y_2=\partial_{u^2}+a_2Y_3+\alpha_2\partial_{u^4}\,,\\
Y_3&=\partial_{u^3}+\alpha_3\partial_{u^4}\,,\qquad\qquad\quad\, Y_4=\partial_{v^1}\,,\\
Y_5&=\partial_{v^2}\,,\qquad\qquad\qquad\qquad\;\;\;
Y_6=\partial_{v^3}+\alpha_6\partial_{u^4}\,,
\end{aligned}
\end{equation}
for functions $a_1,a_2,\alpha_1,\alpha_2,\alpha_3$ of $v^3$ and $\alpha_6$ of $v$. The identities
\eqref{eq:s=4eq9}--\eqref{eq:s=4eq10} are still in force, while \eqref{eq:s=4eq11}--\eqref{eq:s=4eq12} now say that also $\alpha_6$ depends only on $v^3$. The Lie algebra $\gg$ of infinitesimal CR-symmetries is unchanged, with the exception of the scaling symmetry of $\gg=\mathbb R\ltimes\mathbb R^r$, which becomes a vector field of the form $\xi_{r+1}=\sum_{i=1}^{2} u^i\partial_{u^i}+\sum_{j=3}^4(u^j+f^j(v))\partial_{u^j}$.

This concludes our analysis of the structure equations of vector fields \eqref{eq:finalbasis} and symmetries. To move on, we shall consider the complex structure $\cJ$ of the CR manifold.
Expressing it w.r.t.\ the frame \eqref{eq:finalbasis} on $\cD$, we see that its coefficients depend only on the coordinates $v$, since each vector field in \eqref{eq:finalbasis} commute with the symmetries in $\mathbb R^4=\mathrm{span}\big(\partial_{u^1},\ldots,\partial_{u^4}\big)$. Moreover, since $\cJ$ preserves $TU^\cJ_\cD=TU_\cK$ and $\cK$, we have
\begin{equation*}
\begin{aligned}
\cJ Y_1&=\delta Y_1+\gamma Y_2\,,\\
\cJ Y_5&=\gamma_1 Y_1+\gamma_2 Y_2+\gamma_4 Y_4+\gamma_5Y_5\,,\phantom{\frac{1+\delta^2}{\gamma}}\\
\cJ Y_3&=\delta_1Y_1+\delta_2Y_2+\delta_3Y_3+\delta_{4}Y_4+\delta_5Y_5+\delta_6Y_6\,,
\end{aligned}
\end{equation*}
with $\gamma,\gamma_4,\delta_6$ nowhere vanishing. Now, we write any infinitesimal CR-symmetry 
$\xi\in\mathbb R^r\subset\gg$ as
$\xi=\sum_{i=1}^3 c^i(v)Y_i+c^4(v)\partial_{u^4}$
and we note that the functions $c^3$ and $c^4$ depend on $v^3$ only (consider the condition that the Lie brackets of $\xi$ 
with $Y_4$ and $Y_5$ are both in $\cK$). We compute
\begin{equation}
\label{eq:17}
\begin{aligned}
0&=(L_{\xi}\cJ)(Y_5)=[\xi,\cJ Y_5]-\cJ[\xi,Y_5]\\
&=-\gamma_4\sum_{i=1}^2(\partial_{v^1} c^i) Y_i-\gamma_5\sum_{i=1}^2(\partial_{v^2} c^i) Y_i+
\sum_{i=1}^2(\partial_{v^2} c^i) \cJ Y_i
\\
&=\big(-\gamma_4(\partial_{v^1} c^1)+(\delta-\gamma_5)(\partial_{v^2} c^1) 
-\frac{1+\delta^2}{\gamma}(\partial_{v^2} c^2) 
\big)   Y_1\\
&+
\big(-\gamma_4(\partial_{v^1} c^2)
+\gamma(\partial_{v^2} c^1)-(\delta+\gamma_5)(\partial_{v^2} c^2)
\big)Y_2\,,
\end{aligned}
\end{equation}
where we used that $\cJ Y_2=-\frac{1+\delta^2}{\gamma}Y_1-\delta Y_2$, and
 \begin{equation}
\begin{aligned}
0&=(L_{\xi}\cJ)(Y_3)=[\xi,\cJ Y_3]=\sum_{k=4}^6\delta_k[\xi,Y_k]=
\sum_{k=1}^3\delta_{3+k}[\xi,\partial_{v^k}]\\
&=-\big(\delta_4(\partial_{v^1}c^1)+\delta_5(\partial_{v^2}c^1)+\delta_6\dot{c}^1\big) Y_1-\big(\delta_4(\partial_{v^1}c^2)+\delta_5(\partial_{v^2}c^2)+\delta_6\dot{c}^2\big) Y_2\\
&\;\;\,\,\,-\delta_6\big(\dot{c}^3+\sum_{i=1}^2c^i\dot{a_i}\big)Y_3-\delta_6\big(\dot{c}^4+\dot{\alpha_3}c^3\big)\partial_{u^4}\,,
\end{aligned}
 \end{equation}
where we finally used the identities \eqref{eq:s=4eq9}-\eqref{eq:s=4eq10}. This leads to the following system of four differential equations on four unknowns
\begin{equation}
\label{eq:19}
\begin{aligned}
\dot{c}^1&=-\frac{\delta_4}{\delta_6}(\partial_{v^1}c^1)-\frac{\delta_5}{\delta_6}(\partial_{v^2}c^1)\,,\\
\dot{c}^2&=-\frac{\delta_4}{\delta_6}(\partial_{v^1}c^2)-\frac{\delta_5}{\delta_6}(\partial_{v^2}c^2)\,,\\
\dot{c}^3&=-\sum_{i=1}^2c^i\dot{a_i}\,,\\
\dot{c}^4&=-\dot{\alpha_3}c^3\,,
\end{aligned}
\end{equation}
together with the additional constraints on the initial data given by
 \begin{equation}\label{eq:20}
\begin{pmatrix}
\dot{a}_1 & \dot{a}_2 & 0 & 0 \\
0 &0 & \dot{a}_1 & \dot{a}_2\\
-1 & 0 & \frac{\delta-\gamma_5}{\gamma_4} & -\frac{1+\delta^2}{\gamma\gamma_4}\\
0 & -1 & \frac{\gamma}{\gamma_4} & -\frac{\delta+\gamma_5}{\gamma_4}
\end{pmatrix}
\begin{pmatrix}
\partial_{v^1}c^1\\ \partial_{v^1}c^2\\ \partial_{v^2}c^1\\ \partial_{v^2}c^2
\end{pmatrix}
=
\begin{pmatrix}
0\\ 0\\ 0\\ 0
\end{pmatrix}\,.
 \end{equation}
These constraints are obtained from the vanishing of \eqref{eq:17} and by deriving  the next-to-last equation of \eqref{eq:19} w.r.t.\ $v^1$ and $v^2$. The determinant of the square matrix in \eqref{eq:20} is given by
$\frac{1}{\gamma\gamma_4}\big((\dot{a}_1)^2+(\dot{a}_1\delta+\dot{a}_2\gamma)^2\big)$ and it vanishes if and only if 
$\dot{a}_1=\dot{a}_2=0$.

We assume first that this determinant is non-zero on an open subset of $\cU$ and consider
$\xi\in\mathbb R^r\subset\gg$ that is vanishing at a point $p\in\cU$. Then
\eqref{eq:19} becomes a system of ODEs on the functions $c^k=c^k(v^3)$ for $k=1,\ldots,4$,
whose unique solution is $c^k=0$, i.e., $\xi=0$. If the determinant vanishes on an open subset, then $a_1$ and $a_2$ are constants, $c^3=c^4=0$ by \eqref{eq:19} and $\xi$ is a section of $\cK$. Again, this implies $\xi=0$ by $3$-nondegeneracy.
In summary $\dim\gg=4$ if $\gg=\mathbb R^r$ and $\dim\gg=5$ if $\gg=\mathbb R\ltimes\mathbb R^r$.

\medskip

The case with $\xi_{r+1}=\partial_{u^4}+\sum_{i=1}^{3}u^i\partial_{u^i}$ a CR-symmetry is easier.
Let $\Pi=\langle \partial_{u^1},\partial_{u^2},\partial_{u^{3}}\rangle$ be the distribution defined by the first derived subalgebra $\mathbb R^r$ of $\gg=\mathbb R\ltimes\mathbb R^r$, and note that $1\leq\opp{rk}(\Pi\cap TU_\cK)\leq 2$, since $\Pi$ and  $TU_\cK$ are subbundles of $TU$. If $\opp{rk}(\Pi\cap TU_\cK)=2$, then $\Pi\cap TU_\cK$ has codimension $1$ in $\Pi$ so that $\dim\mathbb R^r=\dim\vardbtilde{\pi}_\star(\mathbb R^r)=1$ by $(i)$ of Proposition \ref{prop:smallcodimension}  (the proof extends verbatim with $\mathbb R^r$ in place of $\gg$). Then $\dim\gg=2$, which is a contradiction. 

We now focus on $\opp{rk}(\Pi\cap TU_\cK)=1$. Up to a permutation of the coordinates $u^{1},u^2,u^3$, we may assume that $\Pi\cap TU_\cK$ is generated by a vector field $Z=\partial_{u^1}+a^2\partial_{u^2}+a^3\partial_{u^3}$. Since the distribution $\Pi\cap TU_\cK$ is $\gg$-stable, we see that $a^2$ and $a^3$ are functions of the coordinates $v$ only. Either the function $a^2$ or $a^3$ is not constant, otherwise $Z$ would be a non-trivial infinitesimal CR-symmetry everywhere tangent to $\cK$. Hence, we may assume $da^2\neq 0$.

Consider an infinitesimal CR-symmetry
$
\xi=\sum_{i=1}^{3}\xi^i(v)\partial_{u^i}\in\mathbb R^r
$.
We fix a point $q\in\cU$ and consider the ``constant'' infinitesimal CR-symmetry 
$$\lambda=\sum_{i=1}^{3}\lambda^i\partial_{u^i}\,,\quad\lambda^i=\xi^i|_q\,,$$ 
constructed out of its value at $q$. 
Then $\widehat\xi:=\xi-\lambda\in\gg(q)^0$ and 
\begin{itemize}
	\item[$(i)$] its linear part $\Xi$ satisfies $\opp{Im}(\Xi)=\opp{Im}(\Xi|_{TV})\subset\Pi$, where $TV=\langle \partial_{v^1},\partial_{v^2},\partial_{v^{3}}\rangle$,
	\item[$(ii)$] its $0$-graded symbol
$\Xi_0=\Xi|_\cD$ satisfies $\opp{Im}(\Xi_0)\subset TU^\cJ_\cD=TU_\cK$ by $(ii)$ of Lemma \ref{lem:codimension}.
\end{itemize}
Since $\opp{rk}(\cD/TU_\cD)=3$, the vector bundle $TV$ is identified with the quotient $\cD/TU_\cD$ via the natural projection
$\pi_{TV}|_\cD:\cD\to TV$ corresponding to the decomposition $T\cU=TU\oplus TV$.
In summary 
$\opp{Im}(\Xi)=\opp{Im}(\Xi|_{TV})=\opp{Im}(\Xi_0)\subset TU_\cK$.
Since $\opp{Im}(\Xi)\subset \Pi$, we finally get that
$\opp{Im}(\Xi)\subset\Pi\cap TU_\cK=\mathrm{span}\big(Z|_q\big)$.
Since $q\in\cU$ is generic, this condition holds at all points of $\cU$ and translates into the following system of differential equations
\begin{equation}
\label{eq:d-v}
\begin{aligned}
d\xi^2&=a^2 d\xi^1\,,\\
d\xi^3&=a^3d\xi^1\,,
\end{aligned}
\end{equation}
where $d$ is the exterior differential in the coordinates $v$. If $\xi^1$ is constant, 
then $\xi^2$ and $\xi^3$ are constant too and $\xi$ is just a ``constant'' infinitesimal CR-symmetry.

We set $x:=\xi^1$  and assume $dx\neq 0$ from now on.
Taking the exterior derivative of \eqref{eq:d-v} yields $da^2\wedge dx=0$ and $da^3\wedge dx=0$,
so that $a^2=a^2(x)$, $a^3=a^3(x)$ are functions of $x$. By our assumptions, $a_2$ is (locally) invertible.
Thanks to \eqref{eq:d-v}, we also have $\xi^2=\xi^2(x)$, $\xi^3=\xi^3(x)$ and
$\dot{\xi}^2=a^2$, $\dot{\xi}^3=a^3$,
where the dot is the differentiation w.r.t.\ $x$. In summary, we arrive at
$$
\xi=x\partial_{u^1}+\sum_{i=2}^{3}\Xi^i(x)\partial_{u^i}\,,
$$
where $\Xi^i(x)$ is a primitive of $a^i(x)$, for $i=2,3$. 
We may carry over the arguments unchanged for another infinitesimal CR-symmetry
$\eta=\sum_{i=1}^{3}\eta^i(v)\partial_{u^i}$
such that $y:=\eta^1$ is not constant. Then 
$$
\eta=y\partial_{u^1}+\sum_{i=2}^{3}\Theta^i(y)\partial_{u^i}\,,
$$
where $\Theta^i(y)$ is a primitive of $a^i(y)$, for $i=2,3$. 

Now, the crucial observation is that
$a^2$ is locally invertible also as a function of $y$, so in turn we may write $y=\varphi(x)$ and
$$
\eta=\varphi(x)\partial_{u^1}+\sum_{i=2}^{3}\Theta^i(\varphi(x))\partial_{u^i}\,.
$$
The system of differential equations $d\eta^i=a^id\eta^1$, $i=1,2$, analogous to \eqref{eq:d-v} turn then into
$$
\frac{d\Theta^i}{dy}|_{\varphi(x)}\cdot\dot{\varphi}|_x=a^i|_x\cdot\dot{\varphi}|_x\Rightarrow
\frac{d\Theta^i}{dy}|_{\varphi(x)}=a^i|_x
$$
and since $\frac{d\Theta^i}{dy}|_{\varphi(x)}=a^i|_{\varphi(x)}$ by construction, we finally arrive at 
$a^i|_x=a^i|_{\varphi(x)}=a^i|_y$. Taking $i=2$, yields $y=x$ and the difference of $\xi$ and $\eta$ is just a 
 ``constant'' infinitesimal CR-symmetry. In summary $r\leq 4$ and $\dim\gg\leq 5$.
\subsubsection{The case $\opp{rk}(TU_\cK)=1$}
Here $TU_\cK^\cJ=0$, so $TU_\cD^\cJ$ is transversal to 
$\cK$ and $\cD=\cK\oplus TU_\cD^\cJ$, $TU_\cD=TU_\cK\oplus TU_\cD^\cJ$.
Since $TU+\cK=T\cU$ is integrable, we may integrate the distributions using adapted coordinates 
$(u^1,u^2,u^3,z,v^1,v^2,v^3)$ so that
 \begin{equation}\label{eq:distribution-aligned}
\begin{aligned}
TU&=\langle \partial_{u^1},\partial_{u^2},\partial_{u^3},\partial_z\rangle\,,\\
\cK&=\langle\partial_{v^1},\partial_{v^2},\partial_{v^3},\partial_z\rangle\,,\\
TU_\cK&=\langle\partial_z\rangle\,,
\end{aligned}
 \end{equation}
and the algebra of infinitesimal CR-symmetries consists of vector fields 
$\xi=\sum_{i=1}^3\xi^i\partial_{u^i}+\xi^4\partial_z$, where the functions $\xi^i$, $i=1,2,3$, 
depend only on the coordinates $u$, since $\cK$ is $\gg$-stable. 
By $(v)$ of Proposition \ref{prop:inducedcomplexstructure}, the projection to the leaf space of $\cK$ 
is injective on CR-symmetries, i.e., $\gg\cong\vardbtilde{\pi}_\star(\gg)$.
Explicitly, we decompose any $\xi$ as $\xi=\xi'+\xi''$, with $\vardbtilde{\pi}_\star(\xi)=\xi':=\sum_{i=1}^3\xi^i\partial_{u^i}$, 
$\xi'':=\xi^4\partial_z$, and the component $\xi'$ fully determines $\xi$.
In particular, $\gg$ {\it is effectively represented} 
on each orbit $\mathcal O^G_q=\{(u,z,v)\in\cU\mid v=v_q\}$ and also on 
any $3$-dimensional hypersurface $\mathcal S_{z=z_q,v=v_q}=\{(u,z,v)\in\cU\mid z=z_q, v=v_q\}$ of a given orbit.

Up to a permutation of the coordinates $u^{1},u^2,u^3$, the distribution $TU_\cD$ is generated by $\partial_z$ and two vector fields
\begin{equation}
\label{eq:generatorsTU_D}
\begin{aligned}
X_1&=\partial_{u^1}+a_1\partial_{u^3}\,,\\
X_2&=\partial_{u^2}+a_2\partial_{u^3}\,.
\end{aligned}
\end{equation}
Here $\cD=\langle\partial_{v^1},\partial_{v^2},\partial_{v^3},\partial_z,X_1,X_2\rangle$, where $a_1$ and $a_2$ depend only on the coordinates $u$, since $\cK$ is the Cauchy characteristic space of $\cD$. 
Finally
$TU_\cD^\cJ$ is generated by two vector fields that need to coincide with \eqref{eq:generatorsTU_D} modulo $TU_\cK$, i.e., we have
$\mathbb X_k=X_k+\alpha_k\partial_z$ for $k=1,2$.
We may then write
\begin{equation*}
\begin{aligned}
\cJ\mathbb X_1&=\delta \mathbb X_1+\gamma \mathbb X_2\,,\\
\cJ\mathbb X_2&=-\frac{1+\delta^2}{\gamma}\mathbb X_1-\delta\mathbb X_2\,,
\end{aligned}
\end{equation*}
for some functions $\delta$ and $\gamma$, with the latter nowhere vanishing. 
The fact that an infinitesimal CR-symmetry $\xi$ preserves $\cD$ is equivalent to the condition that 
$\xi'$ preserves $\langle X_1,X_2\rangle$. 
(The action of $\xi$ and $\xi'$ are then related simply by $[\xi',X_k]\equiv [\xi,\mathbb X_k]\!\mod TU_\cK$, $k=1,2$.)
We now split our analysis in two subcases, depending on whether $\opp{rk}(TU_\cL)=1$ or $\opp{rk}(TU_\cL)=0$.

\medskip

If $\opp{rk}(TU_\cL)=1$, then $TU_\cL=TU_\cK=\langle\partial_z\rangle$, and we compute
\begin{equation}
\label{eq:useful-computation}
[\partial_z,\mathbb X_1^{10}]\equiv-\frac{i}{2}\big((\partial_z\delta) \mathbb X_1+(\partial_z\gamma) \mathbb X_2)\mod TU_\cK\,.
\end{equation}
Since $[\cL,\cD_{10}]\subset \cD_{10}\oplus \cK_{01}$ by Lemma \ref{lem:preliminarybrackets} and \eqref{eq:useful-computation} is a purely imaginary vector field in $\cD\otimes\mathbb C$, we see that \eqref{eq:useful-computation} is  in $\cK\otimes \mathbb C$ and that
$\partial_z\delta=\partial_z\gamma=0$. 
We now claim that each submanifold 
$\mathcal S_{z=z_q,v=v_q}$ has a natural structure of $3$-dimensional contact CR manifold preserved by $\vardbtilde{\pi}_\star(\gg)$. 

First
$T\mathcal S\cong (TU/TU_\cK)|_{\mathcal S}$
with the contact distribution that corresponds to $(TU_\cD/TU_\cK)|_{\mathcal S}$ and inherits a natural almost complex structure $\cI$ via the identification $TU_\cD^\cJ\cong (TU_\cD/TU_\cK)$. 
It is is automatically integrable by dimensional reasons. 
Explicitly 
\begin{equation*}
\begin{aligned}
\cI X_1&=\delta X_1+\gamma X_2\,,\\
\cI X_2&=-\frac{1+\delta^2}{\gamma}X_1-\delta X_2\,,
\end{aligned}
\end{equation*}
and a straightforward computation using that $\delta$ and $\gamma$ are $z$-independent and that $\xi$ is $\partial_v$-independent shows that $\xi'$ preserves this CR structure. In particular $\dim\gg=\dim\vardbtilde{\pi}_\star(\gg)\leq 8$. 
We then argue exactly as in the case $d=3$ in \S\ref{sec:3.3} to get that $\dim\gg\leq 5$ as otherwise the complex structure 
on $\cD$ is projectable to $\vardbtilde{\cD}$, which is not possible by $(iv)$ of Proposition \ref{prop:inducedcomplexstructure}.

\medskip

If $\opp{rk}(TU_\cL)=0$, we proceed differently. In fact, the natural morphisms $TU_\cK\hookrightarrow \cK\twoheadrightarrow\cK/\cL$ of real vector bundles compose to an embedding of $TU_\cK$ in $\cK/\cL\cong \cK_{10}/\cL_{10}$, so the latter has a canonical real line subbundle
and $(ii)$ of
Lemma \ref{lem:simple-but-useful} applies.
Modulo a finite cover of $\cU$, the section $Y_{10}$ is canonical, hence preserved by $\gg$, 
and the complex line bundle
$\cD_{10}/\cK_{10}$ has a canonical real line subbundle, hence $\gg$-stable. Since 
$TU^\cJ_\cD\cong \cD/\cK$ are both $\cJ$-stable, we have $TU_\cD^\cJ\cong(TU_\cD^\cJ)_{10}\cong\cD_{10}/\cK_{10}$ and
there exists a $\gg$-stable real line subbundle $\ell$ in  $TU_\cD^\cJ$. 
Applying $\cJ$, we decompose $TU_\cD^\cJ=\ell\oplus \cJ\ell$ into the direct sum of two $\gg$-stable line subbundles.
 
In conclusion the Lie algebra $\gg$ is effectively represented on any orbit $\mathcal O^G_q$ as infinitesimal symmetries 
of a $4$-dimensional manifold endowed with  a rank $3$ distribution 
$TU_\cD|_{\mathcal O^G_q}\subset TU|_{\mathcal O^G_q}$ whose symbol is 
$\gm(q)=\gm_{-2}(q)\oplus\gm_{-1}(q)\cong \big(TU|_q/TU_\cD|_q\big)\oplus TU_\cD|_q$, 
where $TU|_q/TU_\cD|_q$ is $1$-dimensional and 
$TU_\cD|_q=TU_\cK|_q\oplus TU_\cD^\cJ|_q=\mathbb R Y|_q\oplus \big(\ell|_q\oplus \cJ\ell|_q\big)$ 
splits into $3$ preserved lines.
It is straightforward to see that the reduced structure algebra $\mathfrak f_0(q)$
is $1$-dimensional, 
generated by the semisimple element with eigenvalues $0$ on $\mathbb R Y|_q$, $-1$ on $\ell|_q\oplus \cJ\ell|_q$ 
and $-2$ on $TU|_q/TU_\cD|_q$, and that 
$\mathfrak f_k(q)=0$ for all $k\geq 1$. 
Hence $\dim\gg\leq 5$.

 \subsection{Dimension $d=6$}\label{sec:3.5}

We consider now the case $d=6$ prior to the case $d=5$, because for the latter in \S\ref{sec:3.6}
we will use all the different strategies developed from \S\ref{sec:3.3} to \S\ref{sec:3.5}.

 \begin{theorem}
\label{thm:dim6}
If the symmetry algebra $\gg$ of a 7-dimensional 3-nondegenerate CR-hypersurface acts with generic orbits of dimension $d=6$, then $\gg$ acts simply transitively on orbits.
 \end{theorem}

If $d=6$, then $\opp{rk} (TU_\cD)=5$, $\opp{rk} (TU_\cD^\cJ)=4$ and $3\leq\opp{rk}(TU_\cK)\leq4$ 
by dimensional reasons. We therefore split our analysis into two main cases, which will split in turn into subcases.

\subsubsection{The case $\opp{rk}(TU_\cK)=4$}

In this case $\cK=TU_\cK=TU^\cJ_\cD$. By $(ii)$ of Proposition \ref{prop:smallcodimension}, 
we know that $\gg$ is isomorphic to the Abelian Lie algebra 
\begin{equation*}
\label{eq:abelian6}
\mathbb R^r=\mathrm{span}\big(\xi_a=\partial_{u^a}\;(1\leq a\leq 6),\;\xi_{b}=\sum_{i=1}^{6}\xi^i_{b}(v)\partial_{u^i}\;(7\leq b\leq r)\big)\,, 
\end{equation*}
or its scaling extension
 \begin{equation*}\displaystyle\label{eq:scalingextension6}
\mathbb R\ltimes\mathbb R^r=
		\begin{cases}\displaystyle
\mathrm{span}\big(\xi_a=\partial_{u^a}\;(1\leq a\leq 6),\;\xi_{b}=\sum_{i=1}^{6}\xi^i_{b}(v)\partial_{u^i}\;(7\leq b\leq r),\;
\xi_{r+1}=\sum_{i=1}^{6} u^i\partial_{u^i}\big) \\
		\text{or}\\
		\displaystyle
\mathrm{span}\big(\xi_a=\partial_{u^a}\;(1\leq a\leq 5),\;\xi_b=\sum_{i=1}^{5}\xi^i_{b}(v)\partial_{u^i}\;(6\leq b\leq r),\;
\xi_{r+1}=\partial_{u^6}+\sum_{i=1}^{5}u^i\partial_{u^i}\big)
		\end{cases}
 \end{equation*}
Consider the Abelian case. 
Let $\xi=\sum_{i=1}^{6}\xi^i(v)\partial_{u^i}$ be an infinitesimal CR-symmetry vanishing at $p\in\cU$ and 
$\Xi=-\sum_{i=1}^{6}(\partial_{v}\xi^i)|_p\;\partial_{u^i}\otimes dv$
its linear part. Let $w$ be a vector in $TU_\cD|_p$ but not in $TU_\cD^\cJ|_p$, so that $\cJ w$ is in $\cD|_p$ but not in $TU_\cD|_p$. Hence $\cJ w=\lambda\partial_v\mod TU|_p$ for some $\lambda\neq 0$ and $0=\cJ \Xi w=\Xi \cJ w=\lambda \Xi\partial_v$. This shows that $\Xi=0$, so $\gg_0(p)=0$ and $\dim\gg=\dim U=6$. 

Next, the case of the scaling extension with first derived subalgebra $\mathbb R^r$ acting intransitively on $TU$ is similar,
so we omit the proof. 

Finally, if $\gg$ is the scaling extension with transitive first derived subalgebra, we may write any infinitesimal CR-symmetry 
as $\xi=\sum_{i=1}^{6}\xi^i(v)\partial_{u^i}+\mu\xi_{r+1}$, for some $\mu\in\mathbb R$. If $\xi$ vanishes at $p\in\cU$, 
its linear part is 
$\Xi=-\sum_{i=1}^{6}(\partial_{v}\xi^i)|_p\;\partial_{u^i}\otimes dv-\mu\operatorname{Id}_{TU|_p}$,
and its restriction $\Xi|_\cD$ to $\cD|_p$ has rank at least $\opp{rk}(TU_\cD)=5$ if $\mu\neq 0$. 
Since $\opp{Im}(\Xi|_\cD)\subset TU_\cD^\cJ$ by (ii) of Lemma \ref{lem:codimension}, we immediately see that 
$\mu=0$ and, arguing as before, that $\gg_0(p)=0$ and $\dim\gg=6$. This is actually a contradiction: 
the case of the scaling extension with transitive first derived subalgebra cannot happen.

\medskip\par
{\it If $\opp{rk}(TU_\cK)=3$, then $\opp{rk}(TU_\cK^\cJ)=2$ by dimensional reasons 
and we further split our analysis into two subcases, depending on whether $\opp{rk}(TU_\cL)=2$ or $\opp{rk}(TU_\cL)=1$.}

\subsubsection{The case $\opp{rk}(TU_\cK)=3$, $\opp{rk}(TU_\cL)=2$}\label{subsubsec:352}

By Lemma \ref{lem:simple-but-useful}, we may consider local trivializing sections satisfying the normalization conditions 
\eqref{eq:norm-conditions}. For our purposes, it is convenient to let $X_{10}$, $Y_{10}$, $Z_{10}$ be corresponding 
local sections of $\cD_{10}$, $\cK_{10}$, $\cL_{10}$ (instead of quotient bundles), determined up to transformations
$X_{10}\mapsto \lambda e^{i\varphi}X_{10}\!\mod \cK_{10}$, $Y_{10}\mapsto e^{2i\varphi}Y_{10}\!\mod \cL_{10}$, 
and $Z_{10}\mapsto \lambda^{-1}e^{3i\varphi} Z_{10}$, for $\lambda:\cU\rightarrow \mathbb R_+$ and 
$e^{i\varphi}:\cU\rightarrow S^1$, see $(i)$ of Lemma \ref{lem:simple-but-useful}. The normalization conditions read as
$\clL_{2}(Y_{10},\overline{X_{10}})=X_{10}\!\mod \cK_{10}\oplus \cD_{01}$,
$\clL_{3}(Z_{10},\overline{X_{10}})=Y_{10}\!\mod \cL_{10}\oplus \cD_{01}$.

Here $\cL=TU_\cL=TU_\cK^\cJ$, so we have the real line subbundle $TU_\cK/TU_\cL$ of $\cK/\cL\cong  \cK_{10}/\cL_{10}$. Then $(ii)$ of Lemma \ref{lem:simple-but-useful} applies: we may take $Y$ a local section of $TU_\cK$ that is not in $\cL$ at all points of $\cU$, and assume that $e^{i\varphi}=1$, modulo a finite cover of $\cU$. Finally we note that $Z$ is a section of $\cL=TU_\cK^\cJ$ and, since $\cD/\cK\cong TU^\cJ_\cD/TU^\cJ_\cK=TU^\cJ_\cD/\cL$, we are also free to ask that $X$ is a section of $TU^\cJ_\cD$, which is not in $\cL$ at all points of $\cU$.

In summary
 \begin{equation}
\label{eq:summary-for-preliminary-frames}
\begin{aligned}
TU_\cD&=TU^\cJ_\cD\oplus\langle Y\rangle\,,\\
TU_\cD^\cJ&=\langle X,\cJ X\rangle\oplus TU_\cK^\cJ \,,\\
TU_\cK^\cJ&=\langle Z,\cJ Z\rangle\,,
\end{aligned}
 \end{equation}
with the local sections uniquely defined up to transformations of the form
\begin{equation}
\label{eq:gauge-freedom}
X_{10}\mapsto \lambda X_{10}\!\mod \cL_{10}\,,\quad Y_{10}\mapsto Y_{10}\!\mod \cL_{10}\,,
\quad Z_{10}\mapsto \lambda^{-1} Z_{10}\,.
\end{equation}
The normalization conditions say that $[Y_{10},X_{01}]=X_{10}+d_1X_{01} +d_2Y_{01}+ d_3Z_{01}+d_4 Y_{10}+d_5 Z_{10}$ for some complex-valued functions $d_1,\ldots,d_5:\cU\to \mathbb C$ and
$[Z_{10},X_{01}]\equiv Y_{10}+(\mu_1+i\mu_2)Y_{01}\mod \langle X_{01}\rangle_\mathbb C\oplus (\cL\otimes\mathbb C)$
for some real-valued functions $\mu_1,\mu_2:\cU\rightarrow \mathbb R$. 
Since $TU$ is integrable and it contains
$TU_\cD=\langle X,\cJ X,Y,Z,\cJ Z\rangle$ but not $\cJ Y$, we readily see that $\mu_1=1$, $\mu_2=0$.
The inclusion $[(TU_{\cD}^\cJ)_{01},(TU_\cD^\cJ)_{01}]\subset (TU_\cD^\cJ)_{01}$ also says that
$[Z_{01},X_{01}]\equiv 0\!\mod \langle X_{01}, Z_{01}\rangle_\mathbb C$. Summing and subtracting appropriately then yields
\begin{equation}
\label{eq:Lie-brackets-expandedII}
\begin{aligned}
{}[Z,X]&\equiv[\cJ Z,\cJ X]\equiv 2Y\!\mod TU_\cD^\cJ\,,\\
[\cJ Z,X]&\equiv [Z,\cJ X]\equiv 0\!\mod TU_\cD^\cJ\,.
\end{aligned}
\end{equation}
The bundles $\cL_{10}\oplus \cK_{01}$ and $\cK_{01}$ are integrable by Lemma \ref{lem:preliminarybrackets}, so that $[Z_{10},Y_{01}]\equiv \alpha Y_{01}\!\mod \cL\otimes\mathbb C$ for some $\alpha:\cU\to\mathbb C$ and $[Z_{10},Y_{10}]\equiv \alpha Y_{10}\!\mod \cL_{10}$ (to see that the coefficients are in fact both equal to $\alpha$, just sum the two Lie brackets and use that $[Z_{10},Y]$ is a section of $TU\otimes\mathbb C$).

We summarize the Lie brackets obtained so far, together with some direct consequences of the integrability of the Freeman bundles. 
Here ``$\star$'' refers to brackets that can be inferred from others by skew-symmetry and conjugation, and $[X_{10},X_{01}]=iR$ 
is an imaginary vector field nowhere tangent to $\cD$. 
The functions $c_1,\ldots, c_6; d_1,\ldots,d_{10};\sigma_1,\sigma_2,\sigma_3;\alpha$ are complex-valued.

\par\bigskip
\centerline{\rotatebox{0}{\footnotesize
$\begin{array}{||c|c|c|c|c|c|c||}\hline
[-,-] &
X_{10} & Y_{10} & Z_{10} & X_{01}& Y_{01}& Z_{01} \\
\hline\hline
X_{10} & 0 & 
\begin{gathered} c_1X_{10}+c_2Y_{10} \vphantom{\frac:.} \\ +c_3Z_{10}\end{gathered} & 
\star & \notin \cD & \star & \star \\
\hline
Y_{10} & \star & 0 & \star &
\begin{gathered}X_{10}+d_1X_{01}+d_2Y_{01} \vphantom{\frac:.} \\ + d_3Z_{01}+d_4 Y_{10}+d_5 Z_{10}\end{gathered} &
\begin{gathered}\sigma_1Y_{10}-\overline\sigma_1Y_{01}\\ +\sigma_2Z_{10}-\overline\sigma_2 Z_{01}\end{gathered} &
\star \\
\hline
Z_{10} & \star & \alpha Y_{10}+c_6Z_{10}  & 0 &
\begin{gathered}Y+d_6X_{01} \vphantom{\frac:.} \\ +d_7Z_{01}+d_8Z_{10}\end{gathered} &
\begin{gathered} \alpha Y_{01}+d_9 Z_{01}\\ +d_{10}Z_{10}\end{gathered}  &
\sigma_3 Z_{10}-\overline\sigma_3 Z_{01} \\
\hline
X_{01} & \star & \star & \star & 0 & \star & \star \\
\hline
Y_{01} & \star & \star & \star & \star & 0 & \star \\
\hline
Z_{01} & \star & \star & \star & c_4X_{01}+c_5Z_{01} & \star & 0 \\
\hline
\end{array}$}}
\vskip5pt\par
\centerline{\small\it Structure equations of the frame in the case $\opp{rk}(TU_\cK)=3$, $\opp{rk}(TU_\cL)=2$.}
\medskip

We aim to canonically constrain $X_{10}$, $Y_{10}$, and $Z_{10}$, at the same time reducing the freedom in the transformations \eqref{eq:gauge-freedom}. For this, we shall make \eqref{eq:gauge-freedom} explicit:
 \begin{equation}\label{eq:gauge-freedom-explicit}
X_{10}\mapsto \widetilde{X}_{10}=\lambda X_{10}+\beta Z_{10}\,,\quad Y_{10}\mapsto
\widetilde{Y}_{10}=Y_{10}+\gamma Z_{10}\,,
\quad Z_{10}\mapsto \widetilde{Z}_{10}=\lambda^{-1} Z_{10}\,,
 \end{equation}
where $\lambda:\cU\rightarrow \mathbb R_+$ and $\beta,\gamma:\cU\rightarrow\mathbb C$. 
A straightforward computation then yields
\begin{multline}
[\widetilde{Y}_{10},\widetilde{X}_{01}]=[Y_{10}+\gamma Z_{10},\lambda X_{01}+\overline\beta Z_{01}]
=\lambda\big(X_{10}+d_1X_{01}+d_2Y_{01}+ d_3Z_{01}+d_4 Y_{10}+d_5 Z_{10}\big)
\\
+\gamma\lambda\big(Y+d_6X_{01}+d_7Z_{01}+d_8Z_{10}\big)
-\overline{\beta}\big(\overline{\alpha} Y_{10}+\overline{d_9} Z_{10}+\overline{d_{10}}Z_{01}\big)\\
+\gamma\overline{\beta}\big(\sigma_3 Z_{10}-\overline\sigma_3 Z_{01}\big)+\big(Y_{10}(\lambda)+\gamma Z_{10}(\lambda)\big)X_{01}\\
+\big(Y_{10}(\overline\beta)+\gamma Z_{10}(\overline\beta)\big)Z_{01}
-\big(\lambda X_{01}(\gamma)+\overline\beta Z_{01}(\gamma)\big)Z_{10}
\end{multline}
so that $[\widetilde{Y}_{10},\widetilde{X}_{01}]\equiv \widetilde{X}_{10}
+\lambda\big( d_2+\gamma)\widetilde{Y}_{01}
+\big(\lambda d_4+\gamma\lambda-\overline{\beta}\overline{\alpha}) \widetilde{Y}_{10}
\!\mod \langle \widetilde X_{01}\rangle_\mathbb C\oplus (\cL\otimes\mathbb C)$.
Hence:
\begin{itemize}
	\item[$(i)$] We may choose $\gamma$ uniquely from the condition $d_2+\gamma=0$;
	\item[$(ii)$] Henceforth $X_{10}$, $Y_{10}$ and $Z_{10}$ are normalized by $d_2=0$, restricting the residual freedom to transformations \eqref{eq:gauge-freedom-explicit} with $\gamma=0$. 
\end{itemize}
In particular $Y$ and $\cJ Y$ are canonical, hence preserved by all infinitesimal CR-symmetries. 
Since $\cJ Y$ is canonical and transverse to any orbit $\mathcal O^G_q$, it follows that {\it $\gg$ is effectively represented on $\mathcal O^G_q$} (this is easier to see using coordinates that are adapted to the leaves of $\gg$ and $\cJ Y$). 
We check $[\widetilde{Z}_{10},\widetilde {Y}_{01}]=[\lambda^{-1}Z_{10},Y_{01}]\equiv\lambda^{-1}\alpha \widetilde Y_{01}\!\mod \cL\otimes\mathbb C$, whence:
\begin{itemize}
	\item[$(iii)$] If $\alpha\neq 0$ on $\cU$, we may choose $\lambda$ uniquely so that $\lambda^{-1}\alpha$ takes values in $S^1\subset\mathbb C$, normalizing the sections accordingly and restricting \eqref{eq:gauge-freedom-explicit} with $\gamma=0$ and $\lambda=1$;
	\item[$(iv)$] In this case, we finally fix the residual freedom $\beta$ by requiring that $d_4-\overline{\beta}\overline{\alpha}=0$, further normalizing the sections with the addition of $d_4=0$.
\end{itemize}
In other words, if $\alpha\neq 0$, we obtained a pair of canonical parallelisms: 
\begin{equation}
\label{eq:pairs-canononical-parallelisms}
\begin{aligned}
&\underbrace{\big(X,Y,Z,\cJ X,\cJ Z,R=\tfrac12[X,\cJ X]\big)}_{\text{Canonical parallelism on each fixed orbit}\;\mathcal O^G_q}\,,\\
&\underbrace{\big(X,Y,Z,\cJ X,\cJ Z,R=\tfrac12[X,\cJ X],\cJ Y\big)}_{\text{Canonical parallelism on}\;\cU}\,.
\end{aligned}
\end{equation}
The parallelism on $\cU$ says that each infinitesimal CR-symmetry $\xi\in\gg$ is determined by its value at a fixed point $q\in\cU$, proving, once more, that $\gg$ is effectively represented 
on each $\mathcal O^G_q$. Since $\gg$ preserves also the absolute parallelism on the orbit, we have $\dim\gg=\dim(\gg|_{\mathcal O^G_p})\leq \dim\mathcal O^G_q=6$. 
The generic case $\alpha\neq 0$ is settled:
\begin{proposition}
\label{prop:generic1!}
If the structure function $\alpha\neq0$, then $\dim\gg\leq 6$.
\end{proposition}

 If $\alpha=0$ identically, we may still enforce $(i)$ and $(ii)$, however this case is much more involved. Let us first focus on the $Y$-component of the brackets involving $Z_{01}, Y_{10}$ and $X_{01}$:
\begin{align*}
[Z_{01},[Y_{10},X_{01}]]&=[Z_{01},X_{10}+d_1X_{01}+ d_3Z_{01}+d_4 Y_{10}+d_5 Z_{10}]
\\&\equiv Y+Z_{01}(d_4)Y_{10}\!\mod TU^\cJ_\cD\otimes\mathbb C\,,\\
[[Z_{01},Y_{10}],X_{01}]&=[\overline{d_9} Z_{10}+\overline{d_{10}}Z_{01},X_{01}]\\
&\equiv \overline{d_9}Y\!\mod TU^\cJ_\cD\otimes\mathbb C\,,\\
[Y_{10},[Z_{01},X_{01}]]&=[Y_{10},c_4X_{01}+c_5Z_{01}]\\
&\equiv c_4d_4 Y_{10}\!\mod TU^\cJ_\cD\otimes\mathbb C\,,
\end{align*}
so that $d_9=1$ identically by the Jacobi identity. Under our conditions $d_2=0$, $\gamma=0$, $\alpha=0$, and $d_9=1$, 
the Lie bracket between $\widetilde{Y}_{10}$ and $\widetilde{X}_{01}$ reduces to
\begin{align*}
[\widetilde{Y}_{10},\widetilde{X}_{01}]
&=\lambda\big(X_{10}+d_1X_{01}+ d_3Z_{01}+d_4 Y_{10}+d_5 Z_{10}\big)
-\overline{\beta}\big(Z_{10}+\overline{d_{10}}Z_{01}\big)\\
&\;\;\;\,+\big(Y_{10}(\lambda)\big)X_{01}
+\big(Y_{10}(\overline\beta)\big)Z_{01}\\
&\equiv \widetilde X_{10}+\lambda d_4\widetilde Y_{10}+\lambda\big(\lambda d_5-2\mathfrak{Re}(\beta)\big)\widetilde Z_{10}\!\mod\langle\widetilde X_{01},\widetilde Z_{01}\rangle\,,
\end{align*}
and we may enforce the following alternative step to $(iii)$:
  \begin{itemize}
\item[$(iii')$] 
We may choose $\beta$ so that $2\mathfrak{Re}(\beta)=\lambda \mathfrak{Re}(d_5)$, 
henceforth we normalize our local sections additionally by $\mathfrak{Re}(d_5)=0$ and 
restrict \eqref{eq:gauge-freedom-explicit} by $\mathfrak{Re}(\beta)=0$. 
  \end{itemize} 
We note that $\beta$ is not determined uniquely, as its purely imaginary part remains unconstrained.
Thus, in what follows, we let $\beta=i\mathfrak{Im}(\beta)=i\beta'$.
The residual tranformations are
  \begin{equation*}\label{eq:gauge-freedom-explicit-real}
X_{10}\mapsto \widetilde{X}_{10}=\lambda X_{10}+i\beta' Z_{10}\,,\quad Y_{10}\mapsto\widetilde{Y}_{10}=Y_{10}\,,
\quad Z_{10}\mapsto \widetilde{Z}_{10}=\lambda^{-1} Z_{10}\,,
  \end{equation*}
for {\it real-valued} functions $\lambda:\cU\rightarrow \mathbb R_+$ and $\beta':\cU\rightarrow\mathbb R$. 
We now have $\alpha=0$, $d_2=0$, $d_9=1$, $d_5=i\mathfrak{Im}(d_5)=id_5'$. 
Further constraints are implied by the Jacobi identities involving $X_{10}$:
  \begin{equation}\label{eq:Jacobators!}
\begin{aligned}
\operatorname{Jac}(X_{10},Z_{10},Z_{01})&
\equiv\big(\overline{\sigma}_3-\overline{c_4}\big) Y \!\mod\langle X_{10},Z_{10},Z_{01}\rangle\,,\\
\operatorname{Jac}(X_{10},Y_{10},Z_{01})&\equiv\big(\overline{d_{10}}-\overline{\sigma}_1+c_1\big)Y_{01}\!\mod\langle X_{10},Y_{10},Z_{10},Z_{01}\rangle\,,\\
\operatorname{Jac}(X_{10},Y_{01},Z_{01})&\equiv\big(\overline{c_6}-\sigma_1-\overline{d_1}\big)Y_{10}+\big(\overline{d_6}-c_4\big)X_{01}\!\mod\langle X_{10},Y_{01},Z_{10},Z_{01}\rangle\,,\\
\operatorname{Jac}(X_{10},Y_{10},Y_{01})&\equiv\big(\overline{\sigma}_1+d_1+c_1\big)X_{01}\!\mod\langle X_{10},Y_{10},Y_{01},Z_{10},Z_{01}\rangle\,,\\
\operatorname{Jac}(X_{10},Y_{01},Z_{10})&\equiv\big(\overline{c_4}-d_6\big)X_{01}\!\mod\langle X_{10},Y_{01},Z_{10},Z_{01}\rangle\,,
\end{aligned}
  \end{equation}
so that their vanishing is equivalent to $c_4=\overline{d_6}=\sigma_3$, 
$\overline{\sigma_1}=c_1+\overline{d_{10}}=-\big(c_1+d_1\big)$, and $c_6=-c_1$. 
A somewhat long but straightforward computation of the Lie brackets of the vector fields of the new frame allows 
to read off the change on structure functions:
  \begin{align*}
\widetilde c_1 &=c_1-Y_{10}(\ln\lambda)\,,\quad\widetilde c_2=\lambda c_2\,,\\
\widetilde c_3 &=\lambda^2 c_3 -2i\beta'\lambda c_1
-i\lambda Y_{10}(\beta')+i\beta'\lambda Y_{10}(\ln\lambda)\,, \\
\widetilde c_5 &=\lambda c_5-iZ_{01}(\beta')+i\beta'\sigma_3+\lambda X_{01}(\ln\lambda)\,,\\
\widetilde d_3 &=\lambda^2d_3 -2i\beta'\lambda c_1
-i\lambda Y_{10}(\beta')+i\beta'\lambda Y_{10}(\ln\lambda)\,,\\
\widetilde d_4 &=\lambda d_4\,,\quad \widetilde d'_5=\lambda^2d'_5\,,\\
\widetilde d_7 &=\lambda d_7+2i\beta'\overline{\sigma}_3-iZ_{10}(\beta')
+i\beta' Z_{10}(\ln\lambda)\,,\\
\widetilde d_8 &= \lambda \big(d_8+X_{01}(\ln\lambda)\big)
-i\beta'\sigma_3-i\beta'Z_{01}(\ln\lambda)\,,\\
\widetilde \sigma_1&=\sigma_1\,,\quad \widetilde \sigma_2=\lambda\sigma_2\,,\quad
\widetilde \sigma_3=\lambda^{-1}\big(\sigma_3+Z_{01}(\ln\lambda)\big)\,.
  \end{align*}
We recognize the invariant $\sigma_1$ and the relative invariants $c_2$, $d_4$, $\sigma_2$, $c_5+\overline{d_7}-d_8$ 
of weight $+1$, as well as $d'_5$, $c_3-d_3$ of weight $+2$. 

Let us first assume that $\lambda$ can be fixed uniquely to $\lambda=1$ on $\cU$ via an algebraic normalization, 
that we do not need to specify. 
(For example, if one of the relative invariants is non-zero, we may require that it takes values in $S^1\subset\mathbb C$.) 
The structure functions $c_1$, $c_2$, $c_3-d_3$, $c_5+\overline{d_7}-d_8$, $d_4$, $d'_5$, $\sigma_2$, $\sigma_3$ 
are then also invariants, and the remaining ones change accordingly to
  \begin{align*}
\widetilde c_3 &= c_3-2i\beta'c_1 -iY_{10}(\beta')\,, \\
\widetilde c_5 &= c_5-iZ_{01}(\beta')+i\beta'\sigma_3\,,\\
\widetilde d_8 &= d_8 -i\beta'\sigma_3\,.
  \end{align*}
  
If $\sigma_3\neq 0$, we may fix $\beta'$ from the last equation, arrive at a pairs of 
canonical absolute parallelisms as in \eqref{eq:pairs-canononical-parallelisms}, and infer $\dim\gg\leq 6$. 

Otherwise $\sigma_3=0$ and hence $d_8$ is an invariant. We then compute the Jacobi identity
\begin{align*}
\operatorname{Jac}(X_{10},Z_{10},Z_{01})&
=[Z_{10},Y+\overline{d_7}Z_{10}+\overline{d_8}Z_{01}]-[Z_{01},\overline{c_5}Z_{10}]\\
&\equiv [Z_{10},Y]\mod \langle Z_{10}\rangle\quad \equiv Z_{01}\mod \langle Z_{10}\rangle\neq0
\end{align*}
where we used $c_4=\overline{d_6}=\sigma_3=0$ and that $\overline{d_8}$ is an invariant (hence it is constant on the orbits and $Z_{10}(\overline{d_8})=0$). We get a contradiction, this case cannot happen. We then proved:
\begin{proposition}
\label{prop:fixed-uniquely}
If the structure function $\alpha=0$, and $\lambda:\cU\rightarrow \mathbb R_+$ can be fixed uniquely to $\lambda= 1$ on $\cU$, then $\dim\gg\leq 6$.
\end{proposition}

Assume now that $\lambda$ cannot be fixed to $\lambda= 1$ via any algebraic normalization:
in particular, all the relative invariants $c_2$, $d_4$, $\sigma_2$, $c_5+\overline{d_7}-d_8$, $d'_5$, $c_3-d_3$ vanish. 
The relevant changes of structure functions are more easily expressed in terms of $\lambda$ and 
$\widehat\beta=\lambda^{-1}\beta'$ as follows
 \begin{align*}
\widetilde c_1&=c_1-Y_{10}(\ln\lambda)\,,\\
\widetilde c_3
&=\lambda^2 \big(c_3-2i\widehat\beta c_1-iY_{10}(\widehat\beta)\big)\,, \\
\widetilde d_7 
&=\lambda \big(d_7+2i\widehat\beta\overline{\sigma}_3-iZ_{10}(\widehat\beta)\big)\,,\\
\widetilde d_8&=
\lambda \big(d_8+X_{01}(\ln\lambda)
-i\widehat\beta\sigma_3-i\widehat\beta Z_{01}(\ln\lambda)\big)\,,\\
\widetilde \sigma_1&=\sigma_1\,,\quad
\widetilde \sigma_3=\lambda^{-1}\big(\sigma_3+Z_{01}(\ln\lambda)\big)\,.
 \end{align*}
and the Lie brackets drastically simplify to the following form:
\par
{\footnotesize
 \begin{equation*}\label{Table:tab-simplified}
\begin{array}{||c|c|c|c|c|c|c||}\hline
[-,-] & X_{10} & Y_{10} & Z_{10} & X_{01} & Y_{01} & Z_{01} \\
\hline\hline
X_{10} & 0 & c_1X_{10}+c_3Z_{10} & \star & \notin \cD & \star & \star \\
\hline
Y_{10} & \star & 0 & \star &
X_{10}-(\overline{\sigma}_1+c_1)X_{01}+ c_3Z_{01} & \sigma_1Y_{10}-\overline\sigma_1Y_{01} & \star \\
\hline
Z_{10} & \star & -c_1Z_{10} & 0 & Y+\overline{\sigma}_3X_{01} +d_7Z_{01}+d_8Z_{10} &
Z_{01} +(\sigma_1-\overline{c_1})Z_{10} & \sigma_3 Z_{10}-\overline\sigma_3 Z_{01} \\
\hline
X_{01} & \star & \star & \star & 0 & \star & \star \\
\hline
Y_{01} & \star & \star & \star & \star & 0 & \star \\
\hline
Z_{01} & \star & \star & \star & \vphantom{\frac{A^b}A}
\sigma_3 X_{01}+(d_8-\overline{d_7})Z_{01} & \star & 0 \\
\hline
\end{array}\;\;
 \end{equation*}}
\vskip0pt\par
\centerline{\small\it Strucrture equations of the frame in the subcase $\alpha=0$ and $\lambda$ cannot be fixed.}
\medskip

It is only from now on that applying Tanaka's algebraic prolongation pays off!
In fact $\gg$ is represented as infinitesimal symmetries of a $7$-dimensional manifold endowed with
a collection of graded frames, defined up to $\lambda:\cU\to \mathbb R_+$ and $\beta':\cU\rightarrow\mathbb R$, 
and the symbol algebra is a $7$-dimensional metabelian Lie algebra with $2$-dimensional reduced structure algebra:
  \begin{equation*}\label{eq:symbol-with-structure}
\begin{aligned}
\mathfrak f_0(q)&\cong\operatorname{span}\big(\Lambda, B\big) \,,\\
\gm_{-1}(q)&\cong \cD|_q=\big(TU_\cD^\cJ|_q\oplus \mathbb R Y|_q\big)\oplus  \mathbb R \cJ Y|_q\,,\\
\gm_{-2}(q)&\cong T\cU|_q/\cD|_q=\mathbb R R|_q\,.
\end{aligned}
  \end{equation*}
Here $TU_\cD^\cJ=\langle X,\cJ X,Z,\cJ Z \rangle$ and we tacitly identified $R|_q$ with its class modulo $\cD|_q$. 

The Lie brackets of $\gf_{\leq 0}(q):=\gm(q)\rtimes\gf_0(q)$ are more conveniently expressed w.r.t.\ the natural basis in the complexified picture with omitted evaluation at $q$ for simplicity of notation:
\begin{equation*}
\begin{aligned}
{}[X_{10},X_{01}]&=iR\,,\,\,[B,X_{10}]=iZ_{10}\,,\,\,[B,X_{01}]=-iZ_{01}\,,\,\,[\Lambda,B]=-2B\,,\,\,[\Lambda,R]=2R\,, \\
[\Lambda,X_{10}]&=X_{10}\,,\quad [\Lambda,X_{01}]=X_{01}\,,\quad[\Lambda,Z_{10}]=-Z_{10}\,,\quad[\Lambda,Z_{01}]=-Z_{01}\,.
\end{aligned}
\end{equation*}

Now $\gf_{\leq 0}(q)=\gf'_{\leq 0}(q)\oplus\gf''_{\leq 0}(q)$ is the direct sum of two ideals 
  \begin{align*}
\gf'_{\leq 0}(q)&=\gf'_{-2}(q)\oplus \gf'_{-1}(q)\oplus \gf'_{0}(q)\qquad\qquad \gf''_{\leq 0}(q)=\gf''_{-1}(q)\\
&=\gm_{-2}(q)\oplus TU_\cD^\cJ|_q\oplus \gf_0(q)\qquad\qquad\quad\;\;\,\,=\mathrm{span}\big(Y|_q,\cJ Y|_q\big)\,.
  \end{align*}
It is known that the Tanaka prolongation of $\gf_{\leq 0}(q)$ is the direct sum of the prolongations of $\gf_{\leq 0}'(q)$ and 
$\gf_{\leq 0}''(q)$ as ideals. Since $\gf_{\leq 0}''(q)$ has trivial structure algebra $\gf_{0}''(q)=0$ by construction, 
it coincides with its prolongation. We now turn to $\gf_{\leq 0}'(q)$. 

Let $R|_q,X_{10}|_q,X_{01}|_q,Z_{10}|_q,Z_{01}|_q,\Lambda,B$ be the basis of the complexification of $\gf_{\leq 0}'(q)$ 
and denote the dual basis by $\rho,\xi_{10},\xi_{01},\zeta_{10},\zeta_{01},\lambda,\beta$. The Spencer operator
  $$
\delta:\gf'_{-1}(q)^*\otimes \gf'_0(q)\,\bigoplus\,\gf'_{-2}(q)^*\otimes\gf'_{-1}(q)\longrightarrow 
\Lambda^2 (\gf_{<0}'(q))^*\otimes \gf'_{<0}(q)
 $$
is defined as the  component of degree $1$ of the differential of the Chevalley-Eilenberg complex of the 
Lie algebra $\gf_{<0}'(q)$ with values in $\gf_{\leq 0}'(q)$ and it is determined by the following expressions
  \begin{equation}\label{eq:Spencer-differential-component}
\begin{aligned}
\delta\rho&=-i\xi_{10}\wedge\xi_{01}\,,\quad\delta\xi_{10}=\delta\xi_{01}=\delta\zeta_{10}=\delta\zeta_{01}=0\,,\\
\delta\Lambda&=-\xi_{10}\otimes X_{10}-\xi_{01}\otimes X_{01}+\zeta_{10}\otimes Z_{10}+\zeta_{01}\otimes Z_{01}-2\rho\otimes R\,,\\
\delta B&=-i\xi_{10}\otimes Z_{10}+i\xi_{01}\otimes Z_{01}\,,\\
\delta X_{10}&=-i\xi_{01}\otimes R\,,\quad\delta X_{01}=i\xi_{10}\otimes R\,,\\
\delta Z_{10}&=\delta Z_{01}=\delta R=0\,.
\end{aligned}
  \end{equation}
A Spencer $1$-cochain
  \begin{multline*}
\big(a_1\xi_{10}+a_2\xi_{01}+a_3\zeta_{10}+a_4\zeta_{01}\big)\otimes\Lambda+
\big(b_1\xi_{10}+b_2\xi_{01}+b_3\zeta_{10}+b_4\zeta_{01}\big)\otimes B\\
+\rho\otimes\big(c_1X_{10}+c_{2}X_{01}+c_3Z_{10}+c_4Z_{01}\big)
  \end{multline*}
is $\delta$-closed precisely when $a_1=a_2=a_3=a_4=b_3=b_4=c_1=c_2=0$ and $b_1=-c_4$, $b_2=-c_3$. 
In other words, 
$\gf'_1(q)=\mathrm{span}\big(
B_{10}:=\xi_{10}\otimes B-\rho\otimes Z_{01},
B_{01}:=\xi_{01}\otimes B-\rho\otimes Z_{10}
\big)$.


The element $B\in\gf'_0(q)$ acts trivially on $\gf'_1(q)$, while $\Lambda$ acts as a scaling by $-3$. Therefore, 
the semisimple element $S:=-\Lambda$ has spectrum $(-2; -1,-1,+1,+1;0,+2;+3,+3)$ in the eigenbasis
$R|_q$; $X_{10}|_q$, $X_{01}|_q$, $Z_{10}|_q$, $Z_{01}|_q$;  $S$, $B$; $B_{10}$, $B_{01}$ of $\gf'$.

A longer but still straightforward computation with the Spencer operator of higher order implies $\gf'_2(q)=0$;
this verification is included into the \textsc{Maple} supplement accompanying the arXiv posting of the article. 
Hence the Tanaka prolongation of $\gf'_{\leq 0}(q)$ is the $9$-dimensional graded Lie algebra $\gf'_{\leq 0}(q)\oplus\gf'_1(q)$. 
The prolongation of $\gf_{\leq 0}(q)$ is therefore $11$-dimensional,
but since $\gg$ is effectively represented on any fixed orbit, it is sufficient to consider the prolongation of 
$\gf_{\leq 0}'(q)\oplus\mathbb R Y|_q$, which is the $10$-dimensional Lie algebra 
$\bar\gf=\big(\gf'_{<0}(q)\oplus \mathbb R Y|_q\big)\oplus\gf_0'(q)\oplus\gf'_1(q)$ majorizing $\opp{gr}(\gg)$. 
Consequently we get the a-priori bound $\dim\gg\leq 10$.

 \begin{proposition}\label{P6a}
If the structure function $\alpha=0$, and $\lambda:\cU\rightarrow \mathbb R_+$ 
cannot be fixed to $\lambda= 1$ via algebraic normalizations, then $\dim\gg\leq 6$.
  \end{proposition}

In order to prove this we have to exclude all the cases where $\dim\gg\geq 7$.
We are thus led to investigate filtered deformations $\gg$ of the following graded subalgebras of
$\bar\gf$ that include the negative part (since $\gg\cong\gg|_{\cO^G_q}$ is transitive on the orbit),
and have dimension at least $7$:
  \begin{itemize}
\item[$(i)$] $\big(\gf_{<0}'(q)\oplus\R Y|_q\big)\oplus\gf_0'(q)\oplus\gf'_1(q)$,
\item[$(ii)$] $\big(\gf_{<0}'(q)\oplus\R Y|_q\big)\oplus\R B\oplus\gf'_1(q)$,
\item[$(iii)$] $\big(\gf_{<0}'(q)\oplus\R Y|_q\big)\oplus\gf_0'(q)\oplus\ell$, where $\ell$ is a line in $\gf'_1(q)$,
\item[$(iv)$] $\big(\gf_{< 0}'(q)\oplus\R Y|_q\big)\oplus\R B\oplus\ell$, where $\ell$ is a line in $\gf'_1(q)$,
\item[$(v)$] $\big(\gf_{<0}'(q)\oplus\R Y|_q\big)\oplus\gf_0'(q)$,
\item[$(vi)$] $\big(\gf_{<0}'(q)\oplus\R Y|_q\big)\oplus\R B$,
\item[$(vii)$] $\big(\gf_{<0}'(q)\oplus\R Y|_q\big)\oplus\R\Lambda$,
  \end{itemize}
In $(i)$-$(iv)$ we use the fact that the commutator of a non-zero element in degree $1$ with 
$\gf'_{-1}(q)=TU_\cD^\cJ|_q$ gives $B$. 
In $(v)$-$(vii)$ we use the fact that a non-trivial linear combination of the elements $\Lambda$ and $B$ 
is either equal to $B$ or $\opp{Ad}\bigl(\exp\gf'_0(q)\bigr)$-conjugate to $\Lambda$.

Several general observations in the study of the filtered deformations of the above graded algebras $(i)$-$(vii)$ are in order:
\begin{enumerate}
	\item[O1] 
The Lie bracket is compatible with the filtration induced by the graded structure. Explicitly, it has the form $[-,-]=[-,-]_o+\cdots$, where $[-,-]_o$ is the graded bracket and $\cdots$ are terms of overall positive degree (i.e, the degree of the output is higher than the sum of the degrees of the two inputs).
	\item[O2] 
If the semisimple element $S$ enters the graded subalgebra, we may redefine the elements of the graded basis so that $[S,-]=[S,-]_o$. In other words, $S$ can be assumed to have the same spectral properties in the filtered algebra.
	\item[O3] 
Many of the structure equations coming from the brackets of vector fields summarized in the Table 
after Proposition \ref{prop:fixed-uniquely} can be used thanks to Lemma \ref{lem:from-left-to-right} and 
Remark \ref{rem:change-sign}, upon  rewriting the equations in terms of vector fields tangent to the orbit.
	
For instance, the equations coming from higher-order Levi forms and integrability of bundles are clearly tensorial, 
while the equations $\alpha=0$ and $d_9=1$ fall within the scope of (ii) Lemma \ref{lem:from-left-to-right} 
(the section $Y$ is canonical, the bundle $\langle Z_{10}\rangle$ is $G$-stable by \eqref{eq:gauge-freedom-explicit} and one 
may consider the map $\langle Z_{10}\rangle\to TU_\cD\otimes\mathbb C/\langle Z_{10}\rangle$ induced by Lie bracket).
	
The equations that we don't take into account are the direct consequences of the Jacobi identities \eqref{eq:Jacobators!} 
and the vanishing of the relative invariants $c_5+\overline{d_7}-d_8$, $c_3-d_3$.
(The Jacobi identities will be enforced later at the level of the symmetry algebra. 
The vanishing of the relative invariants actually do fall within the scope of Lemma \ref{lem:from-left-to-right}, 
but it is cumbersome to show it. 
In fact, an elaboration of Tanaka theory implies that one can impose all
the constraints obtained so far, but we decided not to discuss it here.) 
\end{enumerate}

We indicate the final Lie brackets in the following Table for the reader's convenience:

\par\medskip
\centerline{\rotatebox{0}{\footnotesize
$\begin{array}{||c|c|c|c|c|c||}\hline
[-,-] &
X_{10} & Y & Z_{10} & X_{01} & Z_{01} \\
\hline\hline
X_{10} & 0 & \star & \star & \notin TU_\cD & \star \\
\hline
Y & \star & 0 & \star &
X_{10}+(d_1-\overline c_1)X_{01}+ (d_3-\overline c_3)Z_{01} & \star \\
\hline
Z_{10} & \star &  Z_{01}+(c_6+d_{10})Z_{10}  &  0 & 
Y+d_6X_{01}+d_7Z_{01}+d_8Z_{10} & \sigma_3 Z_{10}-\overline\sigma_3 Z_{01} \\
\hline
X_{01} & \star & \star & \star & 0 & \star \\
\hline
Z_{01} & \star & \star & \star & c_4X_{01}+c_5Z_{01} & 0 \\
\hline
\end{array}$
}}
\vskip5pt\par
\centerline{\small\it Structure equations of the sub-frame $X$, $Y$, $Z$, $\cJ X$, $\cJ Z$.}

\begin{remark}
Observation (O2) needs a clarification. For instance $[S,X_{10}]_o=-X_{10}$ and, by (O1), we may write
$[S,X_{10}]=-X_{10}+k_1S+k_2B+k_3B_{10}+k_{4}B_{01}$ for some $k_1,k_2,k_3,k_4\in\mathbb C$. The general redefinition $\widetilde X_{10}:=X_{10}+\alpha_1S+\alpha_2 B+\alpha_3B_{10}+\alpha_{4}B_{01}$ is still compatible with the filtration and 
$
[S,\widetilde X_{10}]
=-X_{10}+k_1S+\big(k_2+2\alpha_2\big)B
+\big(k_3+
\alpha_2 r_1+3\alpha_3\big)B_{10}
+\big(k_{4}+
\alpha_2 r_{2}+3\alpha_4\big)B_{01}
$,
with $[S,B]=2B+r_1B_{10}+r_{2}B_{01}$ for some $r_1,r_2\in\mathbb C$, $[S,B_{10}]=3B_{10}$, $[S,B_{01}]=3B_{01}$ by (O1). Choosing
$\alpha_1=-k_1$, $3\alpha_2=-k_2$, $4\alpha_3=-k_3-
\alpha_2 r_1$, $4\alpha_4=-k_{4}-\alpha_2 r_{2}$ yields $[S,\widetilde X_{10}]=-\widetilde X_{10}$.
Similarly $[S,\widetilde X_{01}]=-\widetilde X_{01}$, and setting $\widetilde R:=-i[\widetilde X_{10},\widetilde X_{01}]$ also enforces $[S,\widetilde R]=-2\widetilde R$.

The crucial point here is that the eigenvalue of the graded action of $S$ on $X_{10}$
is different from those on the higher filtration terms $S$, $B$, $B_{10}$, and $B_{01}$ entering the redefinition. 
This is the case for all the basis elements except $Y$, in which case we may only enforce $[S,Y]=kS$ for some $k\in\mathbb C$, after redefinitions. However:
\end{remark}

 \begin{lemma}
We have $[S,Y]=0$ also in the filtered algebra.
 \end{lemma}

 \begin{proof}
Working modulo terms in the stabilizer (i.e., modulo non-negative terms), we write $[Y,Z_{01}]\equiv -Z_{10}-tZ_{01}$  
for some $t\in\mathbb C$ thanks to (O3), whence $[S,[Y,Z_{01}]]\equiv-Z_{10}-tZ_{01}$ too. But
$[[S,Y],Z_{01}]+[Y,[S,Z_{01}]]=k[S,Z_{01}]+[Y,Z_{01}]\equiv  -Z_{10}-\big(t-k\big)Z_{01}$
and the Jacobi identity immediately implies $k=0$.
 \end{proof}

In summary, not only the filtration is preserved by (O1) but (O2) holds: the brackets of the filtered Lie algebra are compatible 
with the grading given by $S$, whenever present, simplifying the commutation relations. For instance, the stabilizer is always rigid, 
either by $S$-grading in cases $(i)$, $(iii)$, $(v)$ and $(vii)$, or by trivial dimensional reasons in cases $(ii)$, $(iv)$, $(vi)$.

We thus checked that the Jacobi identities rule out all the possible filtered deformations.
The computation is straightforward linear algebra, done using the symbolic package \textsc{Maple} -- 
it can be found in the supplement accompanying the arXiv posting of the article. 
Let us note that although one expects quadratic relations in the Jacobi constraints on the deformation parameters,
the ideal of relations has a linear part, resolving which we get another linear part, and so forth until a non-ambiguous
contradiction of the type $1=0$ arises.

This proves Proposition \ref{P6a} and finishes the subcase \S\ref{subsubsec:352}.

\subsubsection{The case $\opp{rk}(TU_\cK)=3$, $\opp{rk}(TU_\cL)=1$}\label{subsubsec:353}

Let us recall that $\opp{rk} (TU_\cD)=5$, $\opp{rk} (TU_\cD^\cJ)=4$, $\opp{rk}(TU_\cK)=3$, and $\opp{rk}(TU_\cK^\cJ)=2$.
Since $\opp{rk}(TU_\cL)=1$, then $TU^\cJ_\cL=0$, i.e., $TU_\cD^\cJ$ and $\cL$ are transversal, and $\cD=TU^\cJ_\cD\oplus \cL$. Now
$TU+\cK=TU+\cL=T\cU$ is integrable, thus we may integrate the distributions using adapted coordinates $(u^1,\ldots,u^5,z,v)$ so that
\begin{equation}
\label{eq:distribution-aligned-II}
\begin{aligned}
TU&=\langle \partial_{u^1},\ldots,\partial_{u^5},\partial_z\rangle\,,\\
\cK&=\langle \partial_{u^4},\partial_{u^5},\partial_z,\partial_v\rangle\,,\\
\cL&=\langle\partial_{z},\partial_v\rangle\,,
\end{aligned}
\end{equation}
and the Lie algebra $\gg$  consists of vector fields
$\xi=\sum_{i=1}^5\xi^i\partial_{u^i}+\xi^6\partial_z$, with $\xi^i$ depending only on 
$u^1,u^2,u^3$ for $i=1,2,3$ (since $\cK$ is $\gg$-stable) and $u^1,\ldots,u^5$ for $i=4,5$ (since $\cL$ is $\gg$-stable).
The projection 
to the leaf space of $\cL$
is injective on symmetries by $(v)$ of Proposition \ref{prop:inducedcomplexstructure}. In other words $\gg\cong\widetilde{\pi}_\star(\gg)$ and, decomposing
any infinitesimal CR-symmetry into $\xi=\xi'+\xi''$ with $\xi':=\sum_{i=1}^5\xi^i\partial_{u^i}$, $\xi'':=\xi^6\partial_z$, we have that $\xi'$ fully determines $\xi$. In particular, $\gg$ is {\it effectively} represented on any fixed orbit $\mathcal O^G_q=\{(u,z,v)\in\cU\mid v=v_q\}$.

By Lemma \ref{lem:simple-but-useful}, we may consider local sections $X_{10}$, $Y_{10}$, $Z_{10}$ of $\cD_{10}$, $\cK_{10}$, $\cL_{10}$, respectively, as in \S\ref{subsubsec:352}. They are normalized sections and they are defined up to transformations depending on 
functions $\lambda:\cU\rightarrow \mathbb R_+$ and $e^{i\varphi}:\cU\rightarrow S^1$.
Now $TU_\cL\subset TU_\cL\otimes\mathbb C\cong \cL_{10}$ is a canonical real line subbundle of $\cL_{10}$, so (ii) of Lemma \ref{lem:simple-but-useful} applies: we may take $Z$ a non-zero sections of $TU_\cL$ and set $e^{i\varphi}=1$. Since $\cD/\cK\cong TU^\cJ_\cD/TU^\cJ_\cK$ and $\cK/\cL\cong TU^\cJ_\cK $ (by
$TU^\cJ_\cL=0$),
we are also free to ask for $X$ to be a section of $TU^\cJ_\cD$, which is not in $\cK$ at all points of $\cU$, and for $Y$ to be a non-zero section of $TU_\cK^\cJ$.
In summary:
\begin{equation}
\label{eq:summary-for-preliminary-frames-secondsubcase}
\begin{aligned}
TU_\cD&=TU^\cJ_\cD\oplus\langle Z\rangle\,,\\
TU_\cD^\cJ&=\langle X,\cJ X\rangle\oplus TU_\cK^\cJ \,,\\
TU_\cK^\cJ&=\langle Y,\cJ Y\rangle\,,
\end{aligned}
\end{equation}
with the normalized sections uniquely defined up transformations of the form
\begin{equation}
\label{eq:gauge-freedom-secondsubcase}
X_{10}\mapsto \lambda X_{10}\mod (TU^\cJ_\cK)_{10}\,,\quad Y_{10}\mapsto Y_{10}\,,
\quad Z_{10}\mapsto \lambda^{-1} Z_{10}\,,
\end{equation}
for some function $\lambda:\cU\rightarrow \mathbb R_+$. In particular, the sections $Y$ and $\cJ Y$ are both preserved by $\gg$.
The normalization conditions say that
$[Y_{10},X_{01}]=X_{10}+d_1X_{01} +d_2Y_{01}+d_3 Y_{10}+d_4(Z_{10}+ Z_{01})$  (here we also used that $TU$ is integrable) and 
$[Z_{10},X_{01}]= Y_{10}+ d_5 X_{01}   + d_6 Y_{01}+ d_7  Z_{01}+d_8  Z_{10}$,
for some complex-valued functions $d_1,\ldots,d_8:\cU\to \mathbb C$. Furthermore 
$[X_{10},Y_{10}]= c_1X_{10}+c_2 Y_{10}$ for some $c_1,c_2:\cU\to \mathbb C$, using the general inclusion
$[(TU_{\cD}^\cJ)_{10},(TU_\cD^\cJ)_{10}]\subset (TU_\cD^\cJ)_{10}$.

We summarize the Lie brackets obtained so far, with some consequences of the integrability of $TU$ and the Freeman bundles (including that $\cL_{10}\oplus \cK_{01}$ and $TU_\cK$ are integrable). 
	
\par\bigskip
\centerline{\rotatebox{0}{\footnotesize
$\begin{array}{||c|c|c|c|c|c|c||}\hline
[-,-] & X_{10} & Y_{10} & Z_{10} & X_{01} & Y_{01} & Z_{01} \\
\hline\hline
X_{10} & 0 & 
\begin{gathered} c_1X_{10}+c_2Y_{10} \vphantom{\frac:.}\end{gathered} & 
\star & \notin D & \star & \star \\
\hline
Y_{10} & \star & 0 & \star &
\begin{gathered}X_{10}+d_1X_{01} +d_2Y_{01}\vphantom{\frac:.}\\ +d_3 Y_{10}+d_4Z\end{gathered} &
\begin{gathered}\sigma_1Y_{10}-\overline\sigma_1Y_{01}\\ +irZ \end{gathered} &
\star \\
\hline
Z_{10} & \star & c_6 Y_{10}+c_7Z_{10}  & 0 &
\begin{gathered} Y_{10}+ d_5 X_{01}   + d_6 Y_{01}\vphantom{\frac:.}\\+ d_7  Z_{01}+d_8  Z_{10}\end{gathered} &
\begin{gathered} d_9 Y_{01}+d_{10} Z_{01}\\ +d_{11}Z_{10}\end{gathered}  &
\sigma_2 Z_{10}-\overline\sigma_2 Z_{01} \\
\hline
X_{01} & \star & \star & \star & 0 & \star & \star \\
\hline
Y_{01} & \star & \star & \star & \star & 0 & \star \\
\hline
Z_{01} & \star & \star & \star & c_3X_{01}+ c_4Y_{01}+c_5Z_{01} & \star & 0 \\
\hline
\end{array}$
}}
\vskip5pt\par
\centerline{\small\it Structure equations of the frame in the case $\opp{rk}(TU_\cK)=3$, $\opp{rk}(TU_\cL)=1$.}
\medskip

\noindent
The functions $c_1,\ldots c_7, d_1,\ldots,d_{11},\sigma_1,\sigma_2$ are complex-valued, whereas $r$ is real, 
since $Z$ is tangent to the orbit while $\cJ Z$ not. 
This fact also yields the additional relations $d_8=c_5+d_7$ (bracketing $Z$ with $X_{01}$) and 
$d_{11}=\overline{c_7}+d_{10}$ (bracketing $Z$ with $Y_{10}$).

We aim to canonically constrain the sections $X_{10}$, $Y_{10}$ and $Z_{10}$, at the same time reducing the freedom in the transformations \eqref{eq:gauge-freedom-secondsubcase}, given by
  \begin{equation}\label{eq:gauge-freedom-explicit-secondsubcase}
X_{10}\mapsto \widetilde{X}_{10}=\lambda X_{10}+\beta Y_{10}\,,\quad Y_{10}\mapsto\widetilde{Y}_{10}=Y_{10}\,,
\quad Z_{10}\mapsto \widetilde{Z}_{10}=\lambda^{-1} Z_{10}\,,
  \end{equation}
for functions $\lambda:\cU\rightarrow \mathbb R_+$ and $\beta:\cU\rightarrow\mathbb C$. 
The change on structure functions is as follows:
  \begin{equation*}\label{eq:structure-function-explicit-secondsubcase}
\begin{aligned}
\widetilde d_1&=d_1+Y_{10}(\ln\lambda)\,,\quad
\widetilde d_2=\lambda d_2-\overline\beta \big(d_1+ Y_{10}(\ln\lambda)+\overline\sigma_1\big)
+Y_{10}(\overline\beta)\,,\\
\widetilde d_3&=\lambda d_3+\overline{\beta}\sigma_1-\beta\,,\quad
\widetilde d_4=\lambda\big(\lambda d_4+\overline{\beta} r\big)\,,\quad
\widetilde d_5=\lambda^{-1}\big(d_5+Z_{10}(\ln\lambda)\big)\,,\\
\widetilde d_6&=\lambda^{-1}\bigg(\lambda d_6
-\overline\beta \big(d_5+Z_{10}(\ln\lambda)-d_9\big)+Z_{10}(\overline\beta)\bigg)\,,\quad
\widetilde d_7=\lambda d_7+\overline\beta d_{10}\,,\\
\widetilde d_8&=\lambda\big(d_8+X_{01}(\ln\lambda)\big)+\overline\beta \big(d_{11}+Y_{01}(\ln\lambda)\big)\,,\quad
\widetilde d_9=\lambda^{-1}d_9\,,\quad
\widetilde d_{10}=d_{10}\,,\\
\widetilde d_{11}&=d_{11}+Y_{01}(\ln\lambda)\,,\quad
\widetilde c_1=c_1-Y_{10}(\ln\lambda)\,,\quad
\widetilde c_2=\lambda c_2-\beta\big(c_1-Y_{10}(\ln\lambda)\big)-Y_{10}(\beta)\,,\\
\widetilde c_3&=\lambda^{-1}\big(c_3+Z_{01}(\ln\lambda)\big)\,,\quad
\widetilde c_4=\lambda^{-1}\bigg(\lambda c_4+Z_{01}(\overline\beta)-\overline\beta\big(c_3+Z_{01}(\ln\lambda)-\overline{c_6}\big)\bigg)\,,\\
\widetilde c_5&=\lambda\big(c_5+X_{01}(\ln\lambda)\big)+\overline\beta\big( \overline{c_7}+ Y_{01}(\ln\lambda)\big)\,,\quad
\widetilde c_6=\lambda^{-1}c_6\,,\quad
\widetilde c_7=c_7+Y_{10}(\ln\lambda)\,,\\
\widetilde \sigma_1&=\sigma_1\,,\quad
\widetilde \sigma_2=\lambda^{-1}\big(\sigma_2+Z_{01}(\ln\lambda)\big)\,,\quad
\widetilde r=\lambda r\,.
\end{aligned}
  \end{equation*}
We recognize the absolute invariants $d_1-\overline{d_{11}}$, $d_1+c_1$, $c_1+c_7$, $d_{10}$, $\sigma_1$, and
the relative invariants $d_5-\overline{c_3}$, $c_3-\sigma_2$, $d_9$, $c_6$ of weight $-1$, as well as $r$ of weight $+1$. 

There are three generic cases that allow us to restrict \eqref{eq:gauge-freedom-explicit-secondsubcase} with
the constraint $\beta=0$:
  \begin{itemize}
	\item[($i'$)] 
If $TU^\cJ_\cK$ is not integrable (equivalently, if $r\neq 0$), we may additionally normalize sections by the condition $d_4=0$;
	\item[($i''$)] 
If $\cL_{10}\oplus (TU^\cJ_\cK)_{01}=\langle Z_{10}, Y_{01}\rangle$ is not integrable (equivalently, if $d_{10}\neq 0$), 
we may additionally normalize sections by the condition $d_7=0$;
	\item[($i'''$)] 
If $\sigma_1$ does not take values in $S^1\subset\mathbb C$, we may additionally normalize sections by the condition $d_3=0$.
  \end{itemize}
If any of these conditions hold, we reduce to the case $\beta=0$ and then $d_2$, $d_3$, $d_7$,  $c_2$ are additional 
relative invariants of weight $+1$ as well as $d_4$ is a relative invariant of weight $+2$. 
If at least one of these overall ten relative invariants is non-zero, then: 
  \begin{itemize}
	\item[($ii$)] 
We further normalize sections by requiring that the above relative invariants take values in $S^1\subset\mathbb C$, 
finally restricting \eqref{eq:gauge-freedom-explicit-secondsubcase} with $\beta=0$ and $\lambda=1$.
  \end{itemize}
In this case, we obtain a pair of canonical parallelisms
  \begin{equation}\label{eq:pairs-canononical-parallelisms-II}
\begin{aligned}
&\underbrace{\big(X,Y,Z,\cJ X,\cJ Y,R=\tfrac12[X,\cJ X]\big)}_{\text{Canonical parallelism on each fixed orbit}\;\mathcal O^G_q}\,,\\
&\underbrace{\big(X,Y,Z,\cJ X,\cJ Z,R=\tfrac12[X,\cJ X],\cJ Z\big)}_{\text{Canonical parallelism on}\;\cU}\,,
\end{aligned}
  \end{equation}
and conclude as in Proposition \ref{prop:generic1!}:

 \begin{proposition}\label{prop:generic2!}
In the generic case, given by ($i$)-($ii$), we have $\dim\gg\leq\dim\mathcal O^G_q=6$.
 \end{proposition}

If ($ii$) cannot be enforced, then $d_2=d_3=d_4=d_7=d_9=c_2=c_6=r=0$ and
$d_5=\overline{c_3}=\overline\sigma_2$, and still 
$X_{10}\mapsto \widetilde{X}_{10}=\lambda X_{10}$, $Y_{10}\mapsto\widetilde{Y}_{10}=Y_{10}$, 
$Z_{10}\mapsto \widetilde{Z}_{10}=\lambda^{-1} Z_{10}$, for some $\lambda:\cU\to \mathbb R_+$. 
In other words, $\gg$ is represented as infinitesimal symmetries of a $7$-dimensional manifold endowed 
with a collection of graded frames defined up to $\lambda:\mathfrak U\to \mathbb R_+$.

We work at a fixed point $p\in\cU$ that is regular for the filtered structure as in \S\ref{sec:2.3} 
and, taking into account the reduction of structure algebra, write
  \begin{equation*}\label{eq:symbol-with-structure-second-occurence}
\begin{aligned}
\gg_0(p)&\subset\gf_0(p)\cong\mathbb R \Lambda\\
\gg_{-1}(p)&\cong TU_\cD|_p=\operatorname{span}\big(X|_p, \cJ X|_p, Y|_p, \cJ Y|_p, Z|_p\big)\subset \gm_{-1}(p)\cong \cD|_p\,,\\
\gg_{-2}(p)&=\gm_{-2}(p)\cong T \cU|_p/\cD|_p\cong\mathbb R R|_p\,,
\end{aligned}
  \end{equation*}
where we identified $R|_p$ with its class modulo $\cD|_p$ and the semisimple element
$S:=-\Lambda$ has spectrum $(-2; -1,-1,0,0,+1,+1;0)$ in the eigenbasis 
$R|_p$, $X|_p$, $\cJ X|_p$, $Y|_p$, $\cJ Y|_p$, $Z|_p$, $\cJ Z|_p$, $\Lambda$. 
Since $p$ is a regular point, the graded algebra $\opp{gr}(\gg)$ satisfies 
$[\gg_{i+1}(p),\gm_{-1}(p)]=[\gg_{i+1}(p),\cD|_p]\subset \gg_{i}(p)$ for all $i\in\mathbb Z$ by \cite[Thm. 2]{Kru2014}.

If $\gg_0(p)=\gf_0(p)$, then in particular $[S,\cJ Z|_p]=\cJ Z|_p\in\gg_{-1}(p)$, a contradiction. Therefore
$\gg_0(p)=0$, and thanks to Lemma \ref{lem:idontknow} and Propositon \ref{prop:generic2!}, we get:

  \begin{proposition}\label{prop:generic3!}
If $r\neq 0$, or $d_{10}\neq 0$, or $\sigma_1$ does not take values in $S^1\subset\mathbb C$, then $\dim\gg\leq 6$.
  \end{proposition}

We finally assume that none of the conditions $(i')$-$(i''')$ is satisfied: $r=d_{10}=0$, $\sigma_1=e^{i\theta}$. 
In this case, the Lie brackets between the normalized sections simplify to 

\par
{\footnotesize
\begin{equation*}
\begin{array}{||c|c|c|c|c|c|c||}\hline
[-,-] & X_{10} & Y_{10} & Z_{10} & X_{01} & Y_{01} & Z_{01} \\
\hline\hline
X_{10} & 0 & c_1X_{10}+c_2Y_{10} & \star & \notin D & \star & \star \\
\hline
Y_{10} & \star & 0 & \star &
\begin{gathered}X_{10}+d_1X_{01} +d_2Y_{01}\vphantom{\frac:.}\\ +d_3 Y_{10}+d_4Z\end{gathered} &
e^{i\theta}Y_{10}-e^{-i\theta}Y_{01} & \star \\
\hline
Z_{10} & \star & c_6 Y_{10}+c_7Z_{10}  & 0 &
\begin{gathered} Y_{10}+ d_5 X_{01}   + d_6 Y_{01}\vphantom{\frac:.}\\+ d_7  Z_{01}+d_8  Z_{10}\end{gathered} &
d_9 Y_{01}+d_{11}Z_{10} & \sigma_2 Z_{10}-\overline\sigma_2 Z_{01} \\
\hline
X_{01} & \star & \star  & \star & 0 & \star & \star \\
\hline
Y_{01} & \star & \star & \star & \star & 0 & \star \\
\hline
Z_{01} & \star & \star & \star & c_3X_{01}+ c_4Y_{01}+c_5Z_{01} & \star & 0 \\
\hline
\end{array}
\end{equation*}}
\vskip0pt\par
\centerline{\small\it Structure equations in the subcase $r=d_{10}=0$ and $\sigma_1=e^{i\theta}$.}
\medskip

Denoting $d_3^\sigma:=d_3e^{-i\theta/2}$, $\beta^\sigma:=\beta e^{-i\theta/2}$ 
($\theta\!\!\mod2\pi$ is an invariant) the transformation rule
$\widetilde d_3=\lambda d_3+\overline{\beta}\sigma_1-\beta$ takes the form
 $$
\widetilde{d_3^\sigma}= \lambda d_3^\sigma-2i\mathfrak{Im}\big(\beta^\sigma\big).
 $$ 
Thus we can fix $\mathfrak{Im}\big(d_3^\sigma\big)=0$ with residual freedom $\mathfrak{Im}\big(\beta^\sigma\big)=0$.
In other words, $d_3=d'_3e^{i\theta/2}$, $\beta=\beta'e^{i\theta/2}$ for some functions $d'_3,\beta':\cU\to\mathbb R$.

It is only from now on that applying Tanaka's algebraic prolongation pays off
and that we can work as in the last part of \S\ref{subsubsec:352}. 
In fact $\gg$ is represented as infinitesimal symmetries of a $7$-dimensional manifold endowed with a collection
of graded frames defined up to $\lambda:\cU\to \mathbb R_+$  and $\beta':\cU\rightarrow\mathbb R$. 
The symbol algebra is a $7$-dimensional metabelian Lie algebra together with the $2$-dimensional reduced structure algebra:
  \begin{equation*}\label{eq:symbol-with-structure-third-occurence}
\begin{aligned}
\gg_{0}(p)&\subset\gf_0(p)
\cong\operatorname{span}\big(\Lambda,B\big) \,,\\
\gg_{-1}(p)&\cong TU_\cD|_p
=\operatorname{span}\big(X|_p,\cJ X|_p,Y|_p,\cJ Y|_p, Z|_p\big)\subset\gm_{-1}(p)\cong\cD|_p\,,\\
\gg_{-2}(p)&=\gm_{-2}(p)\cong T\cU|_p/\cD|_p\cong\mathbb R R|_p\,,
\end{aligned}
  \end{equation*}
where we identified $R|_p$ with its class modulo $\cD|_p$.
The brackets of $\gf_{\leq 0}(p):=\gm(p)\rtimes\gf_0(p)$ are conveniently described in the complexified picture, 
omitting evaluation at $p$ for simplicity of notation:
  \begin{equation*}
\begin{aligned}
{}[X_{10},X_{01}]&=iR\,,\,\,[B,X_{10}]=e^{i\theta/2}Y_{10}\,,\,\,[B,X_{01}]=e^{-i\theta/2}Y_{01}\,,\,\,[\Lambda,B]=-B\,,\,\,[\Lambda,R]=2R\\
{}
[\Lambda,X_{10}]&=X_{10}\,,\quad [\Lambda,X_{01}]=X_{01}\,,\quad[\Lambda,Z_{10}]=-Z_{10}\,,\quad[\Lambda,Z_{01}]=-Z_{01}\,.
\end{aligned}
  \end{equation*}
Note that $\gf_{\leq 0}(p)$ is not the direct sum of two ideals as in \S\ref{subsubsec:352}, since $\Lambda$ acts non-trivially 
on $Z$. However, we may work at a regular point for the filtered structure: $[\gg_{0}(p),\gm_{-1}(p)]\subset \gg_{-1}(p)$, 
and if $\gg_0(p)$ includes an element $a\Lambda+bB$ with $a\neq 0$, then
$[a\Lambda+bB,\cJ Z|_p]=-a\cJ Z|_q\in\gg_{-1}(p)$. This is a contradiction, so $\gg_0(p)\subset \mathbb R B$ 
and we may assume $\gg_{0}(p)=\mathbb R B$ thanks to Lemma \ref{lem:idontknow}.

Since $\gg\cong\gg|_{\mathcal O^G_p}$ is {\it effectively} represented on the orbit $\mathcal O^G_p$, 
we may consider the maximal prolongation of $\gf'_{\leq 0}(p):=\gg_{\leq 0}(p)$, instead of $\gf_{\leq 0}(p)$.  
As in \S\ref{subsubsec:352}, the first prolongation $\gf'_1(p)$ is $2$-dimensional and the second prolongation
vanishes, so the prolongation $\bar\gf=\gf'_{\leq 0}(p)\rtimes \gf'_1(p)$ is $9$-dimensional. 
Explicitly
 $$
\gf'_1(p)=\operatorname{span}\big(B_{10}:=e^{i\theta/2}\xi_{10}\otimes B-i\rho\otimes Y_{01},
B_{01}:=e^{-i\theta/2}\xi_{01}\otimes B+i\rho\otimes Y_{10}\big)
 $$
where $\rho,\xi_{10},\xi_{01},\upsilon_{10},\upsilon_{01},\zeta,\beta$ denotes the dual basis to 
$R|_p,X_{10}|_p,X_{01}|_p,Y_{10}|_p,Y_{01}|_p,Z|_p,B$. 
Consequently we get the a-priori bound $\dim\gg\leq 9$.

 \begin{proposition}\label{P6b}
If $r=d_{10}=0$ and $\sigma_1$ takes values in $S^1\subset\mathbb C$, then $\dim\gg\leq 6$.
 \end{proposition}
 
In order to prove this we have to exclude all the cases where $\dim\gg\geq 7$.
So we investigate filtered deformations $\gg$ of the following graded subalgebras of
$\bar\gf$ that include the non-positive part (since $\gg\cong\gg|_{\cO^G_q}$ is transitive on the orbit
and $\gf'_1(p)$ is one-dimensional): 
  \begin{itemize}
\item[$(i)$] $\gf'_{\leq 0}(p)\oplus \gf'_1(p)$,
\item[$(ii)$] $\gf'_{\leq 0}(p)\oplus\ell$, where $\ell$ is a line in $\gf'_1(p)$,
\item[$(iii)$] $\gf'_{\leq 0}(p)$.
  \end{itemize}
Observations (O1) and (O3) of \S\ref{subsubsec:352} are still in force, but with the Table after Proposition \ref{prop:generic3!}, 
upon restriction to vector fields tangent to the orbit. In this case, all the structure equations can be taken into account, 
as they fall within the scope of Lemma \ref{lem:from-left-to-right}. 
We summarize all the relevant Lie brackets in the following Table (therein $s:\cU\to\mathbb R$):
\par
{\footnotesize
\begin{equation*}
\begin{array}{||c|c|c|c|c|c||}\hline
[-,-] & X_{10} & Y_{10} & Z & X_{01} & Y_{01} \\
\hline\hline
X_{10} & 0 & c_1X_{10}+c_2Y_{10} & \star & \notin TU_D & \star \\
\hline
Y_{10} & \star & 0 & \star &
\begin{gathered}X_{10}+d_1X_{01} +d_2Y_{01}\vphantom{\frac:.}\\ +e^{i\theta/2}s Y_{10}+d_4Z\end{gathered} &
e^{i\theta}Y_{10}-e^{-i\theta}Y_{01} \\
\hline
Z & \star & (c_6+\overline{d_9}) Y_{10}+\overline {d_{11}}Z & 0 &
\begin{gathered} Y_{10}+ (d_5+c_3) X_{01}\vphantom{\frac:.}\\ + (d_6+c_4) Y_{01}+d_8  Z\end{gathered} & \star \\
\hline
X_{01} & \star & \star & \star & 0 & \star \\
\hline
Y_{01} & \star & \star & \star & \star & 0 \\
\hline
\end{array}
\end{equation*}}
\smallskip
\centerline{\small\it Structure equations of the sub-frame $X$, $Y$, $Z$, $\cJ X$, $\cJ Y$.}
\smallskip

The deformation is otherwise unrestricted, since there is no semisimple element as in (O2) of \S\ref{subsubsec:352} at our disposal.   
This leads to a large number of deformation parameters, yet the Jacobi identities rule out all such possibilities (this computation 
can be found in the \textsc{Maple} supplement accompanying the arXiv posting of the article):

This proves Proposition \ref{P6b} and finishes the subcase \S\ref{subsubsec:353}.

\subsection{Dimension $d=5$}\label{sec:3.6}

 \begin{theorem}\label{thm:dim5}
If the symmetry algebra $\gg$ of a 7-dimensional 3-nondegenerate CR-hypersurface acts with generic orbits of dimension $d=5$, 
then $\dim\gg\leq 6$.
 \end{theorem}

We shall now prove this result. First note that $\opp{rk} (TU_\cD)=4$ and $\opp{rk} (TU_\cD^\cJ)$ is either $4$ or $2$, 
by dimensional reasons. Furthermore $2\leq\opp{rk}(TU_\cK)\leq 4$ but, if it is $4$, then $\dim\gg=1$ by (i) of 
Proposition \ref{prop:smallcodimension}, a contradiction.
Therefore $\opp{rk}(TU_\cK)$ is either 2 or 3.

\subsubsection{The case $\opp{rk}(TU_\cD^\cJ)=4$}\label{subsec:3.6.1}

The bundle $TU_\cD$ is $\cJ$-stable, so $TU_\cK$ and $TU_\cL$ are $\cJ$-stable too,
$TU_\cK=TU_\cK^\cJ$ has rank $2$, and $TU_\cL=TU_\cL^\cJ$ has rank either $2$ or $0$. 

If $\opp{rk}(TU_\cL)=2$, then $\cL=TU_\cL=TU_\cK$ and
 \begin{align*}
[\cL_{10},\cD_{01}]&= [\cL_{10},(TU_\cD^\cJ)_{01}]+[\cL_{10},\cK_{01}]\\
&= [(TU^\cJ_\cL)_{10},(TU_\cD^\cJ)_{01}]+[\cL_{10},\cK_{01}]\\
&\subset \big(\cK_{10}\oplus \cD_{01}\big)\cap TU\otimes\mathbb C+\big(\cL_{10}\oplus \cK_{01}\big)
=\cL_{10}\oplus \cD_{01},
 \end{align*}
where we used Lemma \ref{lem:preliminarybrackets}, $(TU_\cD^\cJ)_{01}+\cK_{01}=\cD_{01}$,
$TU_\cD\otimes\mathbb C=TU_\cD^\cJ\otimes\mathbb C=(TU_\cD^\cJ)_{10}\oplus(TU_\cD^\cJ)_{01}$, 
$(TU_\cD^\cJ)_{10}\cap \cK_{10}=\cL_{10}$. 
In this case the Levi form \eqref{eq:h.o.l.f.3} vanishes identically, contradicting $3$-nondegeneracy.
Summarizing, this case cannot happen:

  \begin{proposition}\label{prop:induced5CR}
If $\opp{rk} (TU_\cD^\cJ)=4$, then $\opp{rk}(TU_\cK)=2$ and $\opp{rk}(TU_\cL)=0$. 
In particular, the distributions $TU$ and $\cL$ are transversal.
  \end{proposition}

Therefore $T\cU=TU\oplus\cL$ and there exist local coordinates $(u^1,\ldots,u^5,v^1,v^2)$ on
$\cU\cong U\times V$ such that $TU=\langle \partial_{u^1},\ldots,\partial_{u^5}\rangle$ and
$\cL=\langle\partial_{v^1},\partial_{v^2}\rangle$.
Any infinitesimal CR-symmetry is of the form $\xi=\sum_{i=1}^5 \xi^i(u)\p_{u^i}$, since $\cL$ is $\gg$-stable, 
and $\gg$ is {\it effectively} represented on each fixed orbit $\mathcal O^G_q=\{(u,v)\in\cU\mid v=v_q\}$.
By Proposition \ref{prop:induced5CR}, each of the orbits inherits a natural structure of 
a $2$-nondegenerate $5$-dimensional CR manifold, with the CR-structure 
  $$
\big(TU_\cD|_{\mathcal O^G_q}=TU_\cD^\cJ|_{\mathcal O^G_q},\cJ|_{TU_\cD|_{\mathcal O^G_q}}\big)
  $$ 
and the Cauchy characteristic space $TU_\cK|_{\mathcal O^G_q}$. Of course $\gg\cong\gg|_{\mathcal O^G_q}$ 
is a transitive Lie algebra of infinitesimal symmetries of this structure.

Any maximally symmetric such structure is locally isomorphic to the tube over the future light cone, seen as the
homogeneous manifold $SO^\circ(3,2)/H$ for an appropriate $5$-dimensional closed subgroup $H$ of $SO^\circ(3,2)$,
see \cite{IZ,MS}.

  \begin{proposition}\label{prop:induced5CR-II}
In the case $\opp{rk} (TU_\cD^\cJ)=4$, if the $2$-non\-de\-ge\-nerate CR structure induced on 
one orbit $\cO^G_q$ is not isomorphic to the tube over the future light cone,
then $\dim\gg\leq5$.
  \end{proposition}

Indeed, the submaximal CR symmetry dimension of a $2$-nondegenerate CR structure on a 5-manifold 
is equal to $5$ by \cite{MP}; 
in our case this 5-dimensional CR manifold is homogeneous and the claim actually follows also from \cite{FK2}.
The rest of this subsection will be focused on the proof of the remaining subcase:

  \begin{proposition}\label{prop:induced5CR-III}
In the case $\opp{rk} (TU_\cD^\cJ)=4$, if the induced $2$-nondegenerate CR structure 
is isomorphic to the tube over the future light cone on every orbit $\cO^G_q$, then $\dim\gg\leq5$.
  \end{proposition}

We thus assume that every orbit is maximally symmetric. 
The symmetry dimension of any such orbit is 10, which by effectiveness of the action gives an upper bound for $\dim\gg$.

If $\dim(\gg|_{\mathcal O^G_q})=10$ for one orbit, then $\dim(\gg|_{\mathcal O^G_p})=10$ for all orbits and 
$\gg\cong\gg|_{\mathcal O^G_p}\cong \mathfrak{so}(3,2)$. We claim that this contradicts $3$-nondegeneracy.
First, we recall that the $\mathbb Z$-grading of the Lie algebra of infinitesimal CR-symmetries $\gg\cong\mathfrak{so}(3,2)$ 
is of the form $\gg=\gg_{-2}\oplus\cdots\oplus\gg_{+2}$, with $\gg_{-2}=\mathbb R$, $\gg_{-1}=\mathbb C$, 
$\gg_{0}=\mathfrak{gl}_2(\mathbb R)$, $\gg_{+i}=\gg_{-i}^*$, $i=1,2$. Let 
  $$
\mathfrak{u}(1)=\Bigl\langle J_o=\begin{pmatrix}0 & -1 \\ 1 &0 \end{pmatrix}\Bigr\rangle
  $$ 
be the compact Cartan subalgebra of $\mathfrak{sl}_{2}(\mathbb R)$ and $\mathfrak{h}_0=\mathfrak{u}(1)\oplus\R Z_1$ 
the corresponding Cartan subalgebra of $\gg_{0}=\mathfrak{gl}_2(\mathbb R)$, with $Z_1$ the grading element. 
It is a Cartan subalgebra of $\gg$ too.
Then the Lie algebra $\mathfrak{h}$ of the stabilizer $H$ is $\mathbb Z$-graded in non-negative degrees 
$\mathfrak{h}=\mathfrak{h}_0\oplus\mathfrak{h}_1\oplus\mathfrak{h}_2$, 
where $\mathfrak{h}_k=\mathfrak g_k$ for $k=1,2$, see \cite{IZ} and \cite[Example 4.5]{San}.

The tangent space $T_q\cO^G_q$ at a preferred point $q$ can be  $\mathfrak h_0$-equivariantly identified 
with $\gm=\gg/\mathfrak h\cong\gg_{-2}\oplus\gg_{-1}\oplus\gg_0/\mathfrak h_0$, where 
$(\gg_0/\mathfrak h_0)\otimes\mathbb C\cong \mathbb C_{\beta}\oplus\mathbb C_{-\beta}$ 
is the sum of the root spaces of $\gg$ for the simple long root $\beta$ and its opposite.

Invariant distributions on the orbit correspond to $\mathfrak h$-stable subspaces of $\gm$. 
Since the action of $\mathfrak h_0$ is irreducible on $\gg_{-1}$ and $\gg_0/\mathfrak h_0$, there is a unique 
invariant distribution of rank $4$. It corresponds to $\gg_{-1}\oplus \gg_0/\mathfrak h_0\subset\gm$ 
and its Cauchy characteristic space corresponds to $\gg_0/\mathfrak h_0\subset\gm$. Similarly, 
invariant complex structures on this rank $4$ invariant distribution correspond to $\mathfrak h$-invariant 
complex structures $J$ on $\gg_{-1}\oplus \gg_0/\mathfrak h_0$. 
The invariance under $Z_1$ implies that $\gg_{-1}$ and $\gg_0/\mathfrak h_0$ are $J$-stable, 
while the invariance under $J_o$ yields only one complex structure up to sign on each of them. 
Upon complexification, with $\alpha$ being the simple short root, we get:
  \begin{align}
J|_{\gg_{-\alpha}}&=+i\,,\;\; J|_{\gg_{-\beta-\alpha}}=- i\,,\label{eq:first-complex}\\
J|_{\gg_{\beta}}&=+ i\,,\;\; J|_{\gg_{-\beta}}=- i\,,\label{eq:second-complex}
  \end{align}
where each of \eqref{eq:first-complex} and \eqref{eq:second-complex} is determined up to a sign.
Integrability of the complex structure 
(equivalent to the integrability of the corresponding holomorphic and antiholomorphic distributions), 
fixes the relative sign and $J$ is the complex structure \eqref{eq:first-complex}-\eqref{eq:second-complex}, up to an overall sign. 
In summary: there is a unique up to sign CR-structure of hypersurface type on $\mathcal O^G_q$ preserved by 
$\gg\cong\gg|_{\mathcal O^G_q}$. 

Symmetries in $\gg$ have the $v$-independent form $\xi=\sum_{i=1}^5 \xi^i(u)\p_{u^i}$,
so that $\gg|_{\mathcal O^G_p}$ is not only abstractly isomorphic to $\gg|_{\mathcal O^G_q}$ for all $p\in\cU$ but actually 
{\it equal}, where we identify orbits in the obvious way. The CR-structure on $\mathcal O^G_p$ is then equal to that 
on $\mathcal O^G_q$ up to sign, but $\cJ$ depends smoothly on $v$, so it is equal. Thus $\cJ$ is projectable to 
$\widetilde{\cD}$, contradicting Proposition \ref{prop:inducedcomplexstructure}. 
Thus the case $\dim\gg=10$ is impossible.
\medskip\par
If $\dim\gg<10$, then $\gg$ is a proper subalgebra of $\mathfrak{so}(3,2)\cong\mathfrak{sp}_4(\mathbb R)$ and therefore 
either $\dim\gg<7$ or $\gg$ has dimension $7$ and is {\it conjugated} to one of the two maximal standard parabolics 
$\mathfrak p_1$, $\mathfrak p_2$ of $\mathfrak{sp}_4(\R)$, cf.\ \cite[Thm. 5]{AS}. 
We thus fix an orbit $\mathcal O_q$, with maximally symmetric $5$-dimensional $2$-nondegenerate CR-structure, 
and let $\mathfrak s\cong\mathfrak{sp}_4(\mathbb R)$ be its full symmetry algebra. It is $\mathbb Z$-graded 
$\mathfrak s=\mathfrak s_{-2}\oplus\cdots\oplus\mathfrak s_{+2}$, with $\mathfrak s_{-2}=\mathbb R$, 
$\mathfrak s_{-1}=\mathbb C$, $\mathfrak s_{0}=\mathfrak{gl}_2(\mathbb R)$, $\mathfrak s_{+i}=\mathfrak s_{-i}^*$, 
for $i=1,2$, as before and with stabilizer subalgebra at $q$ given by 
$\mathfrak{h}=\mathfrak{h}_0\oplus\mathfrak{h}_1\oplus\mathfrak{h}_2=\big(\mathfrak{u}(1)\oplus\mathbb R Z_1\big)\oplus\mathfrak s_1\oplus\mathfrak s_2$. 

It is convenient to consider the double-cover $Sp_4(\mathbb R)$ of $SO^o(3,2)$ and its quotient $Sp_4(\mathbb R)/H$ 
by the connected closed subgroup $H$ with $Lie(H)=\mathfrak h$. 
We let $(p_1,q_1, p_2,q_2)$ be a symplectic basis  of 
$\big(V, \O\big):=  \big(\mathbb R^4, -(p^*_1 \wedge q^*_1 + p^*_2 \wedge q^*_2)\big)$
and identify $\mathfrak s\cong\mathfrak{sp}(V)$ with the space
of quadratic polynomials $S^2(V)$ on $V^*$ via $uv : w \mapsto \Omega(u,w)v + \Omega(v,w)u$ for all $u,v,w\in V$. 

The parabolic subalgebra stabilizing the line $\mathrm{span}\big(p_1\big)$ is identified via the orthogonal decomposition 
$V=\cV_1\oplus \cV_2=\mathrm{span}\big(p_1,q_1\big)\oplus\mathrm{span}\big(p_2,q_2\big)$ as
 $$
\mathfrak p_1=\gs_{\ge0}=\gs_{0}\oplus\gs_{1}\oplus\gs_{2}=
\big(S^2(\cV_2)\oplus\mathbb R p_1q_1\big)\oplus p_1\cV_2\oplus\mathbb Rp_1^2.
 $$
The opposite parabolic is $\mathfrak p_1^{\text{opp}}=\gs_{\le0}
=\big(S^2(\cV_2)\oplus\mathbb R p_1q_1\big)\oplus q_1\cV_2\oplus\mathbb Rq_1^2$.

The parabolic subalgebra stabilizing the Langragian plane $\cP=\mathrm{span}\big(p_1,p_2\big)$ is identified 
via the 
Lagrangian decomposition $V=\cP\oplus \cQ=\mathrm{span}\big(p_1,p_2\big)\oplus\mathrm{span}\big(q_1,q_2\big)$ as 
 $$
\mathfrak p_2=\cQ\cP\oplus S^2(\cP)\cong\mathfrak{gl}_2(\mathbb R)\ltimes S^2(\mathbb R^2).
 $$
The opposite parabolic is $\mathfrak p_2^{\text{opp}}=\cQ\cP\oplus S^2(\cQ)$.
We also remark that $\mathfrak h=\mathfrak{h}_0\oplus\mathfrak{h}_1\oplus\mathfrak{h}_2
=\big(\mathbb R(p_2^2+q_2^2)\oplus\mathbb R p_1q_1\big)\oplus p_1\cV_2\oplus\mathbb Rp_1^2$.

\medskip

Other maximal parabolic subalgebras of $\mathfrak s$ are conjugate to the standard ones, i.e., 
either $\mathfrak p_1$ or $\mathfrak p_2$ (this includes the opposite parabolics). We describe those that act locally transitively 
on $\mathcal O_q\cong Sp_4(\mathbb R)/H$ around the fixed base point $q$:

  \begin{lemma}
Let $Sp_4(\mathbb R)/H$ be the tube over the future light cone with its maximally symmetric $2$-nondegenerate CR-structure and 
$\gg$ a maximal parabolic subalgebra of $\mathfrak s\cong\mathfrak{sp}_4(\mathbb R)$. Then $\gg$ acts locally transitively 
around the base point $q$ if and only if $\gg=Ad_{h^{-1}}\mathfrak p^{\text{opp}}$ for $h\in H$, where $\gp$ is 
a maximal standard parabolic of $\gs$ and $\gp^{\text{opp}}$ the corresponding opposite parabolic. The stabilizer subalgebra 
of $\gg$ at the base point $q$ is given by $\gg\cap\mathfrak h=Ad_{h^{-1}}\big(\mathfrak p^{\text{opp}}\cap\gh\big)$.
  \end{lemma}

  \begin{proof}
Let $P=P_1$ be the parabolic subgroup  stabilizing the line $[p_1]$ and
let $\gg$ be conjugated to the subalgebra $\mathfrak\gp=\mathfrak p_1=Lie(P_1)$.
We write $\gg=g_o^{-1}\mathfrak p g_o$, $G=g_o^{-1}Pg_o$,
for some $g_o\in Sp_4(\mathbb R)$ and note that
$G\cdot q\cong\big(g_o^{-1}Pg_o\big)H$ is open in $Sp_4(\mathbb R)/H$ if and only if the $P$-orbit
$\mathcal O^{P}_{g_{o}q}:=Pg_oH$ is open. 
We are led to determine the open orbits of the natural left action of $P$ on $Sp_4(\mathbb R)/H$. 

As intermediate step, consider the natural projection
$\pi:Sp_4(\mathbb R)/H\longrightarrow \mathbb R\mathbb P^3$
to the contact $3$-dimensional projective space $\mathbb R\mathbb P^3\cong Sp_4(\mathbb R)/P$.
It is a $P$-equivariant surjective submersion, so it is an open map, and it sends any open $P$-orbit onto an open $P$-orbit. 
A straightforward task determines the orbit structure of $P\subset Sp_4(\mathbb R)$ acting on $\mathbb R\mathbb P^3$ -- 
the orbits are distinguished by their symplectic relation with the line $[p_1]$. More precisely:
  \begin{itemize}
	\item[$(i)$] 
The line $[p_1]$ gives an orbit $\mathcal O^{P}_{[p_1]}=[p_1]$ consisting of a single point;
	\item[$(ii)$] 
The lines orthogonal to $[p_{1}]$ different from $[p_1]$ constitute an orbit 
$\mathcal O^{P}_{[p_2]}\cong \mathbb R\mathbb P^2\setminus\{\text{point}\}$; 
	\item[$(iii)$] 
The lines not orthogonal to $[p_1]$ constitute the generic orbit $\mathcal O^{P}_{[q_1]}\cong\mathbb R\mathbb P^3\setminus\mathbb R\mathbb P^2$.
  \end{itemize}
In summary, we have the stratification $\mathbb R\mathbb P^{3}=\mathcal O^{P}_{[q_1]}\cup 
\mathcal O^{P}_{[p_2]}\cup \mathcal O^{P}_{[p_1]}$, with the first orbit open.

We claim that every point in $\pi^{-1}(\mathcal O^{P}_{[q_1]})$ determines in fact an open $P$-orbit in 
$Sp_4(\mathbb R)/H$. To see it, first note that $[q_1]=a\cdot [p_1]\cong a P\in Sp_4(\mathbb R)/P$ for 
$a=\begin{psmallmatrix}J_o & 0 \\ 0 & \operatorname{Id}_{2\times 2}\end{psmallmatrix}\in Sp_4(\mathbb R)$, so the fiber 
$\pi^{-1}[q_1]=\big(aP\big)H=\big(a e^{S^2(\cV_2)}\big)H=
  \begin{psmallmatrix} 
J_o & 0 \\
0  & SL_2(\mathbb R)
  \end{psmallmatrix} 
H\subset Sp_4(\mathbb R)/H$. 
In the last identity we used the explicit form of the subgroup  
$e^{S^2(\cV_2)}=
  \begin{psmallmatrix} 
\operatorname{Id}_{2\times 2} & 0 \\
0  & SL_{2}(\mathbb R)  \\
  \end{psmallmatrix}$.
Now $\pi^{-1}(\mathcal O^{P}_{[q_1]})=P\cdot \pi^{-1}[q_1]$, 
so it is enough to check the claim for elements $\widetilde gH\in Sp_4(\mathbb R)/H$ with
  \begin{equation}\label{eq:expressiongtilde}
\widetilde g=\begin{pmatrix} J_o & 0 \\ 0  & D \end{pmatrix}\,,\quad D\in SL_2(\mathbb R)\,.
  \end{equation}
The stabilizer in $P$ of $\widetilde gH$ is given by
$\mathrm{Stab}_{P}(\widetilde g H)=P\cap Ad_{\widetilde g}(H)$, with associated Lie algebra 
$\mathfrak{stab}_{\mathfrak p}(\widetilde gH)=\mathfrak p\cap Ad_{\widetilde g}(\mathfrak h)$.
A straightforward computation using the explicit identifications of $\gp$ and $\gh$ provided beforehand
gives $\mathfrak{stab}_{\mathfrak p}(\widetilde gH)=
\mathrm{span}\big(D(p_2^2+q_2^2),p_1q_1\big)$, which is $2$-dimensional. 

The claim is thus proved and the parabolic subalgebras $\gg=Ad_{g_o^{-1}}\mathfrak p$ acting locally transitive at $q$ 
are those for which $g_o=\widetilde gh$ with $\widetilde g$ as in \eqref{eq:expressiongtilde} and $h\in H$. 
A direct computation shows $\gg=Ad_{h^{-1}}\mathfrak p^{\text{opp}}$, where
$\mathfrak p^{\text{opp}}=\mathfrak s_{0}\oplus\mathfrak s_{-1}\oplus\mathfrak s_{-2}= (S^2(\cV_2)\oplus
\mathbb R p_1q_1\big)\oplus q_1\cV_2\oplus\mathbb Rq_1^2$
is the parabolic subalgebra opposite to $\mathfrak p$. The claim on the stabilizer subalgebra is immediate. 
\smallskip\par

Let now $P=P_2$ be the parabolic subgroup  stabilizing the plane $\cP$, $\gp=\gp_2=Lie(P_2)$, and 
$\gg=g_o^{-1}\mathfrak pg_o$. We will first determine the $P$-orbits $\mathcal O^{P}_{g_{o}q}:=Pg_oH$ 
that are open in $Sp_4(\mathbb R)/H$. To this aim, we consider the twistorial correspondence given by the double fibration
 \begin{center}
\begin{tikzcd}
& Sp_4(\mathbb R)/(P\cap H) \arrow[ddl, "\pi_1" '] \arrow[ddr, "\pi_2"] & \\ \\
 Sp_4(\mathbb R)/H \hskip-12pt && \hskip-12pt Sp_4(\mathbb R)/P
\end{tikzcd}
 \end{center}
and note that, if the orbit $\mathcal O^{P}_{g_{o}q}$ is open, then $\pi_2\big(\pi_1^{-1}(\mathcal O^{P}_{g_{o}q})\big)$ 
is open and $P$-stable in $Sp_4(\mathbb R)/P$. This flag variety is the Lagrangian-Grassmannian
$LG(2,4)$ and its $P$-orbit structure follows from the mutual position w.r.t.\ the Lagrangian plane $\cP$. Explicitly, we have:
  \begin{itemize}
	\item[$(i)$] 
The orbit $\mathcal O^{P}_{\cP}=\{\cP\}$ consisting of a single point;
	\item[$(ii)$] 
The orbit $\mathcal O^{P}_{\mathrm{span}(p_1,q_2)}$ of Lagrangian planes with $1$-dimensional intersection with $\cP$. 
This orbit has dimension $2$ and it is neither open or closed;
	\item[$(iii)$] \
The open orbit $\mathcal O^{P}_{\cQ}$ of the Lagrangian planes transversal to $\cP$. 
  \end{itemize}

It follows that $\pi_2\big(\pi_1^{-1}(\mathcal O^{P}_{g_{o}q})\big)$ always contains $\mathcal O^{P}_{\cQ}$, 
in particular we have $\cQ\in \pi_2\big(\pi_1^{-1}(\mathcal O^{P}_{g_{o}q})\big)$. 
Writing $\cQ=a\cdot \cP\cong aP\in Sp_4(\mathbb R)/P$ for 
 $$
a=\exp\Bigl(\tfrac{\pi}{2}\tfrac{p_1^2+q_1^2}{2}\Bigr)\cdot\exp\Bigl(\tfrac{\pi}{2}\tfrac{p_2^2+q_2^2}{2}\Bigr)=
  \begin{psmallmatrix}
		0 & -1 & & \\
		1 & 0 & & \\
		  &   & 0 & -1\\
			& & 1 & 0
  \end{psmallmatrix},
 $$ 
this condition is rewritten as $xg_ob P=a P$ for some $x\in P$, $b\in H$, i.e., $xg_ob y=a$ for some $x,y\in P$, $b\in H$. 
By renaming the elements, we may assume w.l.o.g. that $y\in P/(P\cap H)\cong e^{\operatorname{span}(q_1p_2, q_2p_2, p_2^2)}$. 
(Note that $\operatorname{span}(q_1p_2, q_2p_2, p_2^2)$ is a Lie algebra.) Then
  \begin{align*}
\gg&=g_o^{-1}\mathfrak pg_o=bya^{-1}\big(x\mathfrak p x^{-1}\big)ay^{-1}b^{-1}\\
&=by\big(a^{-1}\mathfrak p a\big)y^{-1}b^{-1}=b\big(y\mathfrak p^{\text{opp}}y^{-1})b^{-1}\,,
  \end{align*}
with $\mathfrak p^{\text{opp}}=\cQ\cP\oplus S^2(\cQ)$ the parabolic subalgebra opposite to $\mathfrak p$. 
However $q_1p_2, q_2p_2\in \mathfrak p^{\text{opp}}$, so we may assume that $y\in e^{\mathbb Rp_2^2}$. 
A direct check using 
$p_2^2\in\mathrm{span}\big(p_2^2+q_2^2;p_2q_2,q_2^2\big)\subset \mathfrak h+\mathfrak p^{\text{opp}}$ 
shows that in fact $e^{\mathbb Rp_2^2}\subset H\cdot P^{\text{opp}}$, 
whence $\mathfrak g=Ad_{h^{-1}}\mathfrak p^{\text{opp}}$ for some $h\in H$.
  \end{proof}

If $\gg=Ad_{h^{-1}}\mathfrak p_1^{\text{opp}}$, then the stabilizer 
$\gg\cap\mathfrak h=Ad_{h^{-1}}\big(\mathfrak p_1^{\text{opp}}\cap\gh\big)=Ad_{h^{-1}}\gh_0$ is Abelian. 
Since it is conjugated to $\gh_0$, we may use its equivariance to argue as in the $\dim(\gg|_{\mathcal O^G_q})=10$ case: 
on any orbit, there are unique $G$-stable distribution of rank $4$ and complex structure on it. 
Then the $7$-dimensional CR-structure is projectable to $\widetilde{\cD}$, a contradiction. 

If $\gg=Ad_{h^{-1}}\mathfrak p_2^{\text{opp}}$, then the stabilizer 
$\gg\cap\mathfrak h=Ad_{h^{-1}}\big(\mathfrak p_2^{\text{opp}}\cap\gh\big)=Ad_{h^{-1}}(p_1\cQ)$ 
is a solvable Lie algebra of dimension $2$. Of course, it is enough to consider the case
$\gg=\mathfrak p_2^{\text{opp}}=\cQ\cP\oplus S^2(\cQ)$, with
stabilizer $\mathfrak{sol}_2:=p_1\cQ$. We use the maximal vectorial Cartan subalgebra 
$\mathfrak t=\mathrm{span}\big(Z_1,Z_2\big)$ of $\mathfrak{sp}_4(\mathbb R)$, 
where $Z_1=p_1q_1$ and $Z_2=\tfrac12(q_1p_1+q_2p_2)$ are the grading elements of $\mathfrak p_1$ and $\mathfrak p_2$.

  \begin{lemma}
The tangent space at the preferred point can be identified with the $\mathfrak{sol}_2$-module
$\gm=\gg/\mathfrak{sol}_2 \cong\mathrm{span}\big( q_1^2\big)\oplus\mathrm{span}\big( q_1q_2,q_1p_2\big)\oplus\mathrm{span}\big( q_2^2,Z_2\big)$,
with the following action of $\mathfrak{sol}_2$: $Z_1$ acts semisimply with the eigenvalues $-2,-1,0$ 
w.r.t.\ the above direct sum decomposition and 
$X:=-p_1q_2$ acts as $\tfrac12 q_1^2\mapsto q_1q_2\mapsto q_2^2\mapsto 0$, 
$q_1p_2\mapsto 2Z_2\mapsto 0\!\!\mod \mathfrak{sol}_2$.
  \end{lemma}

The proof is immediate and we omit it. Invariant distributions correspond to $\mathfrak{sol}_2$-stable subspaces of $\gm$, 
and those of rank $4$ are $\mathrm{span}\big(q_1q_2,q_1p_2,q_2^2,Z_2\big)$ and $\mathrm{span}\big(q_1^2,q_1q_2,q_2^2,Z_2\big)$. The latter correspond to an integrable distribution, so 
$TU_\cD|_{\cO^G_q}\cong\mathrm{span}\big(q_1q_2,q_1p_2,q_2^2,Z_2\big)$, with the Cauchy characteristic space 
$TU_\cK|_{\cO^G_q}\cong\mathrm{span}\big(q_2^2,Z_2\big)$.
Invariant complex structures on $TU_\cD|_{\cO^G_q}$ correspond to $\mathfrak{sol}_2$-invariant complex structures $J$: invariance by $Z_1$ says that the subspaces $\mathrm{span}\big(q_1q_2,q_1p_2\big)$ and $\mathrm{span}\big(q_2^2,Z_2\big)$
are both $J$-stable, invariance by $X$ says that $J$ on the former determines $J$ on the latter.  
Thus, on each orbit $\mathcal O^G_q$, we obtain a $2$-parameter family of CR-structures of hypersurface type, 
each preserved by $\gg\cong\gg|_{\mathcal O^G_q}$. 

One may check that each CR-structure of this family is integrabile and $2$-nondegenerate. 
Consider the simply transitive algebra 
$\mathfrak c=\mathrm{span}\big(q_1p_2, Z_2\big)\oplus S^2(\cQ)$ (subalgebra of $\gg$) consisting
of right-invariant vector fields and the corresponding algebra $\widecheck{\mathfrak c}$ of left-invariant vector fields. As already seen, elements of $\gc$ have the $v$-independent form
$\xi=\sum_{i=1}^5 \xi^i(u)\p_{u^i}$, so elements of $\widecheck{\mathfrak c}$ are $v$-independent too. 
We have $\mathfrak c$ and $\cJ$ invariant decomposition
  \begin{equation}\label{cDsplit}
\cD=TU_\cD\oplus \cL=\widecheck{\langle q_1q_2,q_1p_2\rangle}\oplus\widecheck{\langle
q_2^2,Z_2\rangle}\oplus \langle\partial_{v^1},\partial_{v^2}\rangle\,,
  \end{equation}
where the indicated subbundles correspond to the bundles $\cD/\cK$, $\cK/\cL$ and $\cL$, respectively.

Since $\cJ$ is invariant by $\mathfrak c$, it acts on the left-invariant generators of $TU_\cD$ in a $v$-dependent fashion, 
i.e., $\cJ=\cJ(v)$ 
meaning that the coefficients of $\cJ$ expressed in a left-invariant frame depend only on $v$; 
using $\mathfrak c$-invariance, we see that $\cJ=\cJ(v)$ also on $\cL$. 
Now a direct computation of the Lie bracket $[\cL_{10},\cD_{01}/\cK_{01}]$ using \eqref{cDsplit} yields the following
alternative: either the $3^{rd}$-order Levi form \eqref{eq:h.o.l.f.3} gives a non-zero contribution in $\cD_{10}/\cK_{10}$ 
or it vanishes identically. The first contradicts Lemma \ref{lem:preliminarybrackets}, the second
contradicts $3$-nondegeneracy.

This completes the proof of Proposition \ref{prop:induced5CR-III} and 
finishes the subcase \S\ref{subsec:3.6.1}.

\subsubsection{The case $\opp{rk}(TU_\cD^\cJ)=2$, $\opp{rk}(TU_\cK)=3$}\label{subsec:3.6.2}

By Proposition \ref{prop:smallcodimension} $(ii)$ there exist coordinates $(u^1,\ldots,u^{5},v^1,v^2)$ on $\cU$ 
such that $\gg$ is isomorphic to the Abelian Lie algebra 
  \begin{equation*}\label{eq:abelian5}
\R^r=\mathrm{span}\big(\xi_a=\partial_{u^a}\;(1\leq a\leq 5),\;\xi_{b}=\sum_{i=1}^{5}\xi^i_{b}(v)\partial_{u^i}\;(6\leq b\leq r)\big)\,, 
  \end{equation*}
or its scaling extension
  \begin{equation*}\displaystyle\label{eq:scalingextension5}
\R\ltimes\R^r=
 \begin{cases}\displaystyle
\mathrm{span}\big(\xi_a=\partial_{u^a}\;(1\leq a\leq 5),\;\xi_{b}=\sum_{i=1}^{5}\xi^i_{b}(v)\partial_{u^i}\;(6\leq b\leq r),\;
\xi_{r+1}=\sum_{i=1}^{5} u^i\partial_{u^i}\big)\\
\text{or}\\ \displaystyle
\mathrm{span}\big(\xi_a=\partial_{u^a}\;(1\leq a\leq 4),\;\xi_b=\sum_{i=1}^{4}\xi^i_{b}(v)\partial_{u^i}\;(5\leq b\leq r),\;
\xi_{r+1}=\partial_{u^5}+\sum_{i=1}^{4}u^i\partial_{u^i}\big)
 \end{cases}
  \end{equation*}

We focus on the Abelian ideal $\R^r$. 
Let $\xi=\sum_{i=1}^{5}\xi^i(v)\partial_{u^i}$  be an infinitesimal CR-symmetry vanishing at a fixed regular point $p\in\cU$ 
for the Weisfeiler filtration (the last summand is not present in the case of the second scaling extension, i.e., $\xi^5=0$ identically), 
with linear part
  $$
\Xi=-\sum_{i=1}^{5}(\partial_{v^1}\xi^i)|_p\;\partial_{u^i}\otimes dv^1-\sum_{i=1}^{5}(\partial_{v^2}\xi^i)|_p\;\partial_{u^i}\otimes dv^2\,.
  $$
The complex structure $\cJ$ induces an embedding $\cJ:TU_\cD/TU^\cJ_\cD\to T\cU/TU$ of real vector bundles of equal rank 2, 
so an isomorphism. Therefore, we may find vectors $v_1,v_2$ in $TU_\cD|_p$ 
so that $\cJ v_1$ and $\cJ v_2$ are linearly independent in  $T_p\,\cU/T_pU$
and $\Xi(\cJ v_k)=\cJ(\Xi v_k)=0$ for $k=1,2$, since $\Xi$ acts trivially on $TU_\cD|_p$. Then $\Xi=0$,  the stabilizer of the Abelian ideal $\mathbb R^r$ vanishes by Lemma \ref{lem:idontknow}, $r=\dim\mathbb R^r\leq 5$ and we have established: 
  
  \begin{proposition}
If $\opp{rk} (TU_\cD^\cJ)=2$ and $\opp{rk}(TU_\cK)=3$, then $\dim\gg\leq 6$.
  \end{proposition}

\subsubsection{The case $\opp{rk}(TU_\cD^\cJ)=2$, $\opp{rk}(TU_\cK)=2$}\label{subsec:3.6.3}
 
This case is more involved and splits in three possibilities, ordered by the rank of $TU_\cL$, 
which can be either $0$, $1$ or $2$, namely:
\begin{itemize}
	\item[$(i)$] $TU\oplus \cL=T\cU$,
	\item[$(ii)$] $\opp{rk}(TU_\cL)=1$,
	\item[$(iii)$] $\cL=TU_\cD^\cJ$.
\end{itemize}
\smallskip\par

\noindent
{\it Subcase (i).} 
If $TU$ and $\cL$ are transversal, then there exist local coordinates $(u^1,\ldots,u^5,v^1,v^2)$ on
$\cU\cong U\times V$ such that $TU=\langle \partial_{u^1},\ldots,\partial_{u^5}\rangle$ and
$\cL=\langle\partial_{v^1},\partial_{v^2}\rangle$.
Any infinitesimal CR-symmetry has the form $\xi=\sum_{i=1}^5 \xi^i(u)\p_{u^i}$, since $\cL$ is $\gg$-stable.
Now $TU_\cD$ is a non-holonomic distribution
with Cauchy characteristic space $TU_\cK$ and $\gg\cong\gg|_{\mathcal O^G_q}$ is {\it effectively} represented 
on any orbit $\mathcal O^G_q$ as a Lie algebra of infinitesimal symmetries of $TU_\cD|_{\mathcal O^G_q}$. 
However $\opp{rk}(TU_\cD^\cJ)=2$, so {\it $TU_\cD$ is not $\cJ$-stable and the orbit $\mathcal O^G_q$ does  not inherit a CR structure simply by restriction}. 

Let $\pi_{TU}:T\cU=TU\oplus \cL\to TU$ be the projection to $TU$ along $\cL$. Since $\cD=TU_\cD\oplus \cL$, we may set $\widecheck{\cD}:=\pi_{TU}(\cD)=TU_\cD$ and define a complex structure
on $\widecheck \cD$ via projection:
  \begin{equation}\label{eq:projectedCR-structure}
\widecheck \cJ:=\pi_{TU}\circ \cJ|_{\widecheck \cD}:\widecheck \cD\rightarrow \widecheck \cD\,.
  \end{equation} 
Clearly $\widecheck{\cK}:=\pi_{TU}(\cK)=TU_\cK$ is $\widecheck \cJ$-stable too, as $\cK=TU_\cK\oplus \cL$. By construction $[\widecheck{\cD},\widecheck{\cD}]\subset TU$, and $\pi_{TU}|_{\cD_{10}}:\cD_{10}\to \widecheck{\cD}_{10}$ induces natural identifications $\cD_{10}/\cL_{10}\cong \widecheck{\cD}_{10}$ and $\cK_{10}/\cL_{10}\cong \widecheck{\cK}_{10}$.

 \begin{proposition}
The complex structure $\widecheck{\cJ}$ given by \eqref{eq:projectedCR-structure} satisfies:
\begin{enumerate}
\item It is not Levi-flat, in the sense that $[\widecheck{\cD}_{10},\widecheck{\cD}_{01}]\not\subset \widecheck{\cD}\otimes\mathbb C$,
\item The  
sections of $\widecheck{\cK}_{10}$ are precisely the sections of $\widecheck{\cD}_{10}$ that send $\widecheck{\cD}_{01}$ into $\widecheck \cD\otimes\mathbb C$,
\item It is $2$-nondegenerate, in the sense that $[\widecheck{\cK}_{10},\widecheck{\cD}_{01}]\not\subset\widecheck{\cK}_{10}\oplus  \widecheck{\cD}_{01}$,
\item It is integrable restricted on $\widecheck{\cK}$, i.e., $[\widecheck{\cK}_{10},\widecheck{\cK}_{10}]\subset \widecheck{\cK}_{10}$.
\end{enumerate}
However, it is not necessarily integrable on $\widecheck{\cD}$.
In particular $\gg$ is effectively represented 
on any orbit as infinitesimal symmetries of a $5$-dimensional $2$-nondegenerate almost CR structure. 
 \end{proposition}

 \begin{proof}
The first two claims are straightforward. 
To prove (3) note that 
$[\cK_{10},\cD_{01}]\not\subset \cK_{10}\oplus \cD_{01}$ by $3$-nondegeneracy translates into 
$[\widecheck{\cK}_{10},\widecheck{\cD}_{01}]\not\subset \cK_{10}\oplus \cD_{01}$ 
by Lemma \ref{lem:preliminarybrackets} since 
$\cK_{10}\equiv \widecheck{\cK}_{10}\!\mod \cL\otimes\mathbb C$ and
$\cD_{01}\equiv \widecheck{\cD}_{01}\!\mod \cL\otimes\mathbb C$. Therefore
  $$
[\widecheck{\cK}_{10},\widecheck{\cD}_{01}]\not\subset
\big(\cK_{10}\oplus \cD_{01}\big)\cap TU_\cD\otimes\mathbb C
=\big(\widecheck{\cK}_{10}\oplus\widecheck{\cD}_{01}\oplus (\cL\otimes\mathbb C)\big)\cap TU_\cD\otimes
\mathbb C=\widecheck{\cK}_{10}\oplus  \widecheck{\cD}_{01}\,.
  $$
To establish $(4)$, we consider sections $Y_1,Y_2$ of $\widecheck{\cK}$ and note that
  \begin{align*}
[Y_1-i\widecheck{\cJ} Y_1,Y_2-i\widecheck{\cJ} Y_2]&=\pi_{TU}[Y_1-i\cJ Y_1,Y_2-i\cJ Y_2]-\pi_{TU}[Y_1-i\cJ Y_1,i(\widecheck{\cJ} Y_2-\cJ Y_2)]\\
&\;\;\;-\pi_{TU}[i(\widecheck{\cJ} Y_1-\cJ Y_1),Y_2-i\cJ Y_2]
  \end{align*}
is still a section of $\widecheck{\cK}_{10}$, since $\pi_{TU}\cK_{10}= \widecheck{\cK}_{10}$ and
$\pi_{TU}[\cK_{10},\cL\otimes\mathbb C]\subset \pi_{TU}\big(\cK_{10}\oplus \cL_{01}\big)=\widecheck{\cK}_{10}$.
  \end{proof}

Now  $\widecheck \cJ$ is not necessarily integrable on $\widecheck{\cD}$, since 
$[\cD_{10},\cL\otimes\mathbb C]\subset\cD_{10}\oplus \cK_{01}$,and we have a dichotomy: 
the integrable case $[\widecheck{\cD}_{10},\widecheck{\cD}_{10}]\subset \widecheck{\cD}_{10}$ and the non-integrable case 
$[\widecheck{\cD}_{10},\widecheck{\cD}_{10}]\not\subset \widecheck{\cD}_{10}$. 

\medskip

If \eqref{eq:projectedCR-structure} is integrable, we may argue exactly as in \S\ref{subsec:3.6.1} when $\opp{rk}(TU_\cL)=0$ 
but with \eqref{eq:projectedCR-structure} in place of $\cJ|_{TU_\cD}$:
if $\dim\gg\geq 7$, then all the orbits are locally isomorphic to the tube over the future light cone, and 
either the CR-structures on all the orbits can be identified by translations along $\cL$
or $\gg$ is conjugated to the parabolic subalgebra $\mathfrak p_2$ of $\mathfrak{sp}_4(\R)\cong\mathfrak{so}(3,2)$. 
In the first case \eqref{eq:projectedCR-structure} is projectable to $\widetilde{\cD}$, 
so the original CR-structure $\cJ$ is projectable too, a contradiction by Proposition \ref{prop:inducedcomplexstructure}. 
In the second case both $\widecheck \cJ=\widecheck \cJ(v)$ on $\widecheck \cD=TU_\cD$
and $\cJ=\cJ(v)$ on $\cL$ are $v$-dependent. Since $\cD_{01}\equiv \widecheck{\cD}_{01}\!\mod \cL\otimes\mathbb C$, the $3^{rd}$-order Levi form \eqref{eq:h.o.l.f.3} of the original CR-structure
either gives a non-zero contribution in $\cD_{10}/\cK_{10}$ or
it vanishes identically, in any case a contradiction. Thus we conclude:

  \begin{proposition}
If $\opp{rk} (TU_\cD^\cJ)=2$, $\opp{rk}(TU_\cK)=2$, $\opp{rk} (TU_\cL)=0$ and the complex structure $\widecheck\cJ$ 
as in \eqref{eq:projectedCR-structure} is integrable, then $\dim\gg\leq 6$.
  \end{proposition}

Now we focus on non-integrable $2$-nondegenerate CR-structures. This means that for any section $X_{10}$ of 
$\widecheck \cD_{10}$, which is not in $\widecheck \cK_{10}$ at all points, and any non-zero section $Y_{10}$ of 
$\widecheck \cK_{10}$, we have $[X_{10},Y_{10}]\equiv r X_{01}+s Y_{01}\mod \widecheck \cD_{10}$
for functions $r,s:\cU\to \mathbb C$, not both identically vanishing. 
The freedom of choosing such a frame $(X_{10},Y_{10})$ of $\widecheck \cD_{10}$ is as follows: 
$\widetilde X_{10}=\alpha X_{10}+\beta Y_{10}$ and $\widetilde Y_{10}=\gamma Y_{10}$ for non-zero functions 
$\alpha,\gamma:\cU\to \mathbb C_\times$ and a function $\beta:\cU\to\mathbb C$. 

\medskip

If $r\neq 0$, there exists a frame $(X_{10},Y_{10})$ of $\widecheck \cD_{10}$ such that 
$[X_{10},Y_{10}]\equiv X_{01}\opp{mod}\widecheck{\cD}_{10}\oplus\widecheck{\cK}_{01}$, 
and the transformations respecting this correspond to $(\alpha,\beta,\gamma)=(\lambda e^{i\phi},\beta,e^{-2i\phi})$, 
with $\lambda:\cU\to \mathbb R_+$ and $\phi:\cU\to\R$. Using the $2$-nondegeneracy condition
  \begin{equation}\label{eq:2nondeg-almost-the-end}
[X_{10},Y_{01}]=t\,X_{01}\,\opp{mod}\widecheck{\cD}_{10}\oplus\widecheck{\cK}_{01}\,,
  \end{equation}
the structure function $t:\cU\to\mathbb C$ is non-zero. It changes as $\widetilde{t}=e^{4i\phi}t$, 
and we may normalize the sections in such a way that $t:\cU\to\R_+$. This normalization fixes $\phi$ 
modulo $\frac{\pi}{2}\mathbb Z$, i.e., up to finite cover of $\cU$, and we may set $\phi=0$.

The residual transformations are
$\widetilde{X}_{10}=\lambda X_{10}+\beta Y_{10}$, $\widetilde{Y}_{10}=Y_{10}$,
for $\lambda:\cU\to\mathbb R_+$, $\beta:\cU\to\mathbb C$.
Writing $[X_{10},Y_{10}]\equiv X_{01}+bY_{01}\opp{mod}\widecheck{\cD}_{10}$ we compute
the change $\widetilde{b}=\lambda b-\bar\beta$. Enforce the normalization $b=0$ via $\beta=\lambda\bar{b}$.
The final transformations are: $\widetilde{X}_{10}=\lambda X_{10}$, $\widetilde{Y}_{10}=Y_{10}$.

Hence $\gg\cong \gg|_{\cO^G_q}$ is effectively represented as infinitesimal symmetries of a $5$-dimensional manifold 
endowed with a collection of graded frames defined up to a scaling. The associated symbol algebra
is a $5$-dimensional metabelian Lie algebra with $1$-dimensional reduced structure algebra:
  \begin{equation}\label{eq:symbol-with-structure-close-to-last--occurence}
\begin{aligned}
\gg_0(q)&\subset \gf_0(q)\cong\mathbb R \Lambda\\
\gg_{-1}(q)&=\gm_{-1}(q)\cong TU_\cD|_q=\operatorname{span}\big(X|_q, \widecheck\cJ X|_q, Y|_q, \widecheck\cJ Y|_q\big)\,,\\
\gg_{-2}(q)&=\gm_{-2}(q)\cong TU|_q/TU_\cD|_q=\mathbb R R|_q\,,
\end{aligned}
  \end{equation}
The semisimple element $S=-\Lambda$ has spectrum $(-2; -1,-1,0,0)$ 
in the frame $R=[X,\widecheck\cJ X]_q$, $X|_q$, $\widecheck\cJ X|_q$, $Y|_q$, $\widecheck\cJ Y|_q$;
its first prolongation is trivial. Consequently, $\dim\gg=\dim(\gg|_{\mathcal O^G_q})\leq 6$.

\medskip

If $r=0$ but $s\neq0$, there exists a frame $(X_{10},Y_{10})$ of $\widecheck \cD_{10}$ such that 
$[X_{10},Y_{10}]=Y_{01}\,\opp{mod}\,\widecheck{\cD}_{10}$ and the transformations restrict to functions 
$(\alpha,\beta,\gamma)=(e^{-2i\phi},\beta,\lambda e^{i\phi})$, with $\lambda:\cU\to \mathbb R_+$ and 
$\phi:\cU\to\R$. Using \eqref{eq:2nondeg-almost-the-end},
the non-zero structure function $t:\cU\to\mathbb C$ changes as $\widetilde{t}=\lambda e^{-5i\phi}t$, so we can set $s=1$.
Up to a finite cover, this normalization fixes both $\lambda$ and $\phi$, so the residual transformations are: 
$\widetilde{X}_{10}=X_{10}+\beta Y_{10}$, $\widetilde{Y}_{10}=Y_{10}$.

We write $[X_{10},Y_{01}]=X_{01}+bY_{01} \opp{mod}\widecheck{\cD}_{10}$, 
$[Y_{10},Y_{01}]=\kappa Y_{10}-\bar{\kappa}Y_{01}$ for some $b,\kappa:\cU\to\mathbb C$, 
and compute the change of structure function $\widetilde{b}=b-\bar{\kappa}\beta-\bar{\beta}$.

If $|\kappa|\neq1$, then we enforce the normalization $b=0$, fixing $\beta$, 
and obtain a canonical absolute parallelism on the orbit $\mathcal O^G_q$. 
In this case $\dim\gg=\dim(\gg|_{\mathcal O^G_q})\leq 5$.

If, on the other hand, $\kappa=e^{2i\psi}$ for a function $\psi:\cU\to\R$ then we compute the change
  $$
\widetilde{be^{i\psi}}=be^{i\psi}-2\mathfrak{Re}(e^{-i\psi}\beta)\,.
  $$
Enforcing the normalization $\mathfrak{Re}\bigl(be^{i\psi}\bigr)=0$, we fix one real dimension of the 
complex function $\beta:\cU\to\mathbb C$ (the residual transformations respect $\mathfrak{Re}(e^{-i\psi}\beta)=0$).
The associated symbol algebra, taking into account the reduction of structure algebra, is given by
  \begin{equation}\label{eq:symbol-with-structure-almost-last-occurence}
\begin{aligned}
\gg_{0}(q)&\subset \gf_0(q)\cong\mathbb R B \,,\\
\gg_{-1}(q)&=\gm_{-1}(q)\cong TU_\cD|_q=\operatorname{span}\big(X|_q, \widecheck\cJ X|_q, Y|_q, \widecheck\cJ Y|_q\big)\,,\\
\gg_{-2}(q)&=\gm_{-2}(q)\cong TU|_q/TU_\cD|_q=\mathbb R R|_q\,,
\end{aligned}
  \end{equation}
where $B$ act as $[B,X_{10}]=ie^{i\psi}Y_{10}$, $[B,X_{01}]=-ie^{-i\psi}Y_{01}$,
and trivially on $Y_{10}$, $Y_{01}$ and $R$
(we used the natural complexified basis and omitted evaluation at $q$ for simplicity). 
If $\gg_{0}(q)=0$, then $\gg_{k}(q)=0$ for all $k\geq 1$ and $\dim\gg=\dim(\gg|_{\mathcal O^G_q})\leq 5$. 
If $\gg_{0}(q)=\mathbb R B$, a computation similar to that at the end of \S\ref{subsubsec:353}
tells us that the first prolongation $\gf_1(q)$ is $2$-dimensional and the second prolongation
vanishes, so overall the Tanaka prolongation is $8$-dimensional. Explicitly
$\gf_1(q)=\operatorname{span}\big(B_{10}:=e^{i\psi}\xi_{10}\otimes B-\rho\otimes Y_{01},
B_{01}:=e^{-i\psi}\xi_{01}\otimes B+\rho\otimes Y_{10}\big)$,
where $\rho,\xi_{10},\xi_{01},\upsilon_{10},\upsilon_{01},\beta$ is the dual basis to 
$R|_q,X_{10}|_q,X_{01}|_q,Y_{10}|_q,Y_{01}|_q,B$.

To exclude the cases $7\leq \dim\gg\leq 8$, we investigate filtered deformations of graded subalgebras 
of the prolongation including the non-positive part and of dimension $\geq 7$. 
Thanks to the observations (O1) and (O3) of \S\ref{subsubsec:352}, 
we consider such filtered deformations with partially restricted structure relations, 
as in the following Table with the condition $\mathfrak{Re}\bigl(be^{i\psi}\bigr)=0$:
\par
{\footnotesize
  \begin{equation*}
\begin{array}{||c|c|c|c|c||}\hline
[-,-] & X_{10} & Y_{10} & X_{01} & Y_{01} \\
\hline\hline
X_{10} &
0 & Y_{01}+c_1X_{10}+c_2Y_{10} & \notin TU_D & X_{01}+d_1X_{10}+d_{2}Y_{10}+bY_{01} \\
\hline 
Y_{10} & \star & 0 & \star & e^{2i\psi}Y_{10}-e^{-2i\psi}Y_{01} \\
\hline
X_{01} & \star & \star & 0 & \star \\
\hline
Y_{01} & \star & \star & \star & 0 \\
\hline
\end{array}
  \end{equation*}}
\smallskip
\centerline{\small\it Structure equations of the sub-frame $X$, $Y$, $\widecheck \cJ X$, $\widecheck \cJ Y$.}
\smallskip

The Jacobi identities rule out all such possibilities (this computation can be found in the \textsc{Maple} supplement 
accompanying the arXiv posting of the article). Thus we conclude: 

  \begin{proposition}
If $\opp{rk} (TU_\cD^\cJ)=2$, $\opp{rk}(TU_\cK)=2$, $\opp{rk} (TU_\cL)=0$ and the complex structure $\widecheck\cJ$ 
as in \eqref{eq:projectedCR-structure} is non-integrable, then $\dim\gg\leq 6$.
  \end{proposition}
\smallskip\par 

\noindent
{\it Subcase (ii).}
If $\opp{rk}(TU_\cL)=1$, then we have a direct sum decomposition $TU_\cD=TU_\cD^\cJ\oplus TU_\cK $ 
(otherwise $TU_\cK^\cJ=TU_\cD^\cJ\cap TU_\cK$ is non-trivial, so $\opp{rk}(TU_\cK^\cJ)=2$,
$TU_\cD^\cJ=TU_\cK^\cJ=TU_\cK\supset TU_\cL$, and $TU_\cL^\cJ=TU_\cL$, which is not possible) 
and the canonical real line subbundle $TU_\cL\subset \cL\cong \cL_{10}$.

  \begin{lemma}\label{lemma:cotto}
There is a frame $(X,\cJ X,Y,\cJ Y,Z,\cJ Z)$ on $\cD$ such that
 $$
TU_\cL=\langle Z\rangle\,,\;\ \cL=\langle Z,\cJ Z\rangle\,,\;\ TU_\cK=\langle Z,e^{\cJ\psi}Y\rangle\,,\;\
\cK=\langle Z,\cJ Z,Y,\cJ Y\rangle\,,\;\ TU_\cD^\cJ=\langle X,\cJ X\rangle\,,
 $$
where $e^{\cJ\psi}Y=\cos\psi Y+\sin\psi \cJ Y$. 
  \end{lemma}

The proof is straightforward, we omit it. 
Note that $TU=\langle X,\cJ X, e^{\cJ\psi}Y,Z, R=\tfrac12 [X,\cJ X]\rangle$ and $TU_\cK^\cJ=0$;
the function $\psi:\cU\to \R\!\mod\,2\pi$ is not fixed at this stage.

The sections as in Lemma \ref{lemma:cotto} are defined up to transformations of the form 
  \begin{equation}\label{eq:gauge-freedom-explicit-final}
X\mapsto \widetilde{X}=\lambda_1 e^{\cJ \theta_1}X\,,\quad Y\mapsto\widetilde{Y}=\lambda_2e^{\cJ \theta_2} Y
+\rho e^{\cJ (\theta_2-\psi)} Z\,,\quad Z\mapsto \widetilde{Z}=\pm\lambda_3Z\,,
  \end{equation}
for some functions $\lambda_1,\lambda_2,\lambda_3:\cU\to\mathbb R_+$, $\rho:\cU\to\mathbb R$, 
$\theta_1,\theta_2:\cU\to\R\!\mod\,2\pi$. 
In terms of the corresponding sections of the holomorphic bundles, we have
  \begin{equation*}
X_{10}\mapsto \widetilde{X}_{10}=\lambda_1 e^{i\theta_1}X_{10}\,,\quad Y_{10}\mapsto\widetilde{Y}_{10}=\lambda_2 e^{i\theta_2} Y_{10}+\rho e^{i(\theta_2- \psi)} Z_{10}\,,
\quad Z_{10}\mapsto \widetilde{Z}_{10}=\pm\lambda_3 Z_{10}\,,
  \end{equation*}
and we enforce the normalization conditions  \eqref{eq:norm-conditions}:
$[Y_{10},X_{01}]=X_{10}\,\opp{mod}\,\cK_{10}\oplus \cD_{01}$ 
and $[Z_{10},X_{01}]=Y_{10}\,\opp{mod}\,\cL_{10}\oplus \cD_{01}$.
Up to a finite cover, this fixes $\psi$ and the frames are defined up to transformations 
$\widetilde{X}_{10}=\lambda X_{10}$, $\widetilde{Y}_{10}=Y_{10}+\rho e^{-i\psi}  Z_{10}$,
$\widetilde{Z}_{10}=\lambda^{-1} Z_{10}$, for some $\lambda:\cU\to\R_+$.

Consider the expression $[Y_{10},X_{01}]=X_{10}+cY_{10}\mod \cL_{10}\oplus \cD_{01} $ with $c:\cU\to\mathbb C$,
and directly compute the change of structure function as
$\widetilde{c}=\lambda(c+e^{-i\psi}\rho)$. We enforce the normalization $\mathfrak{Re}(e^{i\psi}c)=0$,
with the remaining frame freedom
$\widetilde{X}_{10}=\lambda X_{10}$, $\widetilde{Y}_{10}=Y_{10}$, $\widetilde{Z}_{10}=\lambda^{-1} Z_{10}$. 

In summary $\gg$ is a Lie algebra of infinitesimal symmetries of a $7$-dimensional manifold endowed with a collection
of graded frames defined up to scaling. The associated $7$-dimensional symbol algebra, taking into account the reduction 
of structure algebra, is given by
  \begin{equation*}\label{eq:symbol-with-structure-next-to-last-occurence}
\begin{aligned}
\gf_0(p) &\cong\R\Lambda\,,\\
\gm_{-1}(p)&\cong\cD|_p=\mathrm{span}\big(X|_p, \cJ X|_p, Y|_p, \cJ Y|_p, Z|_p, \cJ Z|_p\big)\,,\\
\gm_{-2}(p)&\cong T\cU|_p/\cD|_p\cong\R R|_p\,,
\end{aligned}
  \end{equation*}
where we identified $R|_p$ with its class modulo $\cD|_p$.
The spectrum of the semisimple element $S:=-\Lambda$ acting on $R|_p$; $X|_p$, $\cJ X|_p$, $Y|_p$, etc.,  
is  $(-2; -1,-1,0,0,+1,+1)$, and the Tanaka prolongation is trivial in positive degrees. 
Since $\gf_{-2}(p)\cong\gm_{-2}$ and $\gf_{-1}(p)\subset\gm_{-1}(p)$ has codimension 2,
we conclude:

  \begin{proposition}
If $\opp{rk} (TU_\cD^\cJ)=2$, $\opp{rk}(TU_\cK)=2$, $\opp{rk} (TU_\cL)=1$, then $\dim\gg\leq 6$.
  \end{proposition}
\smallskip\par

\noindent
{\it Subcase (iii).}
If $\opp{rk}(TU_\cL)=2$, then $\cL=TU_\cL=TU_\cK=TU^\cJ_\cD$. Now $TU$ and $\cK$ are integrable distributions with integrable sum 
$T\cU=TU+\cK$, and their intersection $\cL$ is also integrable, so we may find adapted coordinates $(u^1,u^2,u^3,v^1,v^2,w^1,w^2)$ such that 
\begin{equation}
\label{eq:distribution-aligned-last}
\begin{aligned}
TU&=\langle \partial_{u^1},\partial_{u^2},\partial_{u^3},\partial_{w^1},\partial_{w^2}\rangle\,,\\
\cK&=\langle \partial_{v^1},\partial_{v^2},\partial_{w^1},\partial_{w^2}\rangle\,,\\
\cL&=\langle\partial_{w^1},\partial_{w^2}\rangle\,.
\end{aligned}
\end{equation}
The Lie algebra $\gg$ consists of vector fields
$\xi=\sum_{i=1}^3\xi^i\partial_{u^i}+\sum_{j=1}^2\eta^j\partial_{w^j}$,
where the $\xi^i$'s depend only on the coordinates 
$u$, as $\cK$ is $\gg$-stable. By Proposition \ref{prop:inducedcomplexstructure}, the projection 
\eqref{eq:projectionstoleafspaceL} to the leaf space of $\cL$
is injective on CR-symmetries, i.e., $\gg\cong\widetilde{\pi}_\star(\gg)$.
We may thus decompose any infinitesimal CR-symmetry as $\xi=\xi'+\xi''$, $\xi':=\sum_{i=1}^3\xi^i\partial_{u^i}$, $\xi'':=\sum_{j=1}^2\eta^j\partial_{w^j}$, the component $\xi'$ determines $\xi$, and $\gg$ is {\it effectively} represented on any orbit $\mathcal O^G_q=\{(u,v,w)\in\cU\mid v=v_q\}$.

By Proposition \ref{prop:inducedcomplexstructure}, the $5$-dimensional base $\widetilde \cM$ of \eqref{eq:projectionstoleafspaceL}  comes with distributions $\widetilde \cD$, $\widetilde \cK$ and the complex structure $\cJ$ induces a complex structure $\widetilde \cJ$ on $\widetilde \cD/\widetilde \cK$. 
The distribution $TU_\cD$ is $\widetilde\pi$-projectable to $\widetilde {TU_\cD}=\widetilde\pi_*(TU_\cD)$, since $TU_\cD\supset TU_\cL=\cL$ and $[TU_\cD, \cL]=[TU_\cD, TU_\cL]\subset TU_\cD$. Now $TU_\cD/\cL\cong \cD/\cK$, so 
$\widetilde {TU_D}=\widetilde\pi_*(TU_\cD/\cL)\cong  \widetilde\pi_*(\cD/\cK)
=\widetilde \cD/\widetilde \cK$ and $\widetilde {TU_D}$ inherits a complex structure $\mathcal I$ from the complex structure $\widetilde \cJ$.

Let us fix one orbit $\mathcal O^G_q$. It is endowed with the non-integrable distribution $TU_\cD|_{\mathcal O^G_q}$ 
with Cauchy characteristic space $TU_\cK|_{\mathcal O^G_q}=\cL|_{\mathcal O^G_q}$. The leaf space 
  $$
\widetilde\pi|_{\mathcal O^G_q}:\mathcal O^G_q\to\widetilde {\mathcal O^G_q}
  $$ 
by $\cL|_{\mathcal O^G_q}$ then determines  a contact $3$-dimensional CR manifold 
$\widetilde {\mathcal O^G_q}:=\widetilde\pi(\mathcal O^G_q)$ on which the Lie algebra 
$\gg\cong\widetilde{\pi}_\star(\gg)$ acts effectively via the isomorphism $\xi\mapsto\xi'$. 
Explicitly, the contact CR structure is given by
  \begin{equation}\label{eq:we-are-there}
\mathcal C:=\widetilde{TU_\cD}|_{\widetilde{\mathcal O^G_q}}\,,\;\;\mathcal I|_{\widetilde{\mathcal O^G_q}}:\mathcal C\to\mathcal C\,,
  \end{equation}
and the integrability of $\mathcal I|_{\widetilde{\mathcal O^G_q}}$ is of course automatic in $3$ dimensions. In summary: 

  \begin{lemma}
The leaf space $\widetilde \cM$ has local coordinates $(u,v)=(u^1,u^2,u^3,v^1,v^2)$ and it is foliated by $3$-dimensional contact CR manifolds $\widetilde{\mathcal O^G_q}=\{(u,v)\in\widetilde \cM\mid v=v_q\}$. 
  \end{lemma}
  
We stress that \eqref{eq:we-are-there} is {\it not} constructed by projecting to the $3$-dimensional leaf space \eqref{eq:projectionstoleafspaceK}: there is no induced complex structure there, see $(iv)$ Proposition \ref{prop:inducedcomplexstructure}. Rather, we used the projection to the leaf space of $\cL$ and 
the complex structure on $\cD/\cK$, but projecting a fixed $5$-dimensional orbit.

Thus we obtained a $3$-dimensional contact CR-manifold 
  $$
(\widetilde{\mathcal O^G_q},\mathcal C,\mathcal I|_{\widetilde{\mathcal O^G_q}})
  $$
on which $\gg\cong\widetilde{\pi}_\star(\gg)$ acts by infinitesimal symmetries. 
Locally, any maximally symmetric such structure is isomorphic to the CR sphere $SU(1,2)/B$ 
and the submaximal CR symmetry dimension is $3$, see \cite{C,Kru2016}. 
Since the proper subalgebras of $\mathfrak{su}(1,2)$
have dimension at most $5$, we only consider the case where
$\widetilde{\pi}_\star(\gg)\cong\mathfrak{su}(1,2)$ and
show that this contradicts $3$-nondegeneracy.
As already established in \S\ref{sec:3.3}, there is a unique (up to sign) contact CR-structure on each $\widetilde{\mathcal O^G_q}$ preserved by $\widetilde{\pi}_\star(\gg)$. On the other hand, symmetries in $\widetilde{\pi}_\star(\gg)$ have the local form
$\xi'=\sum_{i=1}^3 \xi^i(u)\p_{u^i}$,
so that $\widetilde{\pi}_*(\gg)|_{\widetilde{\mathcal O^G_p}}$ is not only abstractly isomorphic to $\widetilde{\pi}_*(\gg)|_{\widetilde{\mathcal O^G_q}}$ for all $p,q\in\cU$ but in fact {\it equal}, identifying projected orbits in the obvious way.

The CR-structure on $\widetilde{\mathcal O^G_q}$ is then equal to that on $\widetilde{\mathcal O^G_p}$, since $\mathcal I$ depends smoothly on $v$, and $\mathcal I$ is projectable to a complex structure on 
$$\vardbtilde{\cD}=\vardbtilde{\pi}_*(\cD)=
\vardbtilde{\pi}_*(TU_\cD+\cK)=\vardbtilde{\pi}_*(TU_\cD)\,.
$$ In other words $\mathcal I$ is $\vardbtilde{\pi}$-projectable
to the leaf space of $\cK$. It is then easily checked that the original complex structure $\cJ$ on $\mathcal D$ is $\vardbtilde{\pi}$-projectable too, which is a contradiction by Proposition \ref{prop:inducedcomplexstructure} $(iv)$.
Conseqeuently, we conclude:

  \begin{proposition}
If $\opp{rk} (TU_\cD^\cJ)=2$, $\opp{rk}(TU_\cK)=2$, $\opp{rk} (TU_\cL)=0$, then $\dim\gg\leq 5$.
  \end{proposition}

\section{On 3-nondegenerate CR models with many symmetries}\label{sec:4}

Here we provide realizations of submaximal symmetric models 
as real hypersurfaces $\cM^7$ of $\bC^4$ and give an expression for their symmetry algebra --
we consider in turn the cases of infinitesimal symmetries and global automorphisms.
Then we discuss other cases with the symmetry algebra $\gg$ of $\dim\gg\ge4$, 
and finish with the proof of the generalized Beloshapka's conjecture in dimension 7.

\subsection{Homogeneous curves and tubes}\label{sec:5.1}

Starting with a real analytic curve $\mathbbm{r}\subset\R\mathbb{P}^3$ satisfying $\mathbbm{r}\wedge\mathbbm{r}'\wedge\mathbbm{r}''\neq 0$ (here and in the following, the prime symbol means $\frac{d}{dt}$), we consider 
the cone $R=C\mathbbm{r}\subset\R^4$ over it  and its tangent variety $TR\subset\R^4$.
This surface is singular along $R$ and, removing this singularity, 
we get the regular hypersurface $\Sigma=TR\setminus R\subset\R^4(x)$. Define the tube $\cM=\Sigma\times\R^4(y)\subset\R^4(x)\times\R^4(y)=\bC^4(z)$, with its induced CR structure $(\cD,\cJ)$.

We recall that $\mathbbm{r}=[\gamma]$ is called {\it nondegenerate} if, for any local affine parametrization of 
its lift $\R\ni t\mapsto\gamma(t)\in\R^4$, the vectors $\gamma(t),\gamma'(t),\gamma''(t),\gamma'''(t)$ 
are linearly independent for any $t\in\mathbb R$. In other words, they span $\R^4$ for any $t\in\R$. 

 \begin{lemma}
The tube $(\cM,\cD,\cJ)$ is 3-nondegenerate if and only if $\mathbbm{r}$ is nondegenerate.
 \end{lemma}

 \begin{proof}
The hypersurface $\Sigma=TR\setminus R$ can be parametrized via
 $
\psi:(r,s,t)\mapsto r\gamma(t)+s\gamma'(t)
 $, with $s\neq 0$.
Then $\cD=\langle\gamma,\gamma',\gamma'',\cJ \gamma,\cJ \gamma',\cJ \gamma''\rangle$,  where 
$\gamma=\psi_*\p_r$, $\gamma'=\psi_*\p_s$ while $\gamma''$ is a combination of those 
and $\psi_*\p_t$. 
Consequently $\cD$ is 1-step bracket generating if and only if $\mathbbm{r}$ is nondegenerate:
the equality $[\cD,\cD]=T\cM$ fails at all points $(x,y)\in\cM$ with $x\in T_{\lambda a}R$, $\lambda\in\R_\times$,
corresponding to $a\in\mathbbm{r}$, where the curve degenerates.

In this case, we let $Z_k=\frac12(\gamma^{(k)}-i\cJ \gamma^{(k)})\in T\cM\otimes\mathbb C$ 
be the $(10)$-components of the above generators of $\cD$, $0\leq k\leq2$. 
Then one easily computes  the Freeman sequence, cf. \cite[proof of Prop. 38]{KS1}:
 $$
\cD_{10}=\langle Z_0,Z_1,Z_2\rangle\supset\cK_{10}=\langle Z_0,Z_1\rangle
\supset\cL_{10}=\langle Z_0\rangle\supset \{0\}.
 $$
Indeed, $[Z_2,\overline{Z_1}]=\overline{Z_2}\mod\cD_{10}+\cK_{01}$ and 
$[Z_2,\overline{Z_0}]=\frac1s\cdot\overline{Z_1}\mod\cD_{10}+\cL_{01}$.
Consequently, $(\cM,\cD,\cJ)$ is uniformly 3-nondegenerate.
 \end{proof}

It is well known that the maximally symmetric nondegenerate curves $\mathbbm{r}\subset\R\mathbb{P}^3$ are 
rational normal curves,  and that their symmetry algebra is $\sl_2(\R)$ acting irreducibly on $\mathbb R^4\cong S^3\mathbb R^2$. 
Moreover, a maximally symmetric curve is unique up to projective transformation $PSL_4(\R)$ of $\R\mathbb{P}^3$.
Any other locally homogeneous curve has 1-dimensional projective symmetry algebra 
(because the projective curvature is non-trivial and can be normalized to canonical arc-length along the curve). 
The classification 
of homogeneous projective curves corresponds to Jordan normal forms of matrices (in our case of size $4\times4$),
and it has several presentations over $\bC$ and $\R$.  We use our own version of the latter below.

If $\mathbbm{r}$ is assumed (locally) homogeneous with an infintesimal symmetry $v\in\sl_4(\R)$,  then the corresponding
affine hypersurface $\Sigma\subset\R^4(x)$ has a 2-dimensional affine symmetry algebra, generated by $v$
and the radial symmetry corresponding to central element $\rho\in\mathfrak{gl}_4(\R)$.
It does not act transitively on the $3$-dimensional $\Sigma$ and, as we will shortly see, the same concerns the CR-symmetries
of the corresponding tube $\cM=\Sigma\times\R^4(y)\subset\bC^4(z)$. 
\vskip0.2cm\par
We depart with the following.
\begin{lemma}
An orbit $\mathbbm{r}$ of the flow $\exp(tv)$ is a nondegenerate curve
if and only if there is precisely one Jordan block for each eigenvalue of $v$.
 \end{lemma}

 \begin{proof}
It follows directly from Jordan normal forms: if there are two linearly independent eigenvectors for one
eigenvalue, then the flow $\exp(tv)$ has a linear integral, i.e., the orbit lies on a hyperplane, and hence the curve is degenerate.
 \end{proof}

We can bring any locally homogeneous nondegenerate curve $\mathbbm{r}=[\gamma]:I\to\R\mathbb{P}^3$ 
to one of the following parametrized forms, where:
(i) we use Segre notations for the real case without marking and for the complex case with the superscript$^c$;
(ii) the parameter $\tau\in I=\R_+$, while $t\in I=\R$ (and the passage is given by $\tau=e^t$), in order to meet some traditional forms.

 \begin{center}\begin{longtable}{lll}
(1111) & $\gamma_{\a\b}(\tau)=[1:\tau:\tau^\a:\tau^\b]$ & ($1<\a<\b$) \\                         
(211) & $\gamma_{\b}(\tau)=[1:\ln\tau:\tau:\tau^{\b}]$ & ($\b\neq0,1$) \\
(31) & $\gamma(\tau)=[1:\ln\tau:\ln^2\tau:\tau]$ \\ 
(22) & $\gamma(\tau)=[1:\ln\tau:\tau:\tau\ln\tau]$ \\ 
(4) & $\gamma(\tau)=[1:\tau:\tau^2:\tau^3]$ \\ 
($1^c$11) & $\gamma_{\a\b}(t)=[\cos\a t:\sin\a t:e^{t}:e^{\b t}]$ & ($\a\neq0$, $\b\neq1$) \\ 
($1^c2$) & $\gamma_{\a}(t)=[\cos\a t:\sin\a t:e^t:te^t]$ & ($\a\neq0$) \\ 
($2^c$) & $\gamma(t)=[\cos t:\sin t: t\cos t: t\sin t]$ \\ 
($1^c1^c$) & $\gamma_{\a\b}(t)=[\cos\a t:\sin\a t:e^t\cos\b t:e^t\sin\b t]$ & ($\a,\b\neq0$) \\ 
 & $\gamma_{\b}(t)=[\cos t:\sin t:\cos\b t:\sin\b t]$ & ($\b\neq0,\pm1$) 
 \end{longtable}\end{center}

We note that for the largest nilpotent 
type (4) (fitting into the irreducible $\sl_2(\R)\subset\sl_4(\R)$ upper-triangularly as the $4\times4$ Jordan block with the eigenvalue 0), the domain of the map can be taken to be $\R\mathbb{P}^1$, and the same is true for
the real simple type (1111) for integers $\a,\b$. However the curve $\gamma_{\a\b}(\tau)$ becomes
singular for $\tau=\infty$ and degenerate for $\tau=0$.

The rational normal curve of degree $3$ corresponds to type (4)
and also to the curve of Segre real simple type $(1111)$ with $\a=2,\b=3$.

 \begin{proposition}
\label{prop:6-dim}
If $\mathbbm{r}$ is not a rational normal curve, then the CR-symmetry algebra of the corresponding
tube $(\cM,\cD,\cJ)$ is 6-dimensional, generated by $v,\rho$ and translations along $\R^4(y)$.
 \end{proposition}

 \begin{proof}
The symmetry algebra is at least $6$-dimensional and we have a dichotomy: either the tube is locally maximally symmetric or it is not. In the first case, the translations along $\R^4(y)$ 
are the nilradical
of the $8$-dimensional symmetry algebra. In the second case, the tube  is not locally homogeneous by \cite[Thm. 1]{KS1}, the symmetry algebra is $6$-dimensional by Theorem \ref{T1}, generated by $v,\rho$ and the translations along $\R^4(y)$. The latter is the maximal Abelian subalgebra of the symmetry algebra. In both cases, the space $\mathbb R^4$ of translations along $\R^4(y)$ is a canonically defined subalgebra, 
and the other symmetries fits into the $\mathfrak{gl}_4(\R)$ acting on it. 

For instance, the scaling symmetries of the curve $\gamma_{\a,\b}$ of type $(1111)$ are 
$\tilde{v}=\opp{diag}(0,1,\a,\b)$ and $\rho=\opp{diag}(1,1,1,1)$ (as elements of $\mathfrak{gl}_4(\bR)$).
Their combination
 $$
v=4\tilde{v}-(1+\a+\b)\rho=\opp{diag}(-1-\a-\b,3-\a-\b,3\a-1-\b,3\b-1-\a)
 $$
belongs to $\sl_4(\bR)$, and the diagonal comprises the spectrum of $\opp{ad}(v)$ acting on 
$\R^4$.
This is an increasing sequence and, if $v$ belongs to the $\sl_2(\bR)$ sitting irreducibly in $\sl_4(\bR)$,
then it should have the form $(-3\nu,-\nu,\nu,3\nu)$ for some $\nu\neq0$, which implies $\a=2,\b=3$.
For any other values of the parameters, it follows that the curve of type $(1111)$ is not maximally symmetric 
and the symmetry algebra 
is generated by the symmetries already obtained: $v,\rho, \p_{y_k}$, $0\leq k\leq3$. 

The same holds for any other homogeneous nondegenerate curve $\mathbbm{r}$:
unless it is affinely equivalent to the rational normal curve, which can be understood by looking at the spectrum 
of $\opp{ad}(v)$ acting on $\R^4$, the associated tube is not maximally symmetric.
\end{proof}

This proposition has the second main theorem as a corollary:
 \begin{proof}[Proof of Theorem \ref{T2}]
Every tube $\cM$ that corresponds
to one of the nondegenerate projective curves $\mathbbm{r}$ from the above list, which is not a rational normal curve,
has symmetry algebra of submaximal dimension by Proposition \ref{prop:6-dim} and Theorem \ref{T1}. Furthermore, tubes corresponding to
different types from the list or to the same type but with different parameters, have CR-symmetry algebras that are non-isomorphic: in fact, the CR-symmetry algebra is the semi-direct product $\R^2\ltimes\R^4$,
with $\R^2$ generated by $v, \rho$ and $\R^4$ by $\p_{y_k}$, $0\leq k\leq 3$, and the representation of the traceless element $v$ on $\R^4$ encodes the projective symmetry of the curve.
Hence they are not CR-equivalent and we have a continuum of pairwise non-isomorphic submaximally
symmetric 3-nondegenerate CR-structures  in dimension 7.
 \end{proof}
 
 \begin{remark}\label{rem:50}
One can show that the CR-automorphism group 
of the tube  corresponding
to a nondegenerate projective curve $\mathbbm{r}$ is the semi-direct product 
$H\ltimes\R^4$, where $H\subset GL_{4}(\R)$ consists of automorphisms of the cone $R$, 
equivalently linear automorphisms that project to an automorphism of $\mathbbm{r}$
(indeed, automorphisms must preserve the maximal Abelian subgroup $\R^4$, hence descend to 
affine automorphisms of the tangent variety and its singular part $R=C\mathbbm{r}$, and to projective automorphisms of  
$\mathbbm{r}$). In most cases, $H$ is the direct product of the radial scalings $\R_\times$ and the exponents 
$\R_+$ of $v$, but sometimes the latter can be changed, e.g., to $S^1$.
 \end{remark}

\subsection{Submaximal automorphism group}\label{sec:5.2}

As noted in Remark \ref{rem:50}, there is a continuum of 3-nondegenerate 7-dimensional CR-hypersurfaces
with $\dim\opp{Aut}(\cM,\cD,\cJ)=6$. On the other hand, the flat model $\mathcal{R}^7$ has $8$-dimensional automorphism group. 
What about dimension 7?
In this case, there exists a Lie algebra of CR-symmetries of dimension 7 and the CR-manifold is locally flat by Theorem \ref{T1} and \cite[Thm. 3]{KS1}. In other words, it is locally CR-isomorphic to the flat model
$\mathcal{R}^7=\Sigma\times\R^4(y)$, $\Sigma=TR\setminus R$, where $R$ is the rational normal cone in $\R^4(x)$ (here and throughout this subsection).

In order to compute orbits more effectively, let us interpret the  irreducible
$SL_2(\R)$-module $\R^4\cong S^3\R^2$ as the space of cubic polynomials on $(\R^2)^*$.
The following are straghtforward:

 \begin{lemma}
In the above interpretation, $R$ corresponds to polynomials with triple real roots, while $\Sigma$ to
polynomials with one double and one simple real roots.
 \end{lemma}
 
 \begin{corollary}
The orbits of $SL_2(\R)$ on $\R\bP^3\cong\bP(S^3\R^2)$ are: $\mathbbm{r}=\bP R$,
$T\mathbbm{r}\setminus\mathbbm{r}$ as well as two open orbits $\cU_1$ and $\cU_2$ such that
$\cU_1\cup\cU_2=\mathbb R\bP^3\setminus T\mathbbm{r}$, corresponding to binary cubics 
with different roots: three reals for $\cU_1$, and a pair of complex conjugate and one real for $\cU_2$.
 \end{corollary}
 
  \begin{corollary}
The orbits of $GL_{2}(\R)$ on $\R^4_\times=\R^4\setminus\{0\}$ are $R\setminus\{0\}$,
$\Sigma$ and the cones $C\cU_1$ and $C\cU_2$.
 \end{corollary}
 
  \begin{proof}[Proof of Theorem \ref{T3}]
By the above, the Lie algebra of a submaximal automorphism group is a subalgebra of 
$\gg=\mathfrak{gl}_2(\R)\ltimes\R^4\cong \mathfrak{gl}_2(\R)\ltimes S^3\mathbb R^2$. There are only two $7$-dimensional subalgebras in $\gg$, namely
 $$
\mathfrak{s}=\sl_2(\R)\ltimes\R^4\ \text{ and }\ \mathfrak{f}=\mathfrak{b}\ltimes\R^4,
 $$
where $\mathfrak{b}$ is the $3$-dimensional Borel subalgebra of upper-triangular matrices in $\mathfrak{gl}_2(\R)$, up to conjugation.

The first Lie algebra integrates to $SL_2(\R)\ltimes\R^4$, which acts simply transitively on the model 
$\mathcal{R}^7=\Sigma\ltimes\R^4(y)$.
Indeed, $SL_2(\R)$ acts simply transitively on $\Sigma$. To see this, first note that $PSL_2(\R)$
acts transitively on triples of points in $\R\bP^1$, hence also on pairs of marked points. 
Moreover any binary cubic can be mapped by a transformation from $SL_2(\R)$ to $x^2y$, 
and the only transformation from this group leaving $x^2y$ invariant is the identity.

Note that the smallest Lie group with Lie algebra $\sl_2(\R)=\mathfrak{s}/\R^4$ is $PSL_2(\R)$,
the quotient of $SL_2(\R)$ by its center $\bZ_2=\{\pm\Id\}$, but this group does not
extend to the semi-direct product with $\R^4$, so it does not give rise to a $7$-dimensional automorphism group.
The remaining Lie groups with Lie algebra $\sl_2(\R)$, which admit an effective irreducible action  
on $\R^4$, factor through the universal cover of $SL_2(\R)$. Consequently, Lie groups $S$ with 
$\operatorname{Lie}(S)=\mathfrak{s}$ factor through the universal cover of $SL_2(\R)\ltimes\R^4$ 
and therefore give rise to the CR-models derived in \cite{KS1}. 
The automorphism group of each of these has dimension 8, so $S$ is only properly contained in there
and therefore is not an automorphism group.

The second Lie algebra integrates to the Lie group $F=B\ltimes\R^4$ and this action of $F$ descends 
to affine automorphisms $B$ of $\Sigma$ (by the argument we already exploited in \cite{KS1}),
hence to those of $R$ and then to projective automorphisms of the rational normal curve 
$\mathbbm{r}\subset\R\mathbb{P}^3$.  This action of $B$ has two orbits: $\mathbbm{r}_\times
=\R\mathbb{P}^1\setminus\{\infty\}\cong\R$
and the point $\infty\in\R\mathbb{P}^1$.

Restricting to the open orbit $\mathbbm{r}_\times$ (otherwise we again get the CR-model),
we consider the cone $R'=C\mathbbm{r}_\times$ and the regular hypersurface $\Sigma'=TR'\setminus R'\subset\mathbb R^4(x)$.
Its affine automorphism group is $B$, the stabilizer of $\Sigma'$ in $\Sigma$. We note that $\Sigma'$ has two connected components:
the corresponding set of roots consists of a double and a single root, and the components differ
by their order in  $\bP R'\cong\R$ (e.g., the single root comes before the double root). 
The same is true for the group $B$: it is disconnected with two components.

Note that $B$ acts simply transitively on $\Sigma'$. To see this, first note that it
acts transitively on pairs of marked points in $\R$, which can be chosen as $\{1,0\}$. 
Then any binary cubic can be mapped by a transformation from $B$ to $x^2(x-y)$, 
and the only transformation from $B$ leaving it invariant is the identity.
Hence we arrive at the tube $\cM'=\Sigma'\times\R^4(y)$, 
which is an open CR-submanifold of $\mathcal{R}^7$ with automorphism group $F$. 
We have an independent \textsc{Maple} verification of this fact in a supplement accompanying the arXiv posting of the article.

This is the only simply transitive 3-nondegenerate 7-dimensional CR-hypersurface
with the automorphism group of dimension 7. Indeed, we have:
 $$
\pi_0(F)=\pi_0(B)=\bZ_2,\quad \pi_1(F)=\pi_1(B)=\{\Id\},
 $$
thus $F$ can be considered as the universal covering Lie group with Lie algebra $\mathfrak{f}$ 
(see \cite[\S5.1]{KS1} for a related discussion of universal covers of disconnected Lie groups).
Any homogeneous model with $7$-dimensional automorphism group is quotient of $F$ by a discrete normal subgroup, which is trivial. 
Consequently there exists no other Lie group with Lie algebra $\mathfrak{f}$ that can act as automorphism group 
of a locally flat 3-nondegenerate CR-hypersurface. 
 \end{proof}

\subsection{Intransitive tubes with fewer symmetries}\label{sec:5.3}

In the proof of Theorem \ref{T1}, we bounded the dimensions of intransitive symmetry algebras 
of 3-nondegenerate 7-dimensional CR-hypersurfaces. However our bounds were not claimed sharp.
It is yet an open question, what is the maximal value of $\dim\gg$ for the intransitive symmetry algebras 
$\gg$ of $(\cM,\cD,\cJ)$ depending on the dimension of the orbit $d\in[1,6]$.

While we expect this maximal value to be $d$, a simpler claim is that every $d<7$ is realizable as the symmetry 
dimension of a 3-nondegenerate 7-dimensional CR-hypersurface $(\cM,\cD,\cJ)$. 
Let us demonstrate this for larger dimensions $4\leq d\leq 6$.

For $d=6$ we already showed realizations starting with a curve $\mathbbm{r}\subset\R\mathbb{P}^3$ 
with a non-trivial projective infinitesimal symmetry. If we choose a generic curve
$\mathbbm{r}$, then its projective symmetry is trivial, yet the corresponding cone  $R=C\mathbbm{r}$
has the radial symmetry, which is inherited by the submanifold 
$\Sigma'=TR\setminus R\subset\R^4(x)$.
Thus the tube 
 $
\cM'=\Sigma'\times\R^4(y)\subset\bC^4(z)
 $
is 3-nondegenerate and has 
symmetry dimension $d=5$: the symmetries are the radial field $\rho$ and the translations $\p_{y^k}$, $0\leq k\leq 3$.
An independent verification was done in \textsc{Maple}.

To realize the symmetry dimension $d=4$, let us choose the ruled submanifold $\Sigma''$ in $\R^4(x)$ parametrized as 
 $
 \psi(t,r,s)=\g(t)+r\g'(t)+s\g''(t)
 $
for a nondegenerate curve $\g$ in $\R^4(x)$ and define the tube $\cM''=\Sigma''\times\R^4(y)$ as usual. Then 
 $$
T\Sigma''=\langle\g',\g'',\g'''\rangle\subset T\R^4(x),\quad 
\cJ T\Sigma''=\langle\g',\g'',\g'''\rangle\subset T\R^4(y)=\langle\g,\g',\g'',\g'''\rangle,
 $$
and $\cD=T\Sigma''\otimes\bC=T\Sigma''\oplus\cJ T\Sigma''\subset T\cM''$. 
Since $T\Sigma''$ is generated by 
 $$
\psi_*\p_t=\g'+r\g''+s\g''',\quad \psi_*\p_r=\g'\quad\text{and}\quad \psi_*\p_s=\g'',
 $$ 
one checks  that $\cM''$ is 3-nondegenerate (as in \cite[\S6.4]{KS1}). 
For generic $\g$ the only infinitesimal CR-symmetries of $\cM''$ are the translations along $\R^4(y)$.
(This is again independently verified in a \textsc{Maple} supplement.)

We expect that the smaller dimensions $0\leq d\leq3$ are also realizable as the symmetry 
dimension of a 3-nondegenerate 7-dimensional CR-hypersurface $(\cM,\cD,\cJ)$, however 
the corresponding examples clearly cannot be realized as tubes.  

\subsection{On Beloshapka's conjecture}\label{sec:5.4}

Let us finally prove a version of Beloshapka's conjecture
for $n=3$ in a stronger form, depending on the degree of degeneracy of $(\cM,\cD,\cJ)$ and
with sharper upper bounds for symmetry dimensions.

 \begin{proof}[Proof of Theorem \ref{T4}]
Here is a summary of the known and novel results about real-analytic infinitesimal  CR-symmetries in dimension 7:
 \begin{itemize}
\item[(i)] if there exist points of Levi-nondegeneracy in $\cM$ then $\dim\gg\leq24$ by \cite{CM,Ta2};
\item[(ii)] if there are both points of Levi-nondegeneracy and points of Levi-degeneracy in $\cM$
then $\dim\gg\leq17$ by \cite{Kr};
\item[(iii)] if the structure is uniformly Levi-nondegenerate but there are non-spherical points then $\dim\gg\leq13$
for Levi-indefinite case and $\dim\gg\leq12$ for Levi-definite case \cite{Kru2016};
\item[(iv)] if the structure is uniformly Levi-degenerate but there are points of 2-nondegeneracy then $\dim\gg\leq16$,
under an additional regularity assumption on the so-called abstract reduced modified symbols \cite{PZ,SZ}, while in the general case the bounds increases to
$\dim\gg\leq17$ by \cite{B2};
\item[(v)] if the structure is uniformly 3-nondegenerate then $\dim\gg\leq8$ according to \cite{KS1}.
 \end{itemize}
These bounds, together with the restriction argument we exploited, imply the claim.
 \end{proof}
 
 \begin{remark}
All the symmetry bounds in the above proof are sharp except for dimension 17 in case (iv). In fact, we suggest this bound 
is not sharp and that the sharp bound of (iv) is 16, but this will be discussed elsewhere. 
If this is the case, then the only realizable symmetry dimension 17 would come from (ii), as was conjectured in \cite{Kr}.   
One of the strategies for the proof is to use, for 2-nondegenerate structures, the methods 
bounding the dimension of intransitive symmetry algebras that we exploited in this paper in the 3-nondegenerate case.
 \end{remark}

\bigskip\par

\end{document}